\documentclass{amsart} 
\usepackage{amscd,amssymb,amsmath,amsbsy,amsthm}
\usepackage{color,subfig}
\usepackage[colorlinks,plainpages,backref,urlcolor=blue]{hyperref}
\usepackage[width=5.65in,height=8.0in,centering]{geometry}
\usepackage[mathscr]{euscript}

\usepackage{tikz}
\usetikzlibrary{calc,intersections,arrows,patterns,cd}
\usepackage{booktabs}
\usepackage{enumerate}

\makeatletter
\newcommand{\mylabel}[2]{(#2)\def\@currentlabel{#2}\label{#1}}
\makeatother

\theoremstyle{definition}
\newtheorem{theorem}{Theorem}[subsection]
\newtheorem{definition}[theorem]{Definition}
\newtheorem{lemma}[theorem]{Lemma}
\newtheorem{corollary}[theorem]{Corollary}

\newtheorem{proposition}[theorem]{Proposition}

\newtheorem{exm}[theorem]{Example}
\newtheorem{rem}[theorem]{Remark}
\newtheorem*{ack}{Acknowledgement}
\newtheorem{question}[theorem]{Question}

\newenvironment{example}%
{\pushQED{\qed}\begin{exm}}%
{\popQED\end{exm}}  

\newenvironment{remark}%
{\pushQED{\qed}\begin{rem}}%
{\popQED\end{rem}}  

\newcommand{\abs}[1]{\left|#1\right|}     
\newcommand{\set}[1]{\left\{#1\right\}}   
\newcommand{\angl}[1]{\left<#1\right>}    

\newcommand{\Z}{\mathbb{Z}}
\renewcommand{\P}{\mathbb{P}}
\newcommand{\Q}{\mathbb{Q}}
\newcommand{\C}{\mathbb{C}}

\renewcommand{\k}{\Q}  
\newcommand{\M}{\mathsf{M}}    
\newcommand{\LM}{\mathsf{L}}   
\newcommand{\G}{\mathcal{G}}   
\renewcommand{\H}{\mathcal{H}}   
\newcommand{\NS}{S}   
\renewcommand{\L}{\mathcal{L}}   

\newcommand{\CS}{\mathscr{C}}   
\newcommand{\QS}{\mathscr{Q}}   
\newcommand{\ZS}{\mathscr{Z}}   
\newcommand{\FF}{\mathscr{F}}   
\newcommand{\KS}{\mathscr{K}}   
\newcommand{\GS}{\mathscr{G}}   
\newcommand{\OSS}{\mathscr {OS}}  

\newcommand{\bottom}{\hat{0}}  
\newcommand{\onehat}{\hat{1}}

\def\clap#1{\hbox to 0pt{\hss#1}}
\def\mathclap{\mathpalette\mathclapinternal}
\def\mathclapinternal#1#2{%
\clap{$\mathsurround=0pt#1{#2}$}}
\newcommand*{\dotleq}{\mathbin{\kern0.4em\cdot%
\mathclap{\raisebox{-0.05em}{$\leqslant$}}}}

\DeclareMathOperator{\dd}{d}        
\DeclareMathOperator{\op}{op}      
\DeclareMathOperator{\In}{In}      
\DeclareMathOperator{\lk}{lk}      
\DeclareMathOperator{\st}{st}      
\DeclareMathOperator{\clst}{\overline{st}}      
\DeclareMathOperator{\atoms}{a}    
\DeclareMathOperator{\Ka}{{\mathbf{at}}}        
\DeclareMathOperator{\N}{{\mathbf{n}}}         
\DeclareMathOperator{\nbc}{{\mathbf{nbc}}}  
\DeclareMathOperator{\Bl}{Bl}       
\DeclareMathOperator{\Fl}{Fl}       
\DeclareMathOperator{\FS}{\mathscr{F}\kern -.5pt \ell}  
\DeclareMathOperator{\fl}{\Phi}       
\DeclareMathOperator{\supp}{supp}

\DeclareMathOperator{\Sh}{Sheaves}  

\DeclareMathOperator{\pr}{pr}      
\DeclareMathOperator{\rank}{rank}     

\DeclareMathOperator{\coker}{coker}
\DeclareMathOperator{\ann}{ann}
\DeclareMathOperator{\OS}{OS}  
\DeclareMathOperator{\DP}{\mathscr{DP}}  
\DeclareMathOperator{\Dp}{DP}  
\DeclareMathOperator{\POS}{\overline{OS}}  
\DeclareMathOperator{\ShExt}{\mathscr{E}\kern -0.5pt xt}
\DeclareMathOperator{\Hom}{Hom}
\DeclareMathOperator{\ShHom}{\mathscr{H}\kern -0.5pt om}
\DeclareMathOperator{\im}{im} 
\DeclareMathOperator{\irr}{irr} 
\DeclareMathOperator{\divs}{div} 

\newcommand{\hB}{\hat{B}} 
\newcommand{\A}{{\mathcal{A}}}

\newcommand{\cdga}{\textsc{cdga}}

\DeclareMathOperator{\Int}{{I}}        


\begin{document}

\begin{abstract}
We construct a combinatorial generalization of the Leray models for 
hyperplane arrangement complements.  Given a matroid and some combinatorial
blowup data, we give a presentation for a bigraded 
(commutative) differential graded algebra.  If the matroid is realizable
over $\mathbb{C}$, this is the familiar Morgan model for a hyperplane arrangement 
complement, embedded in a blowup of projective space.  In general, we
obtain a \textsc{cdga} that interpolates between the Chow ring of a matroid
and the Orlik--Solomon algebra.  Our construction can also be expressed
in terms of sheaves on combinatorial blowups of geometric lattices.
As a key technical device, we construct a monomial basis via
a Gr\"obner basis for the ideal of relations.  Combining these ingredients,
we show that our algebra is quasi-isomorphic to the classical Orlik--Solomon
algebra of the matroid.
\end{abstract}

\title[A model for the Orlik--Solomon algebra]%
{A Leray model for the Orlik--Solomon algebra}

\author[C. Bibby]{Christin Bibby}
\address{Department of Mathematics, Louisiana State University,
Baton Rouge, LA 70803, USA}
\email{\href{mailto:bibby@math.lsu.edu}{bibby@math.lsu.edu}}
\author[G. Denham]{Graham Denham$^1$} 
\address{Department of Mathematics, University of Western Ontario,
London, ON  N6A 5B7, Canada}  
\email{\href{mailto:gdenham@uwo.ca}{gdenham@uwo.ca}}
\urladdr{\href{http://gdenham.math.uwo.ca/}%
{http://www.math.uwo.ca/\~{}gdenham}}
\thanks{$^1$Partially supported by an NSERC Discovery Grant (Canada)}

\author[E.\ M.\ Feichtner]{Eva Maria Feichtner}
\address{
ALTA, Department of Mathematics and Computer Science, University of Bremen, 
28359 Bremen, Germany}
\email{\href{mailto:emf@math.uni-bremen.de}{emf@math.uni-bremen.de}}

\setcounter{tocdepth}{2}
\maketitle
\tableofcontents

\section{Introduction}

For a given arrangement of complex linear hyperplanes, a number of 
complex algebraic varieties have been defined and extensively studied
over the years. Notably, there is the {\em projective complement\/} of $\A$, 
$U(\A):=\P^{\ell}\,{\setminus}\,\bigcup_{H\in\A}H$, and the {\em wonderful 
compactification\/} $Y(\A,\G)$ obtained by blowing up $\P^{\ell}$ along
proper transforms of suitably chosen subspaces~$\G$. The cohomology
algebras of both 
have been studied and have been described in a purely combinatorial
manner in terms of the matroid associated with the arrangement. These are  
the \emph{projective Orlik--Solomon algebra}, $\POS(\A)\,{\cong}\,
H^*(U(\A))$~\cite{Br73,OS80,Ka04},
and the \emph{De Concini--Procesi algebra}, 
$\Dp(\A,\G)\,{\cong}\, H^*(Y(\A,\G))$~\cite{DP95}, respectively. 

Though they seem to be rather disparate algebraic invariants of arrangements, 
we place them here in one and the same scene. Moreover, in the spirit of 
recent
work of Adiprasito, Huh, and Katz \cite{AHK15} as well as
\cite{ADH20,BHMPW20}, we lift objects of geometric origin to the purely
combinatorial context of matroid theory, where phenomena that are based on the
geometry of spaces miraculously persist. Combinatorial blowups, Gr\"obner bases,
and sheaves on posets thereby replace heavy geometric machinery.

In the geometric setting,
the wonderful compactification $Y(\A,\G)$ is constructed through a sequence of
blowups dictated by a special set $\G$, called a \textit{building set}.
The projective complement $U(\A)$ is a dense open subset of $Y(\A,\G)$, in which
it is realized as the complement of a normal crossings divisor.
The cohomology algebras
of $Y(\A,\G)$ and $U(\A)$ are linked through the Leray spectral sequence
\[
E_2^{pq}=H^p(Y(\A,\G), R^q j_*\Q) \Rightarrow H^{p+q}(U(\A), \Q)
\]
given by the inclusion $j\colon U(\A)\hookrightarrow Y(\A,\G)$: 
the spectral sequence degenerates at $E_3$, so the cohomology of $(E_2,d_2)$
agrees with that of $U(\A)$.   Computing with differential forms, this is the
Morgan \cdga\ model for $U(\A)$ and coincides with the model constructed by De
Concini and Procesi \cite[\S5.3]{DP95}.  Arrangement complements are
well-known to be rationally formal \cite{Br73}, and the mixed Hodge 
structure on their cohomology is pure \cite{Sh93}.  Work of Dupont~\cite{du15}
explains how these observations together are equivalent to the fact that
the edge map $H^\cdot(U(\A),\Q)\to(E_2,d_2)$ is a quasi-isomorphism.

In the more general
combinatorial setting, we introduce the notion of a Leray or Morgan
model of a matroid.  We start with a matroid $\M$, its lattice of
flats $\LM$, and a \textit{combinatorial building set} $\G\subset \LM$.
We study the effect of the blowup of one element of $\G$ at a
time, and we build objects indexed by \textit{partial building sets}
$\H\subseteq\G$. We obtain semilattices $\L(\LM,\H)$ that interpolate between
the geometric lattice $\LM$ and a simplicial poset $\L(\LM,\G)$.
Geometrically, this interpolates between the combinatorics of the original
arrangement and that of the normal crossings divisor obtained through the
sequence of blowups.
We combine elements of the Orlik--Solomon algebras and the De Concini--Procesi
algebras to  define differential graded algebras $B(\LM,\H)$,
which play the role of $(E_2,d_2)$ above.
Indeed, our main result (Theorem \ref{thm:model}) states that each $B(\LM,\H)$,
and in particular $B(\LM,\G)$, is quasi-isomorphic to $\POS(\LM)$.

The bookkeeping that goes along with the sequence of combinatorial blowups is
notationally intensive.  Some of the arguments in the paper are, we believe, 
unavoidably quite technical. In order to help the reader navigate the paper 
without getting lost in the details we provide a roadmap.

\subsection*{Outline of the paper}

In \S\ref{sec:basics} we provide some combinatorial basics tailored to our use.
We summarize the background on sheaves on posets and the poset of intervals.
We elaborate on combinatorial blowups and the notions of building sets 
and nested sets, which form 
the combinatorial core of De Concini--Procesi arrangement models.   

In \S\ref{sec:OS}, we construct an Orlik--Solomon algebra $\OS(\L)$ from 
what we call a locally geometric semilattice $\L$.
Such semilattices include the ones that appear by iteratively blowing up
a geometric lattice. In the realizable case, these algebras model the left 
edge of the Leray spectral sequence: that is, they are the 
global sections of the cohomology sheaf obtained by restricting the constant 
sheaf from a partial blowup to the hyperplane arrangement complement.  
The algebra $\OS(\L)$ has a well-known monomial basis called the $\nbc$ basis.

In \S\ref{sec:DP}, we construct a De Concini--Procesi algebra $\Dp(\LM,\H)$ from
a geometric lattice $\LM$ and partial building set $\H$.
In the realizable case, it is isomorphic to the cohomology of the wonderful 
De~Concini--Procesi model of an arrangement complement obtained by blowing 
up $\P^l$ along the subspaces $\H$. The algebra is also isomorphic to 
the bottom edge of the Leray spectral sequence.
Regardless of realizability, it is also the Chow ring of a smooth toric 
variety associated with a subfan of the Bergman fan~\cite{FY04}. 

Our main object of study is introduced in \S\ref{sec:B}, the 
commutative differential graded algebra $B(\LM,\H)$ associated with a
geometric lattice $\LM$ and a partial building set $\H$.  We define it
by means of a presentation that combines the relations from the Orlik--Solomon
and De Concini--Procesi algebras.  
Using Gr\"obner basis theory, we show that $B(\LM,\H)$ has a monomial basis
that specializes in one direction to the $\nbc$ basis for the Orlik--Solomon
algebra, and in the other to the basis for the De Concini--Procesi algebra of
\cite{FY04}.
This basis is essential for obtaining injective maps between the algebras
$B(\LM,\H)$ (Theorem \ref{thm:injective}).

We show (Proposition~\ref{prop:decomp}) that $B(\LM,\H)$ 
has a bigraded direct-sum 
decomposition, indexed over the semilattice $\L(\LM,\H)$, where the 
summands are tensor products of ``local'' Orlik--Solomon algebras with
``local'' De Concini--Procesi algebras.  In the geometric setting, this
is a familiar picture: the compactification is stratified by intersections of
hypersurfaces.  Each stratum contributes to the Leray spectral sequence
a tensor product of the
cohomology of the hypersurface near the stratum (Orlik--Solomon), 
and the cohomology of the stratum itself (De Concini--Procesi): see
\cite{Loo93,du15,Bi16}.  
In general, though, there is no Leray spectral sequence, and no geometric
reason why there should be such a direct sum decomposition.  Instead, we 
make use of our Gr\"obner basis to show that expected decomposition 
exists in all cases.  

Similarly, we show that the ``local'' De Concini--Procesi algebras
$\Dp_y(\LM,\H)$, for elements $y\in \L(\LM,\H)$, can be decomposed as 
tensor products (Theorem~\ref{thm:D(X)}).  This reflects the fact that
the strata in De Concini and Procesi's compactification are themselves
products of De Concini--Procesi compactifications of arrangements of
lower dimension \cite[\S4.3]{DP95}.

In an effort to arrive at a combinatorial explanation of phenomena like this, 
we use the classical notion of sheaves on posets, inspired by work
of Yuzvinsky~\cite{Yuz95}.  We topologize a finite poset $P$ with the
order topology, in which basic open sets are principal order ideals.
We consider sheaves of Orlik--Solomon algebras, which model the cohomology
sheaf $j_*\Q$ in the realizable case.  For technical reasons, it turns out to 
be more convenient to work with the homology version of the Orlik--Solomon
algebra, the \emph{flag complex} introduced by
Schechtman and Varchenko~\cite{SV91}.
Flag complexes
assemble into a graded sheaf $\FS(\L)$ on the poset $\L:=\L(\LM,\H)$.
A standard differential makes $\FS(\L)$ a cochain complex which is,
in fact, a flasque resolution of a skyscraper sheaf.
Similarly, we define a \emph{De Concini--Procesi sheaf} $\DP$ of
algebras on $\L^{\op}$,
by letting $\DP(\L_{\geq y})=\Dp_y(\LM,\H)$.  In the realizable case, this
is the sheaf of cohomology algebras of strata.

For any poset $P$, the \emph{poset of intervals} $\Int(P)$ is the poset on
pairs $(x,y)\in P\times P^{\op}$ with the order relation given by 
containment.  The poset $\Int(\L)$ turns out to be a key organizational 
device for understanding our combinatorial Leray model.  We define a
graded sheaf
$\CS(\LM,\H)$ on $\Int(\L)$ as the tensor product of the pullbacks of
$\FS$ and $\DP$ along the projections to $\L$ and $\L^{\op}$, respectively.
This is, in fact, a cochain complex of coherent sheaves of $\DP$-modules.
We show (Theorem~\ref{thm:global sections}) that the vector space of
global sections of
$\CS(\LM,\H)$ is the $\Q$-dual of our \cdga\ $B(\LM,\H)$.  Moreover,
this complex turns out to be a $\Gamma$-acyclic
resolution of a sheaf on $\Int(\L)$ obtained from $\DP$ 
(Theorem~\ref{thm:resolution}).

Section \S\ref{ss:injective} is devoted to showing that our algebras 
$B(\LM,\H)$ are quasi-isomorphic (Theorem \ref{thm:model}) to $\POS(\LM)$,
regarded as a \cdga\ with zero differential.  In the sense of rational
homotopy theory, then, these are rational models.  This quasi-isomorphism is
established by induction, by showing that each blowup gives an injective map of 
\cdga s (Theorem~\ref{thm:injective}).  Arguing with sheaves on 
the poset of intervals shows that the map at each step is a quasi-isomorphism
(Theorem~\ref{thm:quasiiso}).


\section{Combinatorial foundations}
\label{sec:basics}
In this section, we review the basic combinatorial background, as well as build
new tools to be used in our setting. Most of this section is technical.
Before we turn our attention to the main objects of study (a sequence of blowups
of semilattices) in \S\ref{ss:blowups}, we quickly discuss some
generalities on posets (sheaves on posets and the poset of intervals) which will
not be used until \S\ref{ss:sheaves}. For more general background
references, we refer to the book of Oxley \cite{Oxbook} on matroid theory and
the book of Orlik and Terao \cite{OTbook} on hyperplane arrangements.

\subsection{Sheaves on posets}
\label{ss:basic sheaves}
The notion of sheaves on posets and their cohomology turns out to be
convenient for analyzing the algebras that we study in this paper.
Our approach is inspired by that of Yuzvinsky~\cite{Yuz95}, and we
generalize his work on the (classical) Orlik--Solomon algebra here.

Let $P$ be a finite, partially ordered set.
For $x\in P$, we will let $P_{\geq x}=\set{y\in P\colon y\geq x}$, and define
the obvious variations analogously.  We will denote
closed intervals by $[x,y]$ for $x,y\in P$.  If $P$ is a ranked poset,
we let $P_q$ denote the subset of elements of rank $q$, for $q\in \Z$.
We give $P$ the topology in which downward-closed sets (order ideals) are 
open.  (We note that this is opposite to Yuzvinsky's convention.)  
A basis for the topology is given by the principal
order ideals, $\set{P_{\leq x}\colon x\in P}$.  We will also let $\abs{P}$
denote the \emph{order complex} on $P$, the abstract simplicial complex
whose simplices are the totally ordered subsets of $P$.

Then a sheaf $\FF$ on $P$ in an abelian category $\CS$ is 
just a diagram of objects of $\CS$ over $P$: more precisely, regarding
$P$ as a category with morphisms $y\rightarrow x$ whenever $y\geq x$, 
a functor $\FF\colon P\to \CS$ is equivalent to a (topological) sheaf on 
$P$ for which the sections over $P_{\leq x}$ equal $\FF(x)$.  
The restriction maps on principal open sets 
are the morphisms $\FF(x\leq y)$.  Since each $x\in P$ is contained in
a unique minimal open set ($P_{\leq x}$), the stalk of $\FF$ at $x$
also equals $\FF(x)$.  

We will take \cite{deH62,Ba75} as fundamental references,
but prove here some basic facts about sheaves and maps between posets.
First, suppose $f\colon P\to Q$ is a map of posets with the property that,
for all $x,y\in Q$ with $x\leq y$, there exist $u,v\in P$ with $u\leq v$
satisfying
$f(u)=x$ and $f(v)=y$.  Then $f$ is called a \emph{quotient}.

If $f\colon P\to Q$ is a map of posets for which the order complex
$\abs{f^{-1}(y)}$ is nonempty and connected, for each $y\in Q$, we will
say \emph{$f$ has connected fibres} for short.

\begin{example}\label{ex:closure}
  A quotient $f\colon P\to Q$ is a surjective map of underlying sets, but
  not conversely.  However, if $f$ admits a splitting, it is easily seen
  to be a quotient.  An
  example of a split surjection is a poset $P$ with a closure operation
  $x\mapsto \overline{x}$.  In this case, if $x=\overline{x}=\overline{y}$
  for some $x,y\in P$, then $y\leq \overline{y}=x$, so the fibres of the
  quotient each have a cone vertex; in particular, they are connected and
  contractible.
\end{example}

\begin{lemma}\label{lem:pointy fibres}
Let $f\colon P\to Q$ be a map of posets for which the preimage of
each principal order ideal is again a principal order ideal.  Then 
$f_*\colon \Sh(P)\to\Sh(Q)$ is an exact functor.
\end{lemma}
\begin{proof}
The direct image is left-exact, so this amounts to checking $f_*$ is also
right-exact.  Let
us suppose $\phi\colon \FF\to\GS$ is a surjective map of sheaves on $P$.
For any $y\in Q$, consider the stalk at $y$ of the map $f_*\phi\colon
f_*\FF\to f_*\GS$.
By hypothesis, $f^{-1}(Q_{\leq y})$ has
a unique maximum element $x\in P$, so 
$(f_*\FF)(y)\to (f_*\GS)(y)$ equals $\phi(x)\colon \FF(x)\to \GS(x)$.
This shows $f_*\phi$ is surjective on stalks, 
so $f_*\phi$ is surjective.
\end{proof}

\begin{lemma}\label{lem:adjiso}
  Suppose $f\colon P\to Q$ is a quotient map with connected fibres.
  Then, for any sheaf $\FF$ on $Q$, the adjunction
  $\eta\colon \FF\to f_*f^*\FF$ is an isomorphism.
\end{lemma}
\begin{proof}
  For any $y\in Q$ and sheaf $\FF$ on $Q$, we compute
  \begin{align}
    f_*f^*\FF(Q_{\leq y})&=(f^*\FF)(f^{-1}(Q_{\leq y}))
\nonumber\\
    &=\varprojlim_{\substack{x\in P\colon \\ f(x)\leq y}} (f^*\FF)(P_{\leq x})
\nonumber\\
    &=\varprojlim_{\substack{x\in P\colon \\ f(x)\leq y}} \FF(Q_{\leq f(x)}).
\label{eq:adjiso}
  \end{align}
  The adjunction map $\eta_y\colon \FF(Q_{\leq y})\to f_*f^*\FF(Q_{\leq y})$ is 
  induced by the restriction maps from $\FF(y)=\FF(Q_{\leq y})$ to terms in the
  limit \eqref{eq:adjiso}.  For any $x\in P$ with $f(x)\leq y$, because
  $f$ is a quotient, there exist elements $x',x''\in P$ for which
\[
  x'\leq x'',\qquad f(x')=f(x), \quad\text{and}\quad
  f(x'')=y.
  \]
  Since the fibres of $f$ are connected, the diagram \eqref{eq:adjiso} contains
  maps which compose to an isomorphism $\FF(f(x'))\to \FF(f(x))$.
  Pre-composing with
  the restriction $\FF(y)=\FF(f(x''))\to \FF(f(x'))$, we obtain
  a map $\FF(y)\to \FF(f(x))$. Then $\FF(y)$ is
  initial in the diagram, so $\eta_y$ is an isomorphism to the diagram's limit.
\end{proof}

\begin{proposition}\label{prop:Quillen}
  Let $f\colon P\to Q$ be a poset quotient.  Suppose that,
  for each $y\in Q$, the order complex $\abs{f^{-1}(Q_{\geq y})}$ is
  contractible, and the fibre $f^{-1}(y)$ is connected.
  Then, for any sheaf $\FF$ on $Q$, the cohomology pullback
  \[
  f^*\colon H^\cdot(Q,\FF)\to H^\cdot(P,f^*\FF)
  \]
  is an isomorphism.
\end{proposition}
\begin{proof}
  The cohomology pullback is induced by applying (derived) global sections to
  the composition
  $\FF\stackrel{\eta}{\longrightarrow} f_*f^*\FF\hookrightarrow Rf_*f^*\FF$,
  identifying $R\Gamma(Q,Rf_*f^*\FF)\cong R\Gamma(P,f^*\FF)$ by means of the
  Leray spectral sequence.

Let 
\[
\begin{tikzcd}
0\ar[r] & \FF\ar[r] & \CS^0\ar[r] & \CS^1\ar[r] & \cdots
\end{tikzcd}
\]
be the Godement resolution of $\FF$.  Explicitly, 
for $y\in Q$, let
\[
\FF_y(x) = 
\begin{cases}
\FF(y) & \text{ for } x\geq y\\
0 & \text{ otherwise}
\end{cases}
\]
so that each $\CS^p$ is a product of copies of the skyscraper sheaves
$\FF_y$, indexed by chains in the poset of length $p+1$.

We claim that
\begin{equation}\label{eq:Quillen}
\begin{tikzcd}
0\ar[r] & f^*\FF\ar[r] & f^*\CS^0\ar[r] & f^*\CS^1\ar[r] & \cdots
\end{tikzcd}
\end{equation}
is a $\Gamma$-acyclic resolution of $f^*\FF$.  Exactness is a property of
inverse image; for acyclicity, it is enough to check that each factor
$H^p(P,f^*(\FF_y))$ is zero for $p>0$, for $y\in Q$.  

For this, we use the Godement resolution again.  Note $f^*(\FF_y)(x)=
\FF(y)$ if $f(x)\geq y$, and zero otherwise.
Global sections of the Godement resolution may be identified with
the simplicial cochain complex on the order complex
of $f^{-1}(Q_{\geq y})$, with constant coefficients, so
\[
H^p(P,f^*(\FF_y))\cong H^p(\abs{f^{-1}(Q_{\geq y})},\Z)\otimes_{\Z} \FF(y),
\]
which is zero for $p>0$ by hypothesis.

We complete the proof as follows.  Applying $H^\cdot(\eta)$ to $\CS^\cdot$,
\begin{align*}
  H^\cdot(Q,\FF) \cong H^\cdot\Gamma \CS^\cdot
  &\stackrel{\cong}{\longrightarrow} H^\cdot\Gamma f_*f^*\CS^\cdot\quad
  \text{by Lemma~\ref{lem:adjiso};}\\
  &= H^\cdot\Gamma f^*\CS^\cdot\cong H^\cdot(P,f^*\FF),\quad
  \text{using the resolution \eqref{eq:Quillen}.}
\end{align*}
\end{proof}

Everitt and Turner~\cite[Cor.\ 2]{ET19} establish a similar result with
slightly different conventions: the novelty here is
the criterion that makes $\eta$ an isomorphism.

\subsection{The poset of intervals}\label{ss:I(L)}

Here is a construction which will be of central importance for us in
\S\ref{sec:model}.  We recall its definition and refer to
\cite{Wal88} and \cite[Ch.\ 10.4]{Koz08} for more details.

\begin{definition}[The poset of intervals]\label{def:I(L)}
  For any poset $P$, let \[\Int(P)=\set{(x,y)\colon x,y\in P, x\leq y},\]
ordered by $(x_0,y_0)\leq (x_1,y_1)$ if and only if $x_0\leq x_1\leq y_1\leq 
y_0$.  Since the maximal elements of $\Int(P)$ are of the form $(x,x)$ for
$x\in P$, $\Int(P)$ has a (unique) minimal open cover 
consisting of the principal
order ideals $\set{\Int(P)_{\leq(x,x)}\colon x\in P}$.
\end{definition}

The construction is easily seen to be functorial: that is, if
$f\colon P\to Q$ is a map of posets, so is the induced map
\[
\Int(f)\colon \Int(P)\to\Int(Q)
\]
defined by letting $\Int(f)(x,y)=(f(x),f(y))$.  Translating definitions,
we see that $f\colon P\to Q$ is a quotient if and only if
the map $\Int(f)\colon \Int(P)\to\Int(Q)$ is surjective.
\begin{definition}\label{def:initial-quotient}
  Say a poset quotient $f\colon P\to Q$ is \emph{initial} if, whenever
  $f(x)=u\leq v$ for $x\in P$ and $u,v\in Q$, there exists some $y\geq x$
  with $f(y)=v$.
\end{definition}
\begin{proposition}\label{prop:initial-quotient}
  If $f\colon P\to Q$ is an initial quotient, then $\Int(f)\colon \Int(P)\to
  \Int(Q)$ is a quotient.  
\end{proposition}
\begin{proof}
  Suppose $f$ is an initial quotient.
  If $(u,v)\leq (u',v')$ in $\Int(Q)$, we have $u\leq u'\leq v'\leq v$.
  Since $f$ is surjective, there exists some $x\in P$ with $f(x)=u$.
  Because $f$
  is initial, we may find $x'$, $y'$, and $y\in P$, successively,
  with $x\leq x'\leq y'\leq y$ mapping to $u'$, $v'$, and $v$, respectively.
  That is,
  \[
  (x,y)\leq (x',y')\quad\text{and}\quad \Int(f)(x,y)=(u,v),\quad\Int(f)(x',y')=
  (u',v').
  \]
\end{proof}

Now let $\pr_1\colon \Int(P)\to P$ and $\pr_2\colon \Int(P)\to P^{\op}$ denote the
two coordinate projections.  Suppose further that $P$ has a minimum element
$\bottom$.  Then there is also a natural inclusion $\iota\colon P^{\op}
\to \Int(P)$, given by letting $\iota(y)=(\bottom,y)$ for $y\in P^{\op}$.
Then $\pr_2\circ\iota=1_{P^{\op}}$.

The maps $\pr_i$ and $\iota$ are, in fact, natural transformations.  In the
case of $\iota$, this means, explicitly, that:
\begin{proposition}\label{prop:iota_natural}
  If $f\colon P\to Q$ is an order-preserving map of posets with minimum
  elements and $f(\bottom)=\bottom$, then the following diagram commutes:
  \[
  \begin{tikzcd}
    P^{\op}\ar[r,"f"]\ar[d,"\iota"] & Q^{\op}\ar[d,"\iota"] \\
    \Int(P)\ar[r,"\Int(f)"] & \Int(Q)
  \end{tikzcd}
  \]
  \end{proposition}

Using a result from the previous section, we obtain cohomology isomorphisms:
\begin{lemma}\label{lem:pr1_iso}
  For any sheaf $\FF$ on $P$, the cohomology pullback 
\[
H^\cdot(P,\FF)\to H^\cdot(\Int(P),\pr_1^*\FF)
\]
is an isomorphism.
\end{lemma}
\begin{proof}
  For each $x\in P$, 
\begin{align*}
\pr_1^{-1}(P_{\geq x}) &= \set{(x',y')\colon x\leq x'\leq y'}\\
&=\Int(P_{\geq x}).
\end{align*}
By \cite[Thm.\ 4.1]{Wal88}, the order complex of $\Int(P_{\geq x})$ is 
homeomorphic to
that of $P_{\geq x}$.  Since the latter has a cone vertex, both order
complexes are contractible.

Likewise, the map $(x,y)\mapsto (x,x)$ is easily seen to be
a closure operation on $\Int(P)$, so if we identify its image with $P$,
we see $\pr_1$ is a quotient with contractible fibres
(Example~\ref{ex:closure}).  The result then follows by
Proposition~\ref{prop:Quillen}.  
\end{proof}

\subsection{Combinatorial blowups}\label{ss:blowups}
The notion of a combinatorial blowup was introduced in \cite{FK04}.
An expository account appears in \cite{Fe05}, and we recall here
some definitions.

\begin{definition}[Locally geometric semilattice]
A finite poset $\L$ is a {\em meet-semilattice} if every
pair of elements has a unique maximal lower bound.  Let $\bottom$
denote the (unique) least element, and let $\L_+ = \L_{>\bottom}$.
We will say a meet-semilattice $\L$ is {\em locally geometric} if, for every
$x\in \L$, the order interval $\L_{\leq x}=[\bottom, x]$ is a 
geometric lattice. This implies that for every $x\leq y$ in $\L$, the interval
$[x,y]$ is a geometric lattice, and we let
\[ d(x,y):= \rank([x,y]).\]
\end{definition}

In what follows, we will often just write ``semilattice'' for a finite 
meet-semilattice.
These posets are ranked by $\Z_{\geq 0}$. We also note that, in a finite
meet-semilattice, if a set of elements $S$ of $\L$ has an upper bound,
that upper bound is unique, in which case we will denote it by
$\bigvee S$.

\begin{definition}[The atomic complex]
\label{def:AtComplex}
For a semilattice $\L$, let $\atoms(\L)$ denote its set of atoms (elements which
cover $\bottom$).
Let $\Ka(\L)$ be the abstract simplicial complex on the set
$\atoms(\L)$ consisting of subsets $S\subseteq\atoms(\L)$ for which
the elements of $S$ have a common upper bound in $\L$.

The join of a simplex defines a natural map $\bigvee:\Ka(\L)\to\L$. Moreover,
there is a natural map $\supp:\L\to\Ka(\L)$ given by 
\begin{equation}\label{def:supp}
\supp(y) = \{a\in\atoms(\L)\colon a\leq y\}.
\end{equation}
\end{definition}

Observe that for $S\in\Ka(\L)$, $\supp\left(\bigvee S\right)\supseteq S$ but
equality does not in general hold. However, for $x\in\L$, we have
$\bigvee(\supp(x)) = x$.
\begin{example}
\label{ex:Ka of lattice}
In a geometric lattice $\LM$, every set of atoms has a common upper bound.
This implies that the atomic complex $\Ka(\LM)$ of a geometric lattice $\LM$ is
just a simplex of dimension $\abs{\atoms(\LM)}-1$.
\end{example}

Any geometric lattice $\LM$ is the lattice of flats of a matroid $\M(\LM)$
defined on its set of atoms.  
By construction, $\M(\LM)$ has no loops or parallel edges.  
We note that $\M(\LM_{\leq x})=\M(\LM)\vert_x$, the restriction matroid.
When $\L$ is a locally geometric semilattice, there is not necessarily a matroid
associated to $\L$, even though we have a matroid $\M(\L_{\leq x})$ associated
to each $x\in \L$.

At this point, we note that the locally geometric semilattices considered
here include the geometric semilattices studied by Wachs and Walker~\cite{WW86},
and we briefly compare the two notions.  We will say a
simplex $S\in \Ka(\L)$ is an {\em independent set} if $\abs{S}=d(\bottom,
\bigvee S)$. 
That is, $S$ is independent when it is a basis in the matroid
$\M(\L_{\leq\bigvee S})$ of the geometric lattice $\L_{\leq\bigvee S}$. A
geometric semilattice $\L$ is characterized by properties ``(G3)'' and ``(G4)''
in \cite{WW86}: the first is that its intervals are geometric, and the
second is that the independent sets in $\Ka(\L)$, taken all together, 
are the independent sets of a matroid on all
of $\atoms(\L)$.  We are led to relax this second condition, because
the poset of boundary strata in the wonderful compactification is, in
general, locally geometric but not geometric.  However, we will show
that some of the combinatorics of geometric semilattices (e.g., $\nbc$-bases)
can be generalized without difficulty.

\begin{definition}[A combinatorial blowup]
\label{def:blowup}
For a semilattice $\L$ and an element $p\in \L_+$, we define a poset
$\Bl_p(\L)$ together with subposets
$\L_{(p)}$ and $\L'_{(p)}$ as follows.  
Let
\begin{align*}
\L_{(p)}&=\set{x\in\L\colon x\not\geq p, \text{~and $p\vee x$ exists in
$\L$}},\\
\L'_{(p)}&= \L_{(p)}\cup\set{(p,x)\colon x\in \L_{(p)}},\\
\intertext{and}
\Bl_p(\L)&=\set{x\in \L\colon x\not\geq p}\cup 
\set{(p,x)\colon x\in \L_{(p)}},
\end{align*}
with order relations
\begin{enumerate}
\item[(i)] $y>x$ in $\Bl_p(\L)$ if $y>x$ in $\L$, 
\item[(ii)] $(p,y)>(p,x)$ in $\Bl_p(\L)$ if $y>x$ in $\L$, 
\item[(iii)] $(p,y)>x$ in $\Bl_p(\L)$ if $y\geq x$ in $\L$, 
\end{enumerate}
where in all three cases $y,x \ngeq p$.

The poset $\Bl_p(\L)$ is called the combinatorial blowup of $\L$ at $p$.
Note that the atoms in $\Bl_p(\L)$ are the atoms in $\L$ together 
with the new element $(p,\bottom)$ that can be thought of as the 
result of blowing up~$p$.  

The construction comes with an order-preserving map $\pi\colon 
\Bl_p(\L)\to\L$ given
by $\pi(x)=x$ and $\pi(p,x)=p\vee x$ for all elements $x$ and $(p,x)$
in $\Bl_p(\L)$.
\end{definition}
Here we record some basic properties.
\begin{lemma}\label{lem:alpha}
Let $\L$ be a semilattice and $p\in\L_+$. Let $\alpha$ denote the poset map
\begin{equation}\label{eq:def_alpha}
\alpha\colon
\L_{(p)}\times\{0<1\}\cong\L'_{(p)}\hookrightarrow\Bl_p(\L),
\end{equation}
Then $\alpha$ is an open embedding, and
\[
\Bl_p(\L)-\im(\alpha)=\set{x\in\L\colon p\vee x\text{~does not exist}}.
\]
\end{lemma}
\begin{proof}
  The second statement follows from the definition.
  For the first, we check that the complement of the image is a closed set,
  which is to say that if $y\geq x$ for some $x\not\in\im(\alpha)$, then
  $y\not\in\im(\alpha)$.  For this we just observe that, if $p\vee x$ does
  not exist and $y\geq x$, then $p\vee y$ does not exist either.
\end{proof}

\begin{proposition}\label{prop:semilattice}
If $\L$ is a locally geometric semilattice, so is $\Bl_p(\L)$.
\end{proposition}
\begin{proof}
The fact that $\Bl_p(\L)$ is a semilattice
appears as \cite[Lem.\ 3.2]{FK04}.  
To see that it is locally geometric, we make the following two observations.
When $x\in\L_{\not\geq p}$, the intervals $[\bottom,x]$ in $\L$ and
$[\bottom,x]$ in $\L'$ are isomorphic. When $x\in\L_{(p)}$, the interval
$[\bottom,(p,x)]$ in $\L'$ is isomorphic to $[\bottom,x]\times\{0<1\}$. 
\end{proof}
\begin{proposition}\label{prop:pi surjective}
If $\L$ is a locally geometric semilattice, 
the map $\pi\colon \Bl_p(\L)\to\L$ is surjective.
\end{proposition}
\begin{proof}
In fact, we only require $\L$ to be atomic.
If $x\not\geq p$, then $x=\pi(x)$.  If $x\geq p$, it follows easily from the
atomic property that $x=p\vee y$ for some $y\not\geq p$, 
in which case $x=\pi(y)$.
\end{proof}

\begin{example}\label{ex:one blowup}
Let $\LM$ be a geometric lattice with maximum element $\onehat$. Then
\[
\Bl_{\onehat}(\LM)=\im(\alpha)=\big(\LM_{<\onehat}\big)\cup 
\set{(\onehat,p)\colon p\in \LM_{<\onehat}},
\]
$\Bl_{\onehat}(\LM)$ is a geometric semilattice,
and $\Ka(\Bl_{\onehat}(\LM))\cong \set{\onehat}\star 
\Ka\big(\LM_{<\onehat}\big)$.  The simplices 
of $\Ka\big(\LM_{<\onehat}\big)$ are the non-spanning sets of 
the matroid $\M(\LM)$.
\end{example}


\subsection{Building sets}
Here we recall what we need from \cite{FK04} about combinatorial
building sets and nested sets in the special case of 
locally geometric semilattices.

Recall that an element $x$ in a semilattice $\L$ is said to be irreducible
if $[\bottom,x]$ is not isomorphic to the direct product of two (nontrivial)
posets.  The reduced Euler characteristic of the order complex
of the open interval $(\bottom,x)$ is known as the $\beta$-invariant of the
geometric lattice $\L_{\leq x}$, written $\beta(\L_{\leq x})$.
It is a classical fact that
$\beta(\L_{\leq x})\neq0$ if and only if $x$ is irreducible~\cite{Cr67}.

Let $\L_{\irr}$ denote the set of irreducible elements of $\L_+$.  We note
$\atoms(\L)\subseteq \L_{\irr}$ for any $\L$.  
For any $x\in \L_+$, the set of elementary divisors of $x$ is defined to be
\[
\divs(x):=\max(\L_{\irr}\cap [\bottom,x]):
\]
obviously, $\abs{\divs(x)}=1$ if and only if $x$ is irreducible.
In fact, for any $x\in \L_+$, the join map
\begin{equation}\label{eq:join decomp}
\bigvee\colon \prod_{y\in \divs(x)}[\bottom,y]\rightarrow [\bottom,x]
\end{equation}
is an isomorphism~\cite[Prop.\ 2.1]{FK04}.  
In the language of matroids, the set $\divs(x)$ indexes the connected components
of $\M(\L_{\leq x})$.
In particular, $\M(\L_{\leq x})$ is connected when $x$ is irreducible.

The isomorphism~(\ref{eq:join decomp}) inspires the definition of 
{\em combinatorial building sets\/} as stated in~\cite[Def.\ 2.2]{FK04}.
The geometric version of this definition goes back to work of De Concini and
Procesi~\cite{DP95}, where a building set indexes the irreducible components of
the normal crossings divisor in a wonderful compactification $Y(\A,\G)$ of an
arrangement complement $U(\A)$.

\begin{definition}[Combinatorial building sets]
\label{def:buildingsets}
Let $\L$ be a semilattice. A subset $\G$ in $\L_+$ is called a {\em (combinatorial) 
building set\/} if for any $x\in \L_+$ and max$\,\G_{\leq x}\,{=}\,\{x_1,\ldots,x_k\}$
there is an isomorphism of posets
\begin{equation}\label{eq:join decomp2}
    \varphi_x\,\colon \,\, \prod_{j=1}^k\,\, [\bottom,x_j] \,\, \longrightarrow \,\, [\bottom,x]
\end{equation}
with $\varphi_x(\bottom, \ldots,x_j,\ldots,\bottom)\, = \, x_j$ for $j=1,\ldots, k$.  We let $F(\L,\G;x)=\max\G_{\leq x}$, the \textit{set of factors} of $x$.
\end{definition}
It is always the case that, if $\G$ is a building set in $\L$, we have
 $\L_{\irr}\subseteq\G$, and in particular $\atoms(\L)\subseteq\G$. 

For expository simplicity, we will always assume that, if $\L$ has a 
maximum element $\onehat$, then
$\onehat\in \G$.  Having a maximum element means $\L$ is the lattice of flats
of a matroid, and $\onehat$ is irreducible exactly when the
matroid is connected, in which case the condition is satisfied automatically.

We highlight a useful property of building sets, 
which states that the set of factors of 
$x\in\L$ induce a partition of $\G_{\leq x}$:
\begin{proposition}[{\cite[Prop. 2.5(1)]{FK04}}]\label{prop:join decomp}
Suppose that $x\in \L$ and $p\in \G$ satisfies $p\leq x$.  Then there is a 
unique  $g\in F(\L,\G;x)$ for which $p\leq g$.
\end{proposition}

For instance, this provides a tool to compare join decompositions of comparable
elements. 
\begin{proposition}
\label{prop:yveeg}
Let $y\in \L_+$ and $h\in\atoms(\L)$ for which $h\not\leq y$.
Write the factors of $y\vee h\in\L$ as $F(\L,\G;y\vee h) = \{x_1,\dots,x_t\}$
where $h\leq x_t$.
Then there is a unique $y'\in[\bottom, x_t]$ such that
\[F(\L,\G; y) = \{x_1,\dots,x_{t-1}\}\cup F(\L,\G; y').\]
\end{proposition}
\begin{proof}
By the join decomposition \eqref{eq:join decomp2} in Definition
\ref{def:buildingsets}, there exist unique $z_i\in[\bottom,x_i]$ such that
$y=z_1\vee\cdots\vee z_t$. Since $h\not\leq y$ and $h\leq x_t$, we must have
$z_t<x_t$ and $z_i=x_i$ for $i<t$. Thus, $y = x_1\vee\cdots\vee x_{t-1}\vee y'$
with $y'=z_t$. Write $F(\L,\G;y') = \{y_1,\dots,y_k\}$, so that our claim is
thus $F(\L,\G;y) = \{x_1,\dots,x_{t-1},y_1,\dots,y_k\}$. This fact was proven
in \cite[p.~44]{FK04}.
\end{proof}


\subsection{Sequences of combinatorial blowups}
De Concini and Procesi's wonderful compactification can be constructed
by iterated blowups of $\P^\ell$, 
along (proper transforms of) linear subspaces indexed by a building set.
The combinatorial version of this applies combinatorial blowups (as in 
\S\ref{ss:blowups}) sequentially.

From now on, our initial semilattice will be, in fact, a geometric 
lattice, which we will denote by $\LM$.  We will also restrict ourselves
to blowups of elements of $\LM$.  The intermediate semilattices in this 
process are indexed by partial building sets, which we define below.

\begin{definition}[Partial building sets]
\label{def:partialblowups}
Let $\LM$ be a geometric lattice, and let $\G\subseteq\LM_+$ be a building set.
A {\em partial building set\/} is a subset $\H\subseteq\G$ for which
$\atoms(\LM)\subseteq\H$ and $\H^\circ:=\H\setminus\atoms(\LM)$ is an order
filter (upward-closed subset) of $\G$. 

Fix a (reverse) linear extension $\prec$ of $\H$ so that $p\geq q$ implies
$p\prec q$. Write $\H^\circ=\{p_1,\dots,p_m\}$ where $p_1\prec\cdots\prec p_m$.
The partial blowup of $\LM$ is defined as the semilattice
\[\L(\LM,\H):=\Bl_{p_{m}}\circ\cdots \circ \Bl_{p_2}\circ \Bl_{p_1}(\LM).\]
\end{definition}

The semilattice $\L(\LM,\H)$ does not depend on the choice of linear
extension $\prec$. We point out that
even though $\atoms(\LM)\subseteq\H$, we are only blowing up the elements of
$\H^\circ=\H\setminus\atoms(\LM)$ in our sequence. It is, however,
straightforward to check that 
$\Bl_p(\L(\LM,\G))\cong \L(\LM,\G)$ when $p\in\atoms(\LM)$.

By construction, $\atoms(\L(\LM,\H)) = \H$, and if $\H'=\H\cup\{p\}$ is also a
partial building set, $\L(\LM,\H)=\Bl_p(\L(\LM,\H))$.
It will be convenient simply to write $p$ 
for the atom $(p,\bottom)$ in $\L(\LM,\H)$.  In order
to avoid confusion, we will only do so when it is clear that $p$ refers
to an atom, rather than an element of higher rank.

Recall that for a semilattice $\L$, we have a natural map $\pi\colon
\Bl_p(\L)\to \L$.
If $\H\subseteq\H'$ are partial building sets, then $\L(\LM,\H')$ 
is obtained from $\L(\LM,\H)$ through a sequence of combinatorial blowups, 
and we let
\begin{equation}
\pi^{\H'}_{\H}\colon \L(\LM,\H')\to\L(\LM,\H)
\end{equation}
denote the composite of these natural
maps.  If $\H=\atoms(\LM)$, then $\L(\LM,\H)=\LM$, and
we will simply write $\pi^{\H'}\colon \L(\LM,\H')\to\LM$.

\begin{remark}\label{rmk:G in L(H)}
Let $\H\subseteq\G$ be a partial building set of a geometric lattice $\LM$.
Since $\H$ is an order filter of $\G$, the elements of $\G\setminus\H$ are
elements of $\LM$ which survive the sequence of blowups prescribed by $\H$.
Identifying $\H$ as the set of atoms in $\L(\LM,\H)$, we may thus view $\G$ as a
subset of $\L(\LM,\H)$. In fact, by \cite[Prop.~3.3]{FK04}, the set $\G$ is a
building set for the semilattice $\L(\LM,\H)$.
\end{remark}

In contrast to arbitrary sequences of combinatorial blowups, our 
constructions have several useful properties.  The first follows from
\cite[Prop.\ 2.3]{FK04}, while the second is a straightforward
generalization of \cite[Thm.\ 2.3(3b')]{DP95}, and the third follows from the
fact that a partial building set is an order filter.
\begin{proposition}\label{prop:BICO}
Let $\H$ be a partial building set in a building set $\G$ in $\LM$.
\begin{enumerate}[(a)]
\item Suppose $\H'=\H\cup\set{p}$ is also a partial building set.
Then the (open) upper order interval $\L(\LM,\H)_{>p}$ is disjoint from
$\G$.  In particular, it contains no irreducible elements.
 \label{prop:BICO1}
\item Suppose $x,q\in \G$ for which $x\not\geq q$ and 
$x\wedge q\neq \bottom$.  If $q\in \H$, then $x\vee q\in \H$ too.
\label{prop:BICO2}
\item If $p\in\G-\H$ and $S\subseteq\H$, then $S_{<p}\subseteq\atoms(\LM)$.
\label{prop:BICO3} 
\end{enumerate}
\end{proposition}

We can keep track of how the factors of an element $y\in \L(\LM,\H')$
behave under a single blow-down, viewing $\G$ as a building set in both
$\L(\LM,\H)$ and $\L(\LM,\H')$ (via Remark \ref{rmk:G in L(H)}).

\begin{proposition}\label{prop:blowdown F}
In the notation above, let $y\in \L(\LM,\H')$.  Then
\[
F(\L(\LM,\H),\G;\pi^{\H'}_{\H}(y)) = 
\begin{cases}
F(\L(\LM,\H'),\G;y)_{\not<p} &  \text{if $\pi(y)\geq p$;}\\
F(\L(\LM,\H'),\G;y) & \text{otherwise.}
\end{cases}
\]
\end{proposition}
\begin{proof}
We abbreviate $\pi=\pi^{\H'}_{\H}$, $\L = \L(\LM,\H)$ and
$\L'=\L(\LM,\H')=\Bl_p(\L)$.
First, if $\pi(y)\not\geq p$, then the restriction of
$\pi\colon [\bottom,y]\to [\bottom,\pi(y)]$ is an isomorphism, and 
the conclusion follows.

Otherwise, $y=(p,x)$ for some $x\in\L_{(p)}$ and $p\in F(\L',\G;y)$.
The set  $F(\L',G;y)_{\not<p}=F(\L',\G;x)_{\not<p}\cup\{p\}$ contains pairwise
incomparable elements in $\L$ whose join is $\pi(y)=p\vee x$, so by
\cite[Prop.~2.8]{FK04}, it equals $F(\L,\G;\pi(y))$. 
\end{proof}

If the element $\onehat$ is contained in a partial building set $\H$,
clearly it must come first in any linear order.  This leads to the
following observation:
\begin{proposition}\label{prop:onehat}
If $\onehat\in \H$, then $\onehat$ is a cone vertex in the atomic
complex: that is,
\[
\Ka(\L(\LM,\H))\cong \set{\onehat}\star \lk_{\onehat}(\Ka(\L(\LM,\H))).
\]
\end{proposition}
\begin{proof}
If $\H=\set{\onehat}\cup\atoms(\LM)$, this is Example \ref{ex:one blowup}.
In general, it follows by induction on $\abs{\H}$.
\end{proof}

We conclude this section with a simple example.
\begin{example}\label{ex:delA3}
Let $\LM$ be the lattice of the rank-$3$ matroid on $[5]$ with two $3$-element
flats, $\set{1,2,4}$ and $\set{1,3,5}$.  This is the intersection lattice of
the hyperplane arrangement defined, in order, by the linear forms
$(x,y,z,x-y,x-z)$.  Let $\G=\G_{\min}$ be the minimal building set,
ordered 
\[
\onehat\prec 124\prec 135\prec 1\prec \cdots \prec 5.
\]
$\G$ contains three
(nontrivial) partial building sets.  Their atomic complexes are cones over
the complexes drawn in Figure~\ref{fig:delA3}.
\end{example}

\begin{figure}
\subfloat[$\H^\circ=\set{\onehat}$]{
\begin{tikzpicture}[scale=0.88]
\tikzstyle{every node}=[circle,draw,fill=white,inner sep=1.7pt,scale=0.75]
\coordinate (P1) at (0,0);
\coordinate (P2) at (-1,-0.7);
\coordinate (P4) at (-1,0.7);
\coordinate (P3) at (1,-0.7);
\coordinate (P5) at (1,0.7);

\fill[pattern=north east lines, pattern color=green!50!blue!50!white] 
(P1) -- (P2) -- (P4);  
\fill[pattern=north east lines, pattern color=green!50!blue!50!white] 
(P1) -- (P3) -- (P5); 

\node[label=below:$1$] (V1) at (P1) {};
\node[label=left:$2$] (V2) at (P2) {}; 
\node[label=left:$4$] (V4) at (P4) {}; 
\node[label=right:$3$] (V3) at (P3) {}; 
\node[label=right:$5$] (V5) at (P5) {}; 


\draw[style=very thick] (V5) .. controls (-1,2) and (-2.5,0.7) .. (V2);
\draw[line width=4pt,color=white] (V4) .. controls (1,2) and (2.5,0.7) .. (V3);

\draw[style=very thick] (V4) -- (V1) -- (V2);
\draw[style=thin] (V2) -- (V4);

\draw[style=very thick] (V5) -- (V1) -- (V3);
\draw[style=thin] (V3) -- (V5);

\draw[style=very thick] (V4) -- (V5);
\draw[style=very thick] (V2) -- (V3);
\draw[style=very thick] (V4) .. controls (1,2) and (2.5,0.7) .. (V3);

\end{tikzpicture}
}
\subfloat[$\H^\circ=\set{\onehat,124}$]{
\begin{tikzpicture}[scale=0.88]
\tikzstyle{every node}=[circle,draw,fill=white,inner sep=1.7pt,scale=0.75]
\coordinate (P1) at (0,0);
\coordinate (P2) at (-1,-0.7);
\coordinate (P4) at (-1,0.7);
\coordinate (P124) at (-0.67, 0);
\coordinate (P3) at (1,-0.7);
\coordinate (P5) at (1,0.7);

\fill[pattern=north east lines, pattern color=green!50!blue!50!white] 
(P1) -- (P3) -- (P5); 

\node[label=below:$1$] (V1) at (P1) {};
\node[label=left:$2$] (V2) at (P2) {}; 
\node[label=left:$4$] (V4) at (P4) {}; 
\node[label=left:$124$] (V124) at (P124) {};
\node[label=right:$3$] (V3) at (P3) {}; 
\node[label=right:$5$] (V5) at (P5) {}; 


\draw[style=very thick] (V5) .. controls (-1,2) and (-2.5,0.7) .. (V2);
\draw[line width=4pt,color=white] (V4) .. controls (1,2) and (2.5,0.7) .. (V3);

\draw[style=very thick] (V5) -- (V1) -- (V3);
\draw[style=thin] (V3) -- (V5);
\draw[style=very thick] (V1) -- (V124) -- (V4);
\draw[style=very thick] (V124) -- (V2);
\draw[style=very thick] (V4) -- (V5);
\draw[style=very thick] (V2) -- (V3);
\draw[style=very thick] (V4) .. controls (1,2) and (2.5,0.7) .. (V3);

\end{tikzpicture}
}
\subfloat[$\H^\circ=\set{\onehat,124,135}$]{
\begin{tikzpicture}[scale=0.88]
\tikzstyle{every node}=[circle,draw,fill=white,inner sep=1.7pt,scale=.75]
\coordinate (P1) at (0,0);
\coordinate (P2) at (-1,-0.7);
\coordinate (P4) at (-1,0.7);
\coordinate (P124) at (-0.67, 0);
\coordinate (P3) at (1,-0.7);
\coordinate (P5) at (1,0.7);
\coordinate (P135) at (0.67, 0);

\node[label=below:$1$] (V1) at (P1) {};
\node[label=left:$2$] (V2) at (P2) {}; 
\node[label=left:$4$] (V4) at (P4) {}; 
\node[label=left:$124$] (V124) at (P124) {};
\node[label=right:$3$] (V3) at (P3) {}; 
\node[label=right:$5$] (V5) at (P5) {}; 
\node[label=right:$135$] (V135) at (P135) {};


\draw[style=very thick] (V5) .. controls (-1,2) and (-2.5,0.7) .. (V2);
\draw[line width=4pt,color=white] (V4) .. controls (1,2) and (2.5,0.7) .. (V3);

\draw[style=very thick] (V1) -- (V135) -- (V3);
\draw[style=very thick] (V135) -- (V5);
\draw[style=very thick] (V1) -- (V124) -- (V4);
\draw[style=very thick] (V124) -- (V2);
\draw[style=very thick] (V4) -- (V5);
\draw[style=very thick] (V2) -- (V3);
\draw[style=very thick] (V4) .. controls (1,2) and (2.5,0.7) .. (V3);

\end{tikzpicture}
}
\caption{The link of $\onehat$ in $\Ka(\L(\LM,\H))=\N(\LM,\H)$ from 
Example~\ref{ex:delA3}.
\label{fig:delA3}}
\end{figure}


\subsection{Nested sets}

An important concept related to building sets are nested sets, which form an
abstract simplicial complex. In the geometric setting, the faces of the nested
set complex index the non-trivial intersections of divisor components in the
wonderful compactification. 
Here, we extend the notion of nested sets to partial building sets. Once again
they 
form an abstract simplicial complex, and in Theorem~\ref{thm:nested} we
will show that this is none other than the atomic complex for the semilattice
$\L(\LM,\H) = 
\Bl_{p_{m}}\circ\cdots \circ \Bl_{p_2}\circ \Bl_{p_1}(\LM)$.
In the classical case, nested sets are independent sets in the semilattice
of the blowup; in our generalization, nested sets are no longer necessarily
independent, and the nested set complex is no longer pure in general (see
Figure~\ref{fig:delA3} for examples).

Throughout this section, let $\H$ be a partial building set in a geometric
lattice $\LM$.
\begin{definition}[Nested set complex]
\label{def:nestedsets}
A subset $\NS$  of $\H$  is called {\em $\H$-nested\/} if, for any set of 
incomparable elements $x_1,\ldots,x_t$ in $\NS$ of cardinality at least two, 
the join $x_1\vee \ldots \vee x_t$ is not contained in~$\H$. 
The $\H$-nested sets form an abstract simplicial complex $\N(\LM,\H)$, called
the {\em nested set complex\/} with respect to $\H$. 

A set of pairwise incomparable elements will be called an {\em antichain\/},
and an antichain with at least two elements will be called a {\em nontrivial
antichain\/}.
\end{definition}

In the definition of a nested set, we emphasize that we 
are comparing elements and using the join in the original lattice $\LM$, rather
than in $\L(\LM,\H)$.
Observe that if $\H$ is chosen to be the maximal building set $\LM_+$,
then the nested set complex coincides with the order complex of $\LM$.

First, we recall a useful relationship between nested sets and the join
decomposition of Definition \ref{def:buildingsets}\eqref{eq:join decomp2}.

\begin{proposition}[{\cite[Prop. 2.8]{FK04}}]\label{prop:nestedfactors}
Let $\G$ be a (full) building set, and let $S\in\N(\LM,\G)$. Then 
$\max S = F(\LM,\G;y)$ where $y=\bigvee S\in \LM$.  That is, the maximal
elements of $S$ in $\LM$ are the factors of $y$ in the join decomposition.
\end{proposition}

Our immediate goal is to show that $\N(\LM,\H)=\Ka(\L(\LM,\H))$, which we do in
Theorem \ref{thm:nested}. The two extremes for choices of $\H$ will be
fundamental, and we state these here:
\begin{proposition}\label{prop:N=K}
If either $\H=\G$ is a building set or $\H=\atoms(\LM)$, then
$\Ka(\L(\LM,\H)) = \N(\LM,\H)$.
\end{proposition}
\begin{proof}
When $\H=\G$, \cite[Thm.~3.4]{FK04} states that $\L(\LM,\G)$ is the face poset
of the simplicial complex $\N(\LM,\G)$, which means that $\Ka(\L(\LM,\G)) =
\N(\LM,\G)$.

On the other hand, when $\H=\atoms(\LM)$, both $\Ka(\L(\LM,\H))$ and
$\N(\LM,\H)$ consist of all subsets of $\H$, hence they are both full simplices
of dimension $\abs{\atoms(\LM)}-1$.
\end{proof}

In order to deduce Theorem~\ref{thm:nested} from these extremes, we will use the
following map to ``blow-down'' subsets of a partial building set.
Suppose that $\H\subseteq \H'$ are partial building sets obtained from a
building set $\G$ in a geometric lattice $\LM$. 
Recall that we may identify the atoms $\atoms(\L(\LM,\H'))$ with $\H'$ (and
similarly for $\H$): through this identification, define for $S\subseteq\H'$ a
set
\[
\Pi(\NS) = \Pi^{\H'}_{\H}(\NS)=
\left\{h\in\H \colon h\leq \pi^{\H'}_{\H}(g) \text{ for some } g\in \NS\right\}.
\]
That is, one first blows down the elements of $S$ from $\L(\LM,\H')$ to
$\L(\LM,\H)$ and then collects the atoms of $\L(\LM,\H)$ lying underneath.
We can view $\Pi$ as a function on the atomic complexes via
\[
\Pi: \Ka(\L(\LM,\H')) \xrightarrow{\ \vee\ } \L(\LM,\H')
\xrightarrow{\pi^{\H'}_\H} \L(\LM,\H) \xrightarrow{\supp} \Ka(\L(\LM,\H)),
\]
where $\supp$ was defined in Definition \ref{def:AtComplex}\eqref{def:supp}.
Clearly, in the case where $\H'=\H\cup\set{p}$, we have
\begin{equation}\label{eq:onestep Pi}
\Pi(S)=\begin{cases}
S & \text{if $p\not\in S$;}\\
(S-\set{p})\cup \supp(p)& \text{otherwise.}
\end{cases}
\end{equation}

In Lemma \ref{lem:Pi:a} below, we will see that $\Pi$ can also be viewed as a
function on the nested set complexes.
Note, however, that $\Pi$ is not a simplicial map; that is, an $i$-simplex is
not necessarily sent to an $i$-simplex.

\begin{example}
\label{ex:delA3 Pi}
Recall the lattice $\LM$ from Example \ref{ex:delA3}. The nested set complex
(equivalently, atomic complex) for each partial building set $\H\subseteq\G =
\{\onehat,124,135\}$ is a cone over the corresponding complex depicted in
Figure \ref{fig:delA3}. 

The function $\Pi$ sends the vertex $135$ in $\N(\LM,\G)$ to the
2-simplex $\{1,3,5\}$ in $\N(\LM,\H)$, where $\H$ is either
$\{\onehat\}$ or $\{\onehat,124\}$.
Observe that the set $\{1,3,5\}$ is $\H$-nested but not $\G$-nested.

As another example, the 1-simplex $S=\{2,124\}\in\N(\LM,\G)$ is preserved when
blowing down to $\N(\LM,\{\onehat,124\})$; that is, $\Pi(S)=S$. But then
applying $\Pi$ again, one obtains the 2-simplex $\{1,2,4\}$ in
$\N(\LM,\{\onehat\})$.
\end{example}

\begin{lemma}
\label{lem:Pi:a}
Let $\H\subseteq\H'$ be partial building sets in a geometric lattice $\LM$.
If $S$ is $\H'$-nested, then $\Pi^{\H'}_{\H}(S)$ is $\H$-nested.
That is, $\Pi^{\H'}_{\H}$ defines a function $\N(\LM,\H')\to\N(\LM,\H)$.
\end{lemma}
\begin{proof}
Let $\Pi=\Pi^{\H'}_{\H}$ for short.
It suffices to check the case that $\H'=\H\cup\set{p}$. Suppose we have a
set 
$S\subseteq\H'$ for which $\Pi(S)$ is not $\H$-nested: that is, there is a
nontrivial antichain $T$ contained in $\Pi(S)$ whose join $\bigvee T$ is an
element of $\H$. We will show that this implies $S$ is not $\H'$-nested.
We have two cases, depending on whether $T\subseteq S$.

First, suppose that $T\subseteq\NS$. Then $\NS$ contains a nontrivial antichain
$T$ whose join is an element of $\H$, hence also $\H'$ since $\H\subseteq\H'$.
Thus, $\NS$ is not $\H'$-nested.

Now suppose that $T\not\subseteq\NS$. Then $p\in\NS$ and $T_{<p}\neq\emptyset$
(by \eqref{eq:onestep Pi}), so $\bottom<\bigvee T_{<p}\leq(\bigvee T)\wedge p$.
Since $\bigvee T\in \H\subseteq \H'$ and $p\in\H'$, this implies that $(\bigvee
T)\vee p\in\H'$ by Proposition \ref{prop:BICO}\eqref{prop:BICO2}.
Now, since $\bigvee T\in\H$ and $p\in\H'\setminus\H$, we must have
$T_{\not<p}\neq \emptyset$. It follows that the set $T_{\not<p}\cup\{p\}$ is a
nontrivial antichain contained in $\NS$. But $(\bigvee T_{\not<p})\vee
p=(\bigvee T)\vee p\in\H'$, which means that $\NS$ cannot be $\H'$-nested.
\end{proof}

\begin{lemma}
\label{lem:Pi:b}
Let $\H\subseteq\H'$ be partial building sets in a geometric lattice $\LM$.
For any simplex $S\in \Ka(\L(\LM,\H))$, there is some simplex $T\in
\Ka(\L(\LM,\H'))$ such that $S\subseteq \Pi^{\H'}_\H(T)$. 
\end{lemma}
\begin{proof}
Let us abbreviate $\Pi=\Pi^{\H'}_\H$.
Again, it suffices to check the case that $\H'=\H\cup\{p\}$, where 
$\L(\LM,\H')=\Bl_p(\L(\LM,\H))$.
Let $S\in \Ka(\L(\LM,\H))$, and write $S=S_{\not<p}\cup S_{<p}$. 
We have two cases: either $\bigvee S_{< p} = p$ or $\bigvee S_{< p}<p$.

If $\bigvee S_{< p} =p$, then we claim $T= S_{\not<p}\cup
\{(p,\bottom)\}$ is a simplex in $\Ka(\L(\LM,\H'))$ with $S\subseteq \Pi(T)$.
It is clear that $S\subseteq\Pi(T)$. To show that $T$ is indeed a simplex,
we first note that $\bigvee S_{\not<p}$ exists in $\L(\LM,\H')$, since this join
exists and is not above $p$ in $\L(\LM,\H)$. Moreover, since $\bigvee S_{<p}=p$
the join $\bigvee S_{\not<p}\vee p = \bigvee S$ exists in $\L(\LM,\H)$, so
$\bigvee T = (p,\bigvee S_{\not\leq p})\in \L(\LM,\H')$. 

Otherwise, $\bigvee S_{<p}<p$, and we claim $T = S$ is
in $\Ka(\L(\LM,\H'))$.  This is
true because $\bigvee S$ exists and is not above $p$ in $\L(\LM,\H)$, 
hence the join exists in $\L(\LM,\H')$.
\end{proof}

We are now ready to prove that the nested set complex and atomic complex
of $\L(\LM,\H)$ agree, which we noted in Proposition \ref{prop:N=K}
is a fundamental fact in the case where $\H$ is a full building set.
\begin{theorem}
\label{thm:nested} 
Let $\LM$ be a geometric lattice
and $\H\subseteq\LM_+$ a partial building set. Then 
a set $\NS\subseteq \H$ is $\H$-nested if and only if $\NS\in\Ka(\L(\LM,\H))$. 
\end{theorem}
\begin{proof}
Let $\G$ be a building set that contains $\H$.

First, assume that $\NS\in \Ka(\L(\LM,\H))$. By Lemma \ref{lem:Pi:b}, there is
a simplex $T\in\Ka(\L(\LM,\G))$ for which $\Pi^\G_\H(T)\supseteq \NS$. 
Since $\G$ is a (full) building set, $T$ is $\G$-nested (see Proposition
\ref{prop:N=K}), which implies that $\Pi^\G_\H(T)$ is $\H$-nested by Lemma
\ref{lem:Pi:a}. The property of being $\H$-nested is inherited by subsets, so
$\NS$ must also be $\H$-nested.

For the converse, we use induction on $\abs{\H}$. The base case, when
$\H=\atoms(\LM)$, was mentioned in Proposition \ref{prop:N=K}. Now assume that every
$\H$-nested set is a simplex of $\Ka(\L(\LM,\H))$, and we will show that the
same holds for the next partial building set $\H'=\H\cup\set{p}\subseteq\G$.
Assume that $\NS$ is $\H'$-nested, and let $\Pi=\Pi^{\H'}_\H$. By Lemma
\ref{lem:Pi:a}, $\Pi(\NS)$ is $\H$-nested hence a simplex in
$\Ka(\L(\LM,\H))$. This means that $\bigvee \Pi(\NS)$ exists in $\L(\LM,\H)$.  
We have two cases, depending on whether $p\in \NS$.

If $p\in\NS$, then $\Pi(\NS) = \NS_{\neq p}\cup\{g\in\atoms(\LM)\colon g<p\}$.
Consider $\bigvee \NS_{\neq p}$ in $\L(\LM,\H)$, which exists because $\bigvee
\Pi(\NS)$ exists. Since $\NS$ is $\H'$-nested, we must have $\bigvee
\NS_{<p}\neq p$ and hence $\bigvee\NS_{\neq p}\not\geq p$. Thus, $\bigvee
\NS_{\neq p}$ exists in $\L(\LM,\H)$. Moreover, since $(\bigvee S_{\neq p})\vee
p = \bigvee\Pi(S)$ exists in $\L(\LM,\H)$, we have $\bigvee S = (p,\bigvee
S_{\neq p})\in\L(\LM,\H')$. Therefore, $S\in\Ka(\LM,\H')$.

Otherwise, $p\notin S$ and $\Pi(S)=S$. The order of blowups implies that
$\NS_{<p}\subseteq \atoms(\LM)$ (Proposition \ref{prop:BICO}\eqref{prop:BICO3});
in particular, $\NS_{<p}$ is an antichain.
We claim that $\bigvee\NS_{<p}\neq p$. This is immediate if it has less than two
elements, and otherwise $\NS_{<p}$ is a nontrivial antichain and the claim
follows from the assumption that $\NS$ is $\H'$-nested. 
Now since $\bigvee\NS_{<p}\neq p$, we have $\bigvee \NS_{<p}<p$; in particular,
$\bigvee\NS_{<p}$ exists and is not above $p$ in $\L(\LM,\H)$.  By the
construction of a blowup, then, $\bigvee\NS$ exists in $\L(\LM,\H')$, which is
to say $\NS\in\Ka(\L(\LM,\H'))$. 
\end{proof}

In the next section, we will study how blowups affect nested sets.  Before
doing so, though, we include one more easy property for later reference.

\begin{proposition}\label{prop:factors1}
Let $\H$ be a partial building set in a geometric lattice $\LM$, and let
$S\subseteq\H$ be a nontrivial antichain. 
If $S$ is $\H$-nested, then $\wedge S=\bottom$.
\end{proposition}
\begin{proof}
For a contradiction, assume $S$ is $\H$-nested and $\wedge S\neq\bottom$.
Let $g,g'\in S$. Since $S$ is an antichain with $\wedge S\neq\bottom$, $g$ and
$g'$ are incomparable with $g\wedge g'\neq\bottom.$ By
Proposition~\ref{prop:BICO}\eqref{prop:BICO2}, $g\vee g'\in\H$, which would
contradict nestedness of $\set{g,g'}\subseteq S$.
\end{proof}

\subsection{Blowing up nested sets}
There is another map on nested sets which we will use, for example, to deal with
the join decomposition in Definition \ref{def:buildingsets}\eqref{eq:join decomp2}
throughout our sequence of
blowups. This map will be of particular importance when working with
\textit{local} building sets in \S\ref{ss:H decomp}.
To define it, 
suppose that $\H$ and $\H'=\H\cup\{p\}$ are partial building sets in $\LM$. 
For a subset $\NS\subseteq\H$, define a subset $\eta(\NS)\subseteq\H'$ by
\[
\eta(\NS) = \eta^{\H'}_{\H}(\NS) = \begin{cases} 
\NS & \text{ if } \bigvee \NS_{<p} \neq p\\
(\NS_{\not<p})\cup\{p\} & \text{ if } \bigvee \NS_{<p} = p
\end{cases}
\]
If $\H$ and $\H'=\H\cup\{p_{m+1},\dots,p_n\}$ are partial building sets, 
we can define $\eta^{\H'}_\H$ as a composition of these maps, and it is
particularly useful to consider $\H'=\G$ (a full building set).
Note that just as with $\Pi$, while $\eta$ is a function on simplicial complexes
(via Lemma \ref{lem:eta:a} below), it is not a simplicial map.

\begin{example}
\label{ex:delA3 eta}
Recall the lattice $\LM$ from Example \ref{ex:delA3} and Figure \ref{fig:delA3}.
For the 2-simplex $S=\{1,2,4\}$ and the 1-simplex $T=\{2,4\}$ in
$\N(\LM,\atoms(\LM))$, we have $\eta(S)=\eta(T)=\{124\}$ in both
$\N(\LM,\H)$ (with $\H^\circ=\{\onehat,124\}$) and $\N(\LM,\G)$.

The 1-simplex $\{2,3\}$ is preserved by $\eta$.
\end{example}

\begin{proposition}
\label{prop:etaS=S}
For $S\subseteq\H$, we have
\[\eta_\H^\G(S) = \eta_\H^\G\left(S\cap \atoms(\LM)\right) \cup
\left(S\cap \H^\circ\right).\]
In particular, if $S\subseteq\H^\circ$, then $\eta_\H^\G(S)=S$.
\end{proposition}
\begin{proof}
The fact that $S\cap \H^\circ$ does not interact with the sequence of blowups
follows from the observation that $S_{<p}\subseteq\atoms(\LM)$ whenever
$p\in\G-\H$ (Proposition \ref{prop:BICO}\eqref{prop:BICO3}).
\end{proof}

\begin{remark}
\label{rmk:eta=factors}
Recall from Remark \ref{rmk:G in L(H)}, we may view $\G$ as a building set for
the semilattice $\L(\LM,\H)$.
For an $\H$-nested set $S$, the factors of $z = \bigvee S\in\L(\LM,\H)$ with
respect to the building set $\G\subseteq\L(\LM,\H)$ are given by $\eta^G_\H(S)$.
More explicitly, 
\[ F\left(\L(\LM,\H),\G;\bigvee S\right) = \eta^G_\H(S).\]
Indeed, by Proposition \ref{prop:join decomp},
the factors of $z$ index a partition of the atoms in $\L(\LM,\H)$ that
lie below $z$. As we apply $\eta$ is the order of blowups, we replace the atoms
below each not-yet-blown-up factor $g\in\G\setminus\H$ with $g$ itself, leaving
other atoms untouched.
\end{remark}

We will show in the next two lemmas that $\eta$ preserves nestedness as well as
incomparability, leading to a useful tool in Proposition \ref{prop:factors2}.

\begin{lemma}
\label{lem:eta:a}
Let $\H\subseteq \H'$ be partial building sets in a geometric lattice $\LM$.
If $S$ is also $\H$-nested, then $\eta^{\H'}_\H(S)$ is $\H'$-nested. That is,
$\eta^{\H'}_\H$ defines a
function $\N(\LM,\H)\to\N(\LM,\H')$.
\end{lemma}
\begin{proof}
It suffices to check the case that $\H'=\H\cup\set{p}$, which we do by
contrapositive. Suppose that $\NS\subseteq\H$ with $\eta(\NS)\notin
\N(\LM,\H')$;  we will show that then $\NS$ is not $\H$-nested. Since
$\eta(\NS)$ is not $\H'$-nested, there is a nontrivial antichain
$T\subseteq\eta(\NS)$ for which $\bigvee T\in\H'$. We have two cases, depending
on whether $p\in T$.

If $p\notin T$, then by definition of $\eta$, $T\subseteq S$ and $\bigvee T\neq
p$. Thus, $\bigvee T\in\H'\setminus\{p\}=\H$ and hence $\NS$ is not $\H$-nested.

Otherwise, we have $p\in T$. Since $T$ is a nontrivial antichain, it contains at
least one element that is incomparable with $p$. This implies that $\bigvee T>p$
and hence $\bigvee T\in\H'\setminus\{p\}=\H$.
Consider the set
\[
U=\set{h\in \NS_{<p}\ \colon  h\not\leq y \text{ for any }y\in T_{\neq p}}
\]
Then $U\cup T_{\neq p}\subseteq \NS$ is an antichain
with $\bigvee(U\cup T_{\neq p}) = \bigvee T$ in $\H$.
It remains to show that $U\cup T_{\neq p}$ is a nontrivial antichain, proving
that $S$ is not $\H$-nested. Since $T_{\neq p}\neq\emptyset$, this is immediate
if $U\neq\emptyset$. If $U=\emptyset$, then we have $\bigvee T_{\neq p}=\bigvee
T>p$ which implies that $T_{\neq p}$ has more than  one element.
\end{proof}

\begin{lemma}
\label{lem:eta:b}
Let $\H\subseteq \H'$ be partial building sets in a geometric lattice $\LM$.
If $\NS\subseteq \H$ is an antichain, then so is $\eta^{\H'}_\H(\NS)$. 
\end{lemma}
\begin{proof}
Again, it suffices to check the case that $\H'=\H\cup\set{p}$.
The statement is trivial if $\bigvee \NS_{<p}\neq p$. So we write
$\NS=\NS_{\not<p}\cup\NS_{<p}$ where $\bigvee\NS_{<p}=p$, and hence
$\eta(\NS) = \NS_{\not<p}\cup\{p\}$. 
The set $\NS_{\not<p}$ is an antichain by assumption; to show that $\eta(\NS)$
is an antichain we need only show that any $y\in \NS_{\not<p}$ is incomparable
with $p$. It is clear that $y\not\leq p$, and we also have $y\not\geq p$ since
$y$ is incomparable with every element of $\NS_{<p}$.
\end{proof}

\begin{proposition}
\label{prop:factors2}
Let $\H$ be a partial building set in a geometric lattice $\LM$.
Let $S\subseteq\H$ and $g\in\H\setminus S$. If $S$ is a nontrivial antichain
with $\bigvee S\geq g$, and $g>h$ for all but at most one element $h\in S$,
then $S$ is not $\H$-nested.
\end{proposition}
\begin{proof}
For a contradiction, assume that $S$ is an $\H$-nested nontrivial antichain with
$\bigvee S\geq g$ and $g>h$ for all but at most one element $h\in S$.
We will consider $\eta=\eta^\G_\H$, where $G\supseteq\H$ is a (full) building
set. By Lemma \ref{lem:eta:a}, $\eta(\NS)$ is $\G$-nested. 
Note that by the assumption $g>h$ for all but at most one element $h\in S$,
we must have $g\in\H^\circ$. We consider two cases, depending on the size of
$\eta(S)$. 

If $\eta(\NS)=\set{p}$ for some $p\in\G$, then (since $S$ has more than one
element) we must have $p\in \G\setminus\H$ and hence $g\prec p$.
But by assumption, $p=\bigvee \NS \geq g$ and $g\in\H^\circ$, contradicting  the
compatible order property.

So assume that $\eta(\NS)$ has more than one element. By Lemma
\ref{lem:eta:b}, the elements of $\eta(\NS)$ are pairwise incomparable. Then by
Proposition \ref{prop:nestedfactors}, the elements of $\eta(\NS)$ must be the
factors of their join $\bigvee\eta(\NS)$. But since $g>h$ for all but maybe one
$h\in\NS$ and
$|\eta(\NS)|>1$, there is some $y\in\eta(\NS)$ such that $g>y$. This contradicts
the factors being maximal under $\bigvee \eta(\NS)$, since 
$\bigvee\eta(\NS)=\bigvee S\geq g$.
\end{proof}

We conclude by remarking that while $\eta$ and $\Pi$ are extremely useful tools
for going between the simplicial complexes along a sequence of blowups, they are
not inverses of each other. 

\begin{example}
Recall the lattice $\LM$ from Examples \ref{ex:delA3}, \ref{ex:delA3 Pi}, and
\ref{ex:delA3 eta}, and Figure \ref{fig:delA3}.
Let $\H^\circ=\{\onehat,124\}$, and let $S=\{2,124\}$, a 2-simplex in $\N(\LM,\H)$.
Then $\Pi(S)=\{1,2,4\}\in\N(\LM,\atoms(\LM))$ but $\eta(\Pi(S))=\{124\}\neq S$.

On the other hand, consider the 1-simplex $T=\{2,4\}$ in 
$\N(\LM,\atoms(\LM))$.
Then $\eta(T) = \{124\}$ in $\N(\LM,\G)$ and $\Pi(\eta(T)) = \{1,2,4\}\neq T$ 
in $\N(\LM,\atoms(\LM))$.
\end{example}


\subsection{Building set decompositions}\label{ss:H decomp}

In the geometric case, a sequence of blowups produces an arrangement of
hypersurfaces whose intersections are themselves products of 
blowups of minors of the original arrangement.  In the case where $\H=\G$ 
is a building set, this is well-known \cite[p.\ 482]{DP95}.  In our
purely combinatorial and partially blown up setting, we will want to show that
the generalized cohomology algebras we construct have analogous tensor product
decompositions. Here, we establish the notation and basic results needed for
such decompositions that appear in \S\ref{ss:DP}.

\begin{definition}[Local intervals]\label{def:factors plus}
Let $\L(\LM,\H)$ be a partial blowup of a geometric lattice $\LM$. 
For each $y\in\L(\LM,\H)$, let
\[
F^+(\L(\LM,\H),\G;y):=F(\L(\LM,\H),\G;y)\cup\set{\onehat}\subseteq\G,
\]
where we recall the notation for the factors of $y$ from Definition
\ref{def:buildingsets}, regarding $\G$ as a building set in the semilattice
$\L(\LM,\H)$ via Remark \ref{rmk:G in L(H)}.
Since in our discussion the building set $\G$ is now regarded as a fixed choice,
we will often abbreviate $F^+(y)=F^+(\L(\LM,\H),\G;y)$.

The set  $F^+(y)$ is $\G$-nested and may alternatively be written as
\[
F^+(y)= \eta_\H^\G\left(\supp_{\N(\LM,\H)}(y)\right)\cup\set{\onehat},
\]
using the point of view of the previous section, 
where $\supp:\L(\LM,\H)\to\N(\LM,\H)$ is the map from 
Definition \ref{def:AtComplex}\eqref{def:supp}.

For each $g\in F^+(y)$, we define an interval in $\LM$ by
\begin{equation}\label{def:Ig}
\LM_{y,g}:=[z_y(g),g], \quad \text{ where } \quad z_y(g):=\bigvee_{\substack{
f\in F(y)\\ f<g}} f. 
\end{equation}
As usual, we will write $z(g)$ in place of $z_y(g)$ when $y$ is understood.
\end{definition}

Each closed interval is $\LM_{y,g}$ a geometric lattice, 
and clearly the half-open intervals $(\LM_{y,g})_+=(z_y(g),g]$ 
for $g\in F^+(y)$ are disjoint.
In this section, we will describe an induced partial building set $\H_{y,g}$ 
on $\LM_{y,g}$.

For a simplex $S$ of a simplicial complex $K$, let 
$\st_K(S):=\set{T\in K\colon S\subseteq T}$ denote the star of $S$ in $K$,
and $\clst_K(S)$ the smallest subcomplex of $K$ containing it.
Let $K_0$ denote the vertices of a simplicial complex $K$.
In our setting with $K = \N(\LM,\H)$, we give the explicit description
\[\clst_{\N(\LM,\H)}(S)_0 = \{p\in\H\colon S\cup\{p\}\in\N(\LM,\H)\}.\]

\begin{proposition}\label{prop:restrict H}
Suppose that $y\in\L(\LM,\H)$ and let $S=\supp(y)\in\N(\LM,\H)$.
For each $p\in\H$, the
set $\set{g\in F^+(y)\colon p\leq g}$ has a unique minimum element, which we
denote by
$\hat{p}$. Then $z_y(\hat{p})<p\vee z_y(\hat{p})$, and 
the assignment $p\mapsto p\vee z_y(\hat{p})\in(\LM_{y,\hat{p}})_+$ defines a map
\[
\zeta=\zeta_{y,\H}\colon \H\to
\bigsqcup_{g\in F^+(y)} (\LM_{y,g})_+.
\]
Furthermore, for every $q\in\im(\zeta)$, there exists
$f\in\clst_{\N(\LM,\H)}(S)_0\cup\atoms(\LM)$ such that $\zeta(f)=q$. 
If, in addition, $f\in\H^\circ$ and $p\neq f$ such that $\zeta(p)=q$, then
$p<f$ and $p\notin\clst_{\N(\LM,\H)}(S)_0$.
\end{proposition}
\begin{proof}
The set $T:=\{g\in F^+(y)\colon p\leq g\}$ is nonempty because $\onehat\in
F^+(y)$. Since $\min T$ is a $\G$-nested antichain with  $\wedge(\min T)\geq
p>\bottom$, Proposition \ref{prop:factors1} implies that $T$ has a unique
minimum. The assertion that $z(\hat{p})<p\vee
z(\hat{p})$ amounts to showing $p\not\leq z(\hat{p})$.  If instead $p\leq
z(\hat{p})= \bigvee (F^+(y))_{<\hat{p}}$, then $p\leq f$ for some $f\in
(F^+(y))_{<\hat{p}}$ by Proposition~\ref{prop:join decomp}, contradicting the
minimality of $\hat{p}$.
Thus, the map $\zeta$ is well-defined. 

Now let $q\in\im(\zeta)$. If $\zeta^{-1}(q)\subseteq\atoms(\LM)$, there is
nothing to show, so suppose $q=p\vee z_y(\hat{p})$ for some $p\in\H^\circ$.
Let $f\in F(\LM,\G;q)$ such that $p\leq f$ (guaranteed by Proposition
\ref{prop:join decomp}). It is clear that $f\in\H^\circ$ and $\zeta(f)=q$, and
we will show that $\{f\}\cup S$ is $\H$-nested.
Suppose that $A\subseteq S$ such that $\{f\}\cup A$ is an antichain with
$f\vee\bigvee A=h\in\H$. 
Since $S$ is $\H$-nested, we have $\bigvee A<h$. 
Note that for $a\in A\subseteq S$, $a\leq \hat{a}$ in $\L(\LM,\H)$, so
$\hat{a}\in\H$ implies $a=\hat{a}$. It follows that, for $a\in A$, if
$\hat{a}=\hat{h}$ then $a=\hat{a}=\hat{h}\geq h$, contradicting $a<h$.
So we must have $\hat{a}<\hat{h}$ for each $a\in A$, thus $\bigvee A\leq
z_y(\hat{h})$.
Then $z_y(\hat{h})<h\vee z_y(\hat{h}) = f\vee z_y(\hat{h})$, implying that
$\hat{f}=\hat{h}$ and $q=\zeta(h)$. By our choice of $f$, this means
that $h\leq f$, and thus $A=\emptyset$.

For our final claim, we still assume $\zeta^{-1}(q)\not\subseteq\atoms(\LM)$ and
let $f\in\H^\circ$ be as in the last paragraph. Suppose that $p\in\zeta^{-1}(q)$
such that $p\neq f$. Abbreviate $z=z_y(\hat{f})=z_y(\hat{p})$.
Then since $f$ is the only element of $F(\LM,\G;q)$ not
below $z$, we have $p<f$ and $f=p\vee(f\wedge z)$. 
But then $U=\{s\in S\colon \hat{s}<f\}$ has $\bigvee U = f\wedge z$, 
 which implies $\{p\}\cup U\subseteq\{p\}\cup S$ is not
nested, as desired.
\end{proof}

In fact, the above proposition implies that $\H$ has a 
partition with blocks indexed by $F^+(y)$, and we obtain a partial building set
for $\LM_{y,g}$ by letting $\H_{y,g}$ 
be the image of a block under the map $\zeta$:

\begin{definition}[Local partial building sets]\label{def:restricted H}
For $y\in\L(\LM,\H)$ and $g\in F^+(y)$, let
\[\H_{y,g}=\zeta_{y,\H}(\H)\cap (\LM_{y,g})_+\]
\end{definition}

The fact that $\H_{y,g}$ is indeed a partial building set will be proved in
Proposition \ref{prop:H restricts} below; first we provide an example.

\begin{example}\label{ex:local H}
Let $\LM$ be the lattice of set partitions of $[7]=\{1,2,\dots,7\}$, and
consider the building set $\G=\LM_{\irr}$ of irreducibles (partitions with one
nonsingleton block). Let us denote
$\onehat=1234567$, $p=123456|7$, and $q=123|4|5|6|7 \in \LM$, and consider the
partial building set $\H = \atoms(\LM)\cup\{\onehat,p\}$.

Let $y=(p,q)\in\L(\LM,\H)$, for which $S=\supp_{\N(\LM,\H)}(y)=\{12,13,23,p\}$.
Then $\clst_{\N(\LM,\H)}(S)_0=\{\onehat,p\}\cup\{ij\colon 1\leq i<j<7\}$
and $F^+(y) = \{q,p,\onehat\}$, and we obtain three intervals:
\begin{itemize}
\item $\LM_{y,q} = [\bottom,q]$, with $\H_{y,q} = \{12,13,23\}$ and $\G_{y,q} =
\{12,13,23,q\}$.
\item $\LM_{y,p} = [q,p]$, with $\H_{y,p} = \{p\}\cup\atoms(\LM_{y,p})$. Note
that the atoms include elements such as $1234|5|6|7$ (which is also in $\G$) and
$123|45|6|7 = q\vee 45$ (which is not itself in $\G$).
\item $\LM_{y,\onehat} = [p,\onehat]$ with $\H_{y,\onehat} = 
\G_{y,\onehat} = \{\onehat\}$.
\end{itemize}
\end{example}

\begin{lemma}\label{lem:G restricts}
Let $\G$ be a building set for a geometric lattice $\LM$, and let
$y\in\L(\LM,\G)$. For each $g\in F^+(y)$, the set $\G_{y,g}$ is a 
building set for $\LM_{y,g}$.
\end{lemma}
\begin{proof}
Fix $g\in F^+(y)$, and let $x\in(z(g),g]$.  Write
$F(\LM,\G;x)=\set{h_1,\ldots,h_k}$, so that
\begin{equation}\label{eq:joindecompx}
[\bottom, x]\cong\prod_{i=1}^k[\bottom, h_i].
\end{equation}
Write $z(g)=\bigvee_{i=1}^k f_i$ for some unique $f_i\in[\bottom,h_i]$, via the
join decomposition in \eqref{eq:joindecompx}. Let $T=\{i\colon f_i\neq h_i\}$,
so that $h_i\vee z(g)=\zeta(h_i)\in\G_{y,g}$ for each $i\in T$.
We claim that $\{h_i\vee z(g)\colon i\in T\}$ is the set of maximal elements of
$\G_{y,g}$ which lie below $x$. Indeed, if $p\vee z(g)\in\G_{y,g}$ with $p\in\G$
and $z(g)<p\vee z(g)\leq x = \bigvee_{i=1}^k h_i$, then $p\leq h_i$ for some
unique $i\in T$ by Proposition \ref{prop:join decomp}, and hence $p\vee
z(g)\leq h_i\vee z(g)$. It remains to see that these index factors in a join
decomposition of $[z(g),x]$, as follows:
\[ [z(g),x] \cong \prod_{i=1}^k [f_i,h_i] \cong \prod_{i=1}^k [z(g),z(g)\vee
h_i]\cong\prod_{i\in T} [z(g),z(g)\vee h_i].\]
Therefore, $\G_{y,g}$ is a building set for $\LM_{y,g}=[z(g),g]$.
\end{proof}

\begin{proposition}\label{prop:H restricts}
Let $\H\subseteq\G$ be a partial building set for a geometric lattice $\LM$, and
let $y\in\L(\LM,\H)$. For each $g\in F^+(y)$, the set $\H_{y,g}$ is a partial
building set for $\LM_{y,g}$.
\end{proposition}
\begin{proof}
Let us abbreviate $F(y) := F(\L(\LM,\H),\G; y)$. 
Viewing this set of factors $F(y)\subseteq\G$ as a set of atoms in $\L(\LM,\G)$,
let $y' = \bigvee F(y)\in\L(\LM,\G)$. Then $F^+(y) = F^+(\L(\LM,\G), \G; y')$, and
for every $g\in F^+(y)$, $\G_{y',g}$ is a building set in $\LM_{y,g} =
\LM_{y',g}$ by Lemma \ref{lem:G restricts}. 
Our claim is that, for every $g\in F^+(y)$, $\H_{y,g}^\circ$ is an order
filter in $\G_{y',g}^\circ$.

For this, let $g\in F^+(y)$ and
$p,q\in \zeta_{y',\G}^{-1}(\LM_{y',g})\cap \G^\circ$.
Assume that $\zeta_{y',\G}(p)\leq \zeta_{y',\G}(q)$. That is, we assume
$\hat{p}=g=\hat{q}$ and $p\vee z(g)\leq q\vee z(g)$. Then $p\leq q\vee z(g)$
and, by Proposition \ref{prop:join decomp}, this implies that either $p\leq q$
or $p\leq f$ for some $f\in(F^+(y))_{<g}$. By minimality of $\hat{p}$ in
Proposition \ref{prop:restrict H} and $g=\hat{p}$, the latter case cannot
happen, thus  $p\leq q$. Since $\H^\circ$ is an order filter of $\G^\circ$, this
means that if $p\in\H^\circ$ then $q\in\H^\circ$. In particular, this means that
if $\zeta_{y',\G}(p)=\zeta_{y,\H}(p)\in\H^\circ_{y,g}$, then
$\zeta_{y',\G}(q)=\zeta_{y,\H}(q)\in\H^\circ_{y,g}$.
\end{proof}

Since $\H_{y,g}$ is a partial building set for $\LM_{y,g}$, 
we may also consider the complex of $\H_{y,g}$-nested sets,
$\N(\LM_{y,g},\H_{y,g})$. 
Just as before, this is isomorphic to the atomic complex 
$\Ka(\L(\LM_{y,g},\H_{y,g}))$.
The next statement relates these local building sets to $\H$-nested sets.

\begin{proposition}\label{prop:local N(H)}
Let $y\in\L(\LM,\H)$ and $S=\supp_{\N(\LM,\H)}(y)$, and recall
$\zeta=\zeta_{y,\H}$ from Proposition \ref{prop:restrict H}.
Given $T\subseteq\clst_{\N(\LM,\H)}(S)_0$, 
the set $T\cup S$ is $\H$-nested if and only if, for every $g\in F^+(y)$, the
set $\zeta(T)\cap\H_{y,g}$ is $\H_{y,g}$-nested.
\end{proposition}
\begin{proof}
Suppose that $g\in F^+(y)$ and $\zeta(T)\subseteq\H_{y,g}$ is a nontrivial
antichain with $\bigvee\zeta(T)\in\H_{y,g}$, and we argue that $T\cup S$ cannot
be $\H$-nested. Let $h\in\H^\circ$ such that $\zeta(h)=\bigvee\zeta(T)$. Then 
$(\bigvee T)\vee z_y(g) = \bigvee_{p\in T} p\vee z_y(g) = h\vee z_y(g) = h\vee
z'$ where $z'=\bigvee\{f\in F(y)_{<g}\colon f\not<h\}$.
Let $Z=\{s\in S\colon s\leq z_y(g)\}$, noting that $\bigvee Z = z_y(g)$, and 
let $A=\max(T\cup Z)\subseteq T\cup S$. Since
$p\not\leq z_y(g)$ for any $p\in T$, we have $T\subseteq A$ and hence $A$ is a
nontrivial antichain. We further note that $(\bigvee A)\vee z'=\bigvee(T\cup
Z)\vee z' = (\bigvee T)\vee z_y(g)=h\vee z'$ and
$\bigvee A\leq h$. Since $h\wedge z'=\bottom$, it follows that $\bigvee
A=h\in\H$, thus $T\cup S$ is not $\H$-nested. 

Conversely, suppose that $U\subseteq S$ such that $T\cup U$ is a nontrivial
antichain with $\bigvee(T\cup U)=h\in\H$, and we argue that for $g=\hat{h}$,
the set $\zeta(T)\cap\H_{y,g}$ is not $\H_{y,g}$-nested. 
Since $\{p\}\cup S\in\N(\LM,\H)$ for
every $p\in T$, the set $T$ is necessarily a nontrivial antichain.
For $p\in T\cup U$, we have $p<h$ and hence either $\hat{p}=\hat{h}$ or $p\leq
z_y(g)$. In the case $p\in U$, then $p\leq z_y(g)$ (as in the proof of
Proposition \ref{prop:restrict H}, since $U\subseteq S$).
 Now let $A=\{p\in T\colon \hat{p}=\hat{h}\}$.
If $A=\emptyset$, then $p\leq z_y(g)$ for all $p\in T\cup U$, leading to a
contradiction between $h=\bigvee(T\cup U)\leq z_y(g)$ and $z_y(g)<h\vee z_y(g)$.
If $A=\{p\}$, then $h\vee z_y(g)=p\vee z_y(g)$ while $p<h$, contradicting
Proposition \ref{prop:restrict H} since $p\in\clst(S)_0$.
It follows that $\zeta(A)\subseteq\zeta(T)\cap\H_{y,g}$ is a nontrivial
antichain with $\bigvee\zeta(A)=\zeta(h)\in\H_{y,g}$, completing the
proof.
\end{proof}

Finally, we examine how blowups affect the local building sets.

\begin{proposition}\label{prop:blowup local H}
Let $\H$ and 
$\H' = \H\cup\{p\}$ be partial building sets in a geometric lattice $\LM$,
let $y\in\L(\LM,\H')$, and let 
$\hat{p} = \min\{f\in F^+(y)\colon p\leq f\}$. 
Then for any $g\in F^+(\pi^{\H'}_\H(y))$, 
\[ \H^\circ_{\pi^{\H'}_\H (y), g} = 
\begin{cases}
(\H'_{y,g})^\circ & \text{ if } g\neq \hat{p}\\
(\H'_{y,\hat{p}})^\circ\setminus\{p\vee z(\hat{p})\} & \text{ if } g = \hat{p}.
\end{cases}\]
\end{proposition}
Note that $\hat{p}$ exists by Proposition \ref{prop:restrict H}, 
and $\H'_{y,g}$ is well-defined by Proposition \ref{prop:blowdown F}.
\begin{proof}
We abbreviate $\pi=\pi_\H^{\H'}$.
If $\pi(y)\not\geq p$, then $F^+(y)=F^+(\pi(y))$ (Proposition \ref{prop:blowdown
F}) and the statement follows from the simple fact that $\H'=\H\cup\{p\}$. So
assume $\pi(y)\geq p$, that is, $y=(p,x)$ for some $x\in\L_{(p)}$. In this case,
$F^+(\pi(y)) = F^+(x)_{\not<p}\cup\{p\}$ (Proposition \ref{prop:blowdown F})
and $\hat{p}=p$. Then
$\H^\circ_{\pi(y),p} = \emptyset$ and $(\H'_{y,p})^\circ=\{p\}$, while
$\H^\circ_{\pi(y),g}=(\H'_{y,g})^\circ$ for $g\in F^+(x)_{\not<p}$.
\end{proof}

\begin{example}
Let $\LM$ be the partition lattice for $[7]$, $\H=\atoms(\LM)\cup\{\onehat\}$,
$p=123456|7$, $q=123|4|5|6|7$, and $y=(p,q)\in\L(\LM,\H\cup\{p\})$.
Recall from Example \ref{ex:local H} that $F^+(y)=\{q,p,\onehat\}$. Now note that
$\pi(y)=p$ and $F^+(\pi(y))=\{p,\onehat\}$, so that
$\LM_{\pi(y),p}=[\bottom,p]$ with $\H^\circ_{\pi(y),p}=\emptyset$ while
$\LM_{y,p}=[q,p]$ with $p\in\H'_{y,p}$.

For a different flavor, consider the same $\LM$, $\H$, $p$, and $q$, but now
$y=q\vee 456|7 = 123|456|7$ with $\pi(y)=y$. Then
$F^+(y)=F^+(\pi(y))=\{123,456,\onehat\}$. The difference in local building sets
in passing from $\H$ to $\H'=\H\cup\{p\}$ is in the interval $\LM_{y,\onehat}$,
where $\H^\circ_{\pi(y),\onehat}=\emptyset$ while $\H'_{y,\onehat}\ni p$.
\end{example}


\section{A combinatorial model for open neighborhoods}
\label{sec:OS}
Let $\A$ be a collection of affine hyperplanes in $\C^\ell$, 
and let $M(\A):=\C^{\ell+1}-\bigcup_{H\in \A}H$,  its complement.
The Orlik--Solomon algebra of $\A$ is a combinatorial presentation 
of the (integral) cohomology ring $H^\cdot(M(\A),\Z)$ that appears in the
literature with various levels of generality.  It depends only on
the intersection poset $\L(\A)$, so we will denote it by $\OS(\L(\A))$.
If $\A$ is a central hyperplane arrangement, then $\L(\A)$ is a geometric
lattice, and $\OS(\L(\A))$ is defined by generators and relations
from the underlying matroid.
If $\A$ is an affine-linear arrangement, $\L(\A)$ is a geometric semilattice,
and the presentation of $\OS(\L(\A))$ acquires monomial relations from 
subsets of atoms with no upper bound: we refer to Yuzvinsky~\cite{Yu01} 
and Kawahara~\cite{Ka04} for details.  

Here, then, we define a yet more general Orlik--Solomon algebra, for any
locally geometric semilattice $\L$, which we denote $\OS(\L)$.

\subsection{The Orlik--Solomon algebra of a semilattice}\label{ss:defOS}
For a semilattice $\L$, let $E(\L)$ denote the exterior algebra over $\k$
on the generators 
$\set{e_g\colon g\in \atoms(\L)}$.  Define a derivation 
$\partial$ on $E(\L)$ of degree $-1$ by 
letting $\partial(e_i)=1$ for each $i$, then extending via the Leibniz
rule.  

If $\LM$ is a geometric lattice, define an ideal of $E(\LM)$ by 
\begin{equation}\label{eq:defOS}
I(\LM)=(\partial(e_C)\colon \text{$C$ is a circuit in $\M(\LM)$}),
\end{equation}
where $e_J:=e_{j_1}\cdots e_{j_k}$ for any subset
$J=\set{j_1,\ldots,j_k}$ with entries in increasing order.
More generally, though, suppose that $\L$ is a locally geometric semilattice.
For each $x\in \L$, the interval $[\bottom,x]$ is a geometric
lattice, and we let
\begin{align}\label{eq:OSrelations}
I_1(\L)&=\sum_{x\in\L}I([\bottom,x])\\
&=\big(\partial(e_C)\colon\text{$C$ is a circuit in $\M(\L_{\leq x})$ for some 
$x\in \L$}\big). \nonumber
\end{align}
Now let
\begin{equation}\label{eq:extSRrelns}
I_2(\L)=(e_J\colon \text{$J\subseteq\atoms(\L)$ but $J\not\in\Ka(\L)$).}
\end{equation}
\begin{definition}[The algebra $\OS$]\label{def:OS}
The {\em Orlik--Solomon algebra} of a locally geometric semilattice $\L$ is, by 
definition,
\[
\OS(\L):=E(\L)/(I_1(\L)+I_2(\L)).
\]
\end{definition}

\begin{definition}[The projective Orlik--Solomon algebra]\label{def:proj OS}
If $\LM$ is a geometric lattice, we let $\POS(\LM)$ denote the subalgebra of 
$\OS(\LM)$ generated in degree $1$ 
by $\ker\partial=\left<e_i-e_j\colon 1\leq i<j\leq n\right>$.
It is known (see \cite{Ka04}) that $\OS(\LM)\cong \POS(\LM)\oplus
\POS(\LM)[-1]$; this gives rise to a (non-canonical) isomorphism of algebras
$\OS(\LM)\cong \POS(\LM)\otimes_{\k}\k[e_0]$, where the latter is a 
$1$-dimensional exterior algebra.
\end{definition}

\begin{remark}
\label{rmk:proj OS}
The motivation comes from the realizable case.  If $\A$ is a complex 
hyperplane arrangement with intersection lattice $\LM$, the algebras
$\OS(\LM)$ and $\POS(\LM)$ are, respectively, the cohomology algebras
of the affine and projective arrangement complements, 
$M(\A)=\C^{\ell+1}\setminus\cup_{H\in\A}H$ and 
$U(\A)=\P^{\ell}\setminus\cup_{H\in\A}\P H$.  The isomorphism
$\OS(\LM)\to H^\cdot(M(\A),\Q)$ is realized on the level of forms by
$e_j\mapsto \frac{1}{2\pi i}\frac{\dd f_j}{f_j}$, where $f_j$
is a linear form for which $H_j=f_j^{-1}(0)$: see \cite{OTbook} for details.
The differential $\partial\colon \OS^\cdot(\LM)\to\OS^{\cdot+1}(\LM)$
is realized by contraction along the Euler vector field, and
$H^\cdot(U(\A),\Q)$ is the kernel.
\end{remark}

\begin{example}\label{ex:newOS2}
If $\L$ is obtained from a geometric lattice
$\LM$ by blowing up a building set in a compatible order, then $\Ka(\L)=
\N(\LM,\G)$ (Proposition \ref{prop:N=K}).  Since
the order intervals in the face poset of a simplicial complex
are Boolean, their matroids have no circuits, and $I_1(\L)=0$.  
This is the combinatorial
abstraction of the normal crossings property of the boundary divisor in 
the wonderful compactification.  We find that
$\OS(\L)=E(\L)/I_2(\L)$, the exterior face ring of $\N(\LM,\G)$.
\end{example}
Now we show that the following well-known properties of Orlik--Solomon algebras 
extend to our slightly more general context.
First, an additive basis for $\OS(\L)$ when $\L$ is geometric 
goes back to a result of Bj\"orner~\cite{Bj82}.
Recall that an independent set $J$ in a matroid 
$\M$ with a totally ordered ground set $E$ is a {\em broken
circuit} if there exists some $g<\min J$ for which $\set{g}\cup J$ is
a circuit.  
The collection of all subsets of $E$ which do not contain a broken circuit,
denoted $\nbc(\M)$, is a pure simplicial complex on $E$ of dimension
one less than the rank of $\M$.  

\begin{definition}[The no-broken-circuit complex]\label{def:general nbc}
Let $\L$ be a locally 
geometric semilattice and $\prec$ a total order on $\atoms(\L)$.
We say that a simplex $S\in\Ka(\L)$ is a {\em broken circuit}
if there exists some $g\prec \min S$ for which $\set{g}\cup S$ is a circuit
in $\M(\L_{\leq x})$, for some $x\geq g\vee\bigvee S$.

Let $\nbc_{\prec}(\L)$ denote the set of simplices 
in $\Ka(\L)$ which do not contain a broken circuit.
Clearly $\nbc_{\prec}(\L)$ is a subcomplex of $\Ka(\L)$.  We will simply
write $\nbc(\L)$ if the order on which it depends is understood.
\end{definition}

An order of the atoms $\atoms(\L)$ gives an order of the variables of
$E(\L)$.  We extend it lexicographically to a term order on $E(\L)$.
In particular, if $\H$ is a partial building set for a geometric lattice $\LM$,
a compatible order on $\H$ gives a lexicographic term order for $E(\L(\LM,\H))$.

\begin{example}[Example~\ref{ex:delA3}, continued]
The $\nbc$ complexes 
for the three semilattices in Example~\ref{ex:delA3} are shown 
in Figure~\ref{fig:delA3}: they are the cones at the vertex $\onehat$ of
the respective $1$-complexes shown in bold.
\end{example}

This leads to a monomial basis for the Orlik--Solomon algebra, given in the next
theorem. This theorem and the corollaries we state here are well-known for
geometric semilattices, but the arguments are ``local'' in nature and hence hold
without change for locally geometric semilattices.

\begin{theorem}[Basis of $\OS(\L)$, {\cite[Thm.~2.8]{Yu01}}]\label{thm:general nbc}
Let $\L$ be a locally geometric semilattice with a fixed linear order on
$\atoms(\L)$. The generators of the ideal $I_1(\L)+I_2(\L)$ form a Gr\"obner
basis with respect to the graded lexicographic order.  The
corresponding additive basis for $\OS^i(\L)$ consists of monomials
\[
\set{e_J\colon J\in \nbc(\L)\text{~and~}\abs{J}=i}.
\]
\end{theorem}

In the case of hyperplane arrangements, the following is due to 
Brieskorn~\cite{Br73}.
\begin{corollary}[The Brieskorn decomposition]\label{cor:Brieskorn}
Let $\L$ be a locally geometric semilattice, and let $i\geq 0$. Then
\[
\OS^i(\L)=\bigoplus_{x\in \L_i} \OS^i(\L_{\leq x}),
\]
where $\L_i$ denotes the set of elements in $\L$ of rank $i$.
\end{corollary}
\begin{proof}
Broken circuits have a local definition, so $\nbc(\L_{\leq x})$ is the
full subcomplex of $\nbc(\L)$ on the vertices $\atoms(\L_{\leq x})$, 
for each $x\in \L$.
The decomposition follows by comparing the additive bases
provided by Theorem~\ref{thm:general nbc}.
\end{proof}

\begin{corollary}\label{cor:subarrangements split}
If $x,y\in\L$ and $x\leq y$, the inclusion $\atoms(\L_{\leq x})\subseteq
\atoms(\L_{\leq y})$ induces an injective algebra homomorphism
$\OS(\L_{\leq x})\to \OS(\L_{\leq y})$.
\end{corollary}

\begin{remark}\label{rem:partial}
The derivations $\partial$ on $\OS(\L_{\leq x})$, for each $x$, are compatible
with inclusions $\OS(\L_{\leq x})\hookrightarrow \OS(\L_{\leq y})$ for 
$x\leq y$, but not with their left-inverses 
$\OS(\L_{\leq y})\twoheadrightarrow \OS(\L_{\leq x})$.
\end{remark}

Given its local nature, it is natural to consider a sheaf of
Orlik--Solomon algebras.
\begin{definition}[The sheaf $\OSS$]\label{def:OS sheaf}
Let $\L$ be a locally geometric semilattice, and define
a graded sheaf of vector spaces $\OSS^\cdot(\L)$ on $\L^{\op}$ as follows.
For each $i\geq 0$, define a sheaf of algebras
$\OSS^i(\L)$ for $x\in \L$ by
\begin{equation*}
\OSS^i(\L)(x)=\OS^i(\L_{\leq x}),
\end{equation*}
with restriction maps for $x\leq y$ given by $\OSS^i(\L)(x)\hookrightarrow
\OSS^i(\L)(y)$.
\end{definition}
By the remark above, the map $\partial$ extends to sheaves to make a 
chain complex
\begin{equation}\label{eq:OS complex}
\begin{tikzcd}
\OSS^0 & \OSS^1\ar[l,"\partial",swap]  \cdots
& \OSS^i\ar[l,"\partial",swap]  \cdots
& \OSS^r\ar[l,"\partial",swap].
\end{tikzcd}
\end{equation}
We note that the surjections $\OS(\L_{\leq y})\to \OS(\L_{\leq x})$
also allow us to define a graded sheaf on $\L$, a special case of which 
plays a key role in \cite{Yuz95}.  With this structure, though, $\partial$
is not a map of sheaves (Remark~\ref{rem:partial}).

\begin{proposition}\label{prop:OS exact}
The complex \eqref{eq:OS complex} is exact except in degree $0$.
\end{proposition}
\begin{proof}
We check the claim on stalks.  The stalk at $x$
is the complex $(\OS^\cdot(\L_{\leq x}),\partial)$
geometric lattice of rank $\geq1$.  For $x\neq \bottom$,
this is exact by
\cite[Lem.\ 3.13]{OTbook}.  For $x=\bottom$, the complex is concentrated 
in degree $0$.
\end{proof}
\subsection{The flag complex}\label{ss:flag complex}
Classically, the graded $\k$-dual of the Orlik--Solomon algebra is 
identified with a vector space spanned by flags in the intersection lattice.
The flags index explicit homology cycles in a hyperplane arrangement complement.
As usual, the combinatorics of the flag complex extends beyond its topological
origins, as we show here.
The following construction appears first for
hyperplane arrangements in \cite[\S 2]{SV91}.

\begin{definition}[The flag complex]\label{def:flag complex}
Let $\L$ be a locally geometric semilattice.  For each $i\geq0$, let
$\widetilde{\Fl}^i(\L)$ denote the $\k$-span of all chains 
$Y:=(y_0<y_1<\cdots<y_i)$, where
$y_j\in \L_j$ for $0\leq j\leq i$.  
Let $\widetilde{\Fl}^i(\L)^\vee$ be the $\k$-dual, and denote the 
dual basis vectors by $Y^\vee$.

If $Y\in\widetilde{\Fl}^i(\L)$ is a chain as above, $0<j<i$,
and $y\in \L_j$ satisfies $y_{j-1}<y<y_{j+1}$, let
$Y(\widehat{y_j}; y):=(y_0<y_1<\cdots<y_{j-1}<y<y_{j+1}<\cdots <y_i)$.
Define $\Fl^i(\L)$ to be the quotient of $\widetilde{\Fl}^i(\L)$ by the sums
$\sum_y Y(\widehat{y_j}; y)$,
for each chain $Y$ of length $i+1$ and each $0<j<i$.
\end{definition}

An ordered, independent set of atoms $(h_1,\ldots,h_i)$ of $\L$ defines 
a flag $Y(h_1,\ldots,h_i)$ 
by letting $y_j=h_1\vee \cdots \vee h_j$, for $0\leq j\leq i$, provided
the join exists.
This extends to a map $\tilde{\fl}\colon E(\L)\to \tilde{\Fl}(\L)^\vee$ by setting
\begin{equation}\label{eq:fl map}
\tilde{\fl}(e_S)=\sum_{\sigma\in \Sigma_i}(-1)^{\abs{\sigma}}Y(g_{\sigma(1)},
\ldots,g_{\sigma(i)})^\vee,
\end{equation}
where $e_S$ is the monomial indexed by an independent set $S=\set{g_1,\ldots,
g_i}$, provided $S\in\Ka(\L)$, and zero otherwise. 
Here, $\Sigma_i$ denotes the symmetric group on $\set{1,\ldots,i}$.

We will make use of the dual version of Brieskorn's decomposition.  
This appeared in the arrangement case as \cite[(2.1.2)]{SV91} and follows
directly from Corollary \ref{cor:Brieskorn}.
\begin{lemma}\label{lem:flag brieskorn}
If $\L$ is a locally geometric semilattice, for each $i\geq0$, 
\begin{equation}\label{eq:flag brieskorn}
\Fl^i(\L)\cong \bigoplus_{x\in \L_i}\Fl^i(\L_{\leq x}).
\end{equation}
\end{lemma}

\begin{lemma}\label{lem:flag complex}
The map $\tilde{\fl}$ induces a map $\fl\colon\OS(\L)\to \tilde{\Fl}(\L)^\vee$
whose corestriction to $\Fl(\L)^\vee$ is an isomorphism.
\end{lemma}
\begin{proof}
This was noted for hyperplane arrangements in \cite[\S2]{SV91}; however,
the argument there applies without change for any geometric lattice.  That is,
$\fl\colon \OS(\L_{\leq x})\to \Fl(\L_{\leq x})^\vee$ is an isomorphism
for each $x\in\L$.  The global result then
follows using the decompositions of Corollary~\ref{cor:Brieskorn} and
Lemma~\ref{lem:flag brieskorn}.
\end{proof}

The graded vector space $\Fl(\L)$ is a cochain complex with respect to
a differential
$\delta$ defined for chains $Y=(y_0<\cdots<y_i)$ by the formula
\begin{equation}\label{eq:def delta}
\delta(Y)=(-1)^i\sum_{y_{i+1}\in \L_{i+1}}(y_0<y_1<\cdots<y_{i+1}). 
\end{equation}

The next result was stated for arrangements as \cite[Thm.~2.4(b)]{SV91};
again, the proof is combinatorial and applies any geometric lattice.
\begin{proposition}\label{prop:flag complex}
For any $x\in \L$, 
the map $\fl\colon (\OS(\L_{\leq x}), \partial)\to (\Fl(\L_{\leq x}),
\delta)^\vee$ is an isomorphism of chain complexes.
\end{proposition}

Just as with the Orlik--Solomon algebra, the flag complex is a local
construction. Thus, we consider the following sheaf on $\L$:
\begin{definition}[Sheaf $\FS$]\label{def:flag sheaf}
Let $\L$ be a locally geometric semilattice, and define a graded sheaf
$\FS^\cdot$ on $\L$ by  $\FS^\cdot(x)=\Fl^\cdot(\L_{\leq x})$ for $x\in\L$,
using restriction maps dual to those of $\OSS$,
$\rho_{x_1,x_0}\colon \Fl^j(\L_{\leq x_1})\to \Fl^j(\L_{\leq x_0})$ for all 
$x_0\leq x_1$ and $j\geq0$.  We note that the maps $\rho_{x_1,x_0}$ are 
coordinate projections, with respect to the direct sum decomposition 
\eqref{eq:flag brieskorn}.
\end{definition}
\begin{proposition}\label{prop:flag sheaf complex}
The differential $\delta$ from 
\eqref{eq:def delta} makes $\FS^\cdot$ a cochain 
complex of sheaves.  The augmented complex 
\begin{equation}\label{eq:flag sheaf}
\begin{tikzcd}
0\ar[r] & \KS\ar[r] & \FS^0\ar[r,"\delta"] & \cdots\ar[r,"\delta"] & \FS^i
\ar[r,"\delta"] & \cdots\ar[r,"\delta"] & \FS^{r}\ar[r] & 0
\end{tikzcd}
\end{equation}
is flasque and 
exact, where $\KS(x):=\Q$ for $x=\bottom$ and $\KS(x):=0$ otherwise.
\end{proposition}
\begin{proof}
The fact that $\delta$ is compatible with restrictions is dual to 
Remark~\ref{rem:partial}.  Likewise, the exactness of \eqref{eq:flag sheaf}
is obtained from Proposition~\ref{prop:OS exact} by taking $\Q$-duals.

We check each $\FS^i$ is flasque directly.  Suppose 
$a\in \FS^i(U)\subseteq\bigoplus_{y\in U}\Fl^i(\L_{\leq y})$
is a section over a downward-closed set $U\subseteq\L$, and suppose
$x$ is a minimal element of $\L-U$.  Then $x$ covers some elements
$x_1,\ldots,x_k\in U$.  We show we can extend $a$ to $x$ by
considering two cases.  If $i\geq \rank(x)$, then $i>\rank(x_j)$ for each
$1\leq j\leq k$, so $\FS^i(x_j)=0$ for each $j$, and we let $a_x=0$.
Otherwise, we write each
\[
\FS^i(x_j)=\bigoplus_{\substack{y\leq x_j\colon \\ y\in \L_i}} \Fl^i(\L_{\leq y})
\]
using Lemma \ref{lem:flag brieskorn}\eqref{eq:flag brieskorn}.  Since it
is a section,  $a\in \FS^i(U)$ has the 
property that $(a_{x_j})_y=(a_{x_{j'}})_y$ whenever $y\leq x_j$ and
$y\leq x_{j'}$.  So for every $y\in\L_i$ with $y\leq x$, 
let $(a_x)_y=(a_{x_j})_y$ if $y\leq x_j\leq x$, and $0$ otherwise, which
determines an element $a_x\in \Fl^i(\L_{\leq x})$ that restricts to each
$a_{x_j}$.
\end{proof}

\begin{corollary}\label{cor:Fl-global}
As graded vector spaces, $\Fl(\L)\cong \Gamma(\FS(\L))$.
\end{corollary}
\begin{proof}
We fix $i\geq0$ and use Brieskorn's Lemma~\ref{lem:flag brieskorn} to 
construct a map $\theta\colon \Fl(\L)\to \Gamma(\FS(\L))$ using maps
$\theta_x\colon \Fl(\L)\to \FS(\L)(x)$ for each $x\in \L$: we let
\[
\theta_x\colon \Fl^i(\L)\cong\bigoplus_{y\in\L_i} \Fl^i(\L_{\leq y})
\to \Fl^i(\L_{\leq x})\cong\bigoplus_{
\substack{y\in\L_i\\
y\leq x}} \Fl^i(\L_{\leq y})
\]
be the obvious coordinate projection.  Clearly for all $x\geq z$, we have
$\rho_{x,z}\circ\theta_x=\theta_z$, so this induces a map $\theta\colon
\Fl(\L)\to \Gamma(\FS(\L))$.  The inverse of $\theta$ is given
for each $x\in\L$ by a corresponding coordinate inclusion.

\end{proof}

\subsection{Blowups and the Orlik--Solomon algebra}\label{ss:blowup OS}
In this section, we examine the effect of a single blowup on Orlik--Solomon 
algebras of semilattices.  

\begin{theorem}\label{thm:OS_phi}
Suppose $\H$ and $\H'=\H\cup\set{p}$ are 
partial building sets associated to a building set $\G$ in a geometric lattice
$\LM$, and consider the blown-up semilattices $\L=\L(\LM,\H)$ and
$\L'=\Bl_p(\L)=\L(\LM,\H')$.   

The homomorphism of exterior algebras
$\phi_E\colon E(\L)\to E(\L')$, defined by letting
\begin{equation}\label{eq:OS_phi}
\phi_E(e_g) = \begin{cases} e_g & \text{ if } g\not\leq p\\ e_g+e_p & \text{ if
} g\leq p \end{cases} 
\end{equation}
for each $g\in\atoms(\L)$,
induces an injective  homomorphism $\phi_{\OS}\colon \OS(\L)\to \OS(\L')$.
\end{theorem}

We break up the proof of Theorem \ref{thm:OS_phi} into two pieces: Lemma
\ref{lem:OS_phi} (which shows that $\phi_{\OS}$ is well-defined) and Lemma
\ref{lem:OS injective} (which shows that $\phi_{\OS}$ is injective).
We will often write $\phi$ in place of both $\phi_E$ and $\phi_{\OS}$.
We start with an elementary observation.
\begin{lemma}\label{lem:OSreln}
If $C=\set{g_1,\ldots,g_k}$ and $g_1\prec \cdots \prec g_k$, then
\begin{equation}\label{eq:OSreln}
\partial(e_C)=(e_{g_2}-e_{g_1})\cdots(e_{g_i}-e_{g_{i-1}})\cdots (e_{g_k}-
e_{g_{k-1}}).
\end{equation}
\end{lemma}
This leads to the first main lemma:
\begin{lemma}\label{lem:OS_phi}
The map $\phi_E$ induces a well-defined homomorphism 
$\phi_{\OS}:\OS(\L)\to\OS(\L')$.
\end{lemma}
\begin{proof}
We show that $\phi$
sends relations to relations by considering the two cases.

If $\partial(e_C)\in I_1(\L)$ for a circuit $C$, let 
$x=\bigvee C\in\L$.  The closure of a circuit is irreducible, so $x\not\geq p$
by Proposition~\ref{prop:BICO}\eqref{prop:BICO1}.  
If $x<p$, then $g<p$ for each $g\in C$.  Using the expression in Lemma
\ref{lem:OSreln}\eqref{eq:OSreln}, we find $\phi(\partial(e_C))=
\partial(e_C)
\in I_1(\L')$, since $C$ remains a circuit in $\L'$.  If $x=p$, again
$\phi(\partial(e_C))=\partial(e_C)$, which is a signed sum of monomials 
$e_{C-{g_i}}$ for $1\leq i\leq k$.  Since, for each $i$,
$\bigvee(C-\set{g_i})=p$ in  $\L$, the set $C-\set{g_i}$ has no upper bound in
$\L'$.  It follows that each term $e_{C-{g_i}}\in I_2(\L')$.

Last, if $x$ and $p$ are incomparable, reorder $C$ to assume $g_i\leq p$
for $1\leq i\leq r$, and $g_i\not\leq p$ for $r+1\leq i\leq k$.  We may
assume $r\geq1$; if not, again $\phi(\partial(e_C))=\partial(e_C)\in 
I_1(\L')$.  Then by Lemma~\ref{lem:OSreln}, 
\begin{align*}
\partial(e_C)= & (e_{g_2}-e_{g_1})\cdots(e_{g_{r+1}}-e_{g_{r}})\cdots (e_{g_k}-
e_{g_{k-1}}),\text{~so}\\
\phi(\partial(e_C)) = & \partial(e_C)+
(e_{g_2}-e_{g_1})\cdots(e_{g_r}-e_{g_{r-1}})e_p(e_{g_{r+2}}-e_{g_{r+1}})\cdots
(e_{g_k}-e_{g_{k-1}}).
\end{align*}
Since $x\in\L'$, $C$ remains a circuit and $\partial(e_C)\in I_1(\L')$.
Each monomial in the right summand is indexed by a set 
$\set{p}\cup (C-\set{g,h})$, where $g\leq p$ and $h\not\leq p$ in $\L$.  Since
$C$ is a circuit, in $\L$ we have 
\begin{align*}
p\vee\bigvee(C-\set{g,h}) &= p\vee \bigvee(C-\set{h})\\
&=p \vee x,
\end{align*}
if these upper bounds exist.  But $p\wedge x\neq\bottom$, so by 
Proposition~\ref{prop:BICO}\eqref{prop:BICO2}, $p\vee x\not\in\L$.
So each $\set{p}\cup(C-\set{g,h})\not\in
\Ka(\L')$, and the remaining monomials are in $I_2(\L')$.

Thus, for circuits $C$ of $\L$ we have $\partial(e_C)\in I_1(\L')+I_2(\L')$.

If $e_J\in I_2(\L)$, monomials in $\phi(e_J)$ are indexed by the sets $J$ and
$J_g:=J\cup\set{p}-\set{g}$ for each $g\in J$ with $g\leq p$.  If 
$J_g\in \Ka(\L')$, then $\Pi(J_g)\in\Ka(\L)$, by Lemma~\ref{lem:Pi:a}.
But $J\subseteq \Pi(J_g)$, so $J\in\Ka(\L)$, a contradiction.
\end{proof}

We use the deg-lex monomial order on $\OS(\L)$, with the order on generators
induced by the order on $\G$. 
For an element $f$ in an Orlik--Solomon algebra, let $\In(f)$ denote its initial
monomial when written in terms of the monomial basis from Theorem
\ref{thm:general nbc}.

\begin{lemma}\label{lem:OS lead of phi}
For a monomial $e_J$ where $J\in\nbc(\L)$, we have
\[
\In(\phi(e_J)) = \begin{cases}
e_J & \text{ if } \bigvee J_{<p}\neq p\\
e_{J-\set{g}}e_p & \text{ if }\bigvee J_{<p}= p
\end{cases}
\]
where $g=\min_\prec J_{<p}$ and, as usual, $J_{<p}:=\{h\in J: h<p\}$.
\end{lemma}
\begin{proof}
If $J\in\nbc(\L)$, then 
\[
\phi(e_J)=e_J+e_p\sum_{h\in J_{\leq p}}\pm e_{J-\set{h}}.
\]
Since $e_p\prec e_h$ for each $h\leq p$,  the leading term in $\phi(e_J)$
is $e_J$, provided that $J\in \Ka(\L')$.  
Since $J\in\Ka(\L')$ if and only if $\bigvee J_{<p}\neq p$, we are done with
the first case.

Suppose $J\not\in\Ka(\L')$, then, which equivalently means $\bigvee J_{<p}=p$.
Let $J_h:=J\cup\set{p}-\set{h}$ for each $h\in J_{<p}$.  
Since $J$ is independent, $\bigvee(J-\{h\})_{<p} \lneq \bigvee J_{<p} = p$ and
hence $J_h\in\Ka(\L')$ for each $h\in J_{<p}$. This implies that the lead term
of $\phi(e_J)$ is $\pm e_{J_g}$, where $g=\min_\prec J_{<p}$.
\end{proof}

\begin{lemma}\label{lem:OS injective}
The homomorphism $\phi_{\OS}:\OS(\L)\to \OS(\L')$ is injective.
\end{lemma}
\begin{proof}
Finally, if $J\in\nbc(\L)$ with $\bigvee J_{<p} =p$, and $g=\min_{\prec}J_{\leq
p}$, then  the $\nbc$ property implies $g=\min_\prec \atoms(\L_{\leq p})$ as
well.
It follows that for no two $\nbc$ sets $J$ are the
lead terms of $\phi(e_J)$ the same.  We conclude that $\phi$ is injective.
\end{proof}


\section{A combinatorial model for the closed strata}
\label{sec:DP}
At this point we recall a second algebra which is also given by a 
geometric lattice and partial building set.  
We will show that, like the Orlik--Solomon algebras in \S\ref{sec:OS}, its local
versions form a sheaf, and we study how the algebra behaves under a blowup.

In the case where the geometric lattice is the intersection lattice of a complex
hyperplane arrangement, this is the cohomology ring of the De Concini--Procesi 
compactification from \cite{DP95}.  The form of the presentation here 
follows \cite{FY04} in the case where $\H=\G$, a full building set.  The
cohomology of a partial blowup is understood in the same way by the work
of Dupont~\cite{du15}.  Beyond the realizable setting, it is also the Chow ring 
of a smooth toric variety associated with a subfan of the Bergman fan, 
an observation that 
has its origins with Feichtner and Yuzvinsky~\cite{FY04}.  It is not clear 
that the Chow ring interpretation is important to us here, but we mention 
it because of the central role it plays for $\H\subset\G$ and $\G=\LM_+$ 
in the recent paper of Adiprasito, Huh and Katz~\cite{AHK15}.

\subsection{The De Concini--Procesi algebra for a partial building set}
\label{ss:DP}

We begin with an algebra presentation.
Suppose that $\H$ is a partial building set for a geometric lattice $\LM$.
For each $g\in \H$, we define an element
\begin{equation}\label{eq:cg}
c^{\H}_g = \sum_{\substack{ h\in \H\colon\\ g\leq h}} x_h\in \Q[x_g\colon
g\in\H],,  
\end{equation}
which we will abbreviate by $c_g$ when the choice of partial building set is
clear.

\begin{definition}[The algebra $\Dp$]\label{def:FYalg}
For a partial building set $\H\subseteq\G\subseteq\LM_+$, we define 
\[\Dp(\LM,\H):=\Q[x_g\colon g\in \H]/J(\H),\] where the ideal 
$J(\H)$ is generated by:
\begin{enumerate}[(i)]
\item $x_T$ 
whenever $T\not\in \N(\LM,\H)$;
\label{eq:FYalg:i}
\item $c^{\H}_g$ for each $g\in\atoms(\LM)$.
\label{eq:FYalg:ii}
\end{enumerate}
We will write $\Dp(\H)$ in place of $\Dp(\LM,\H)$ when the geometric 
lattice $\LM$ is understood.
The algebra $\Dp(\H)$ is graded by assigning degree $2$ to each variable $x_g$.
\end{definition}

For every element of the semilattice $\L(\LM,\H)$, we define a quotient 
of $\Dp(\LM,\H)$ as follows.
\begin{definition}\label{def:DPy}
Let $y\in\L(\LM,\H)$, and suppose $y=\bigvee S$ for a nested set
$S\in\N(\LM,\H)$. Define
\[\Dp_y(\LM,\H):=\Q[x_g\colon g\in \H]/J_y(\H),\] where the ideal 
$J_y(\H)$ is generated by:
\begin{enumerate}[(i)]
\item $x_T$ 
whenever $S\cup T\not\in \N(\LM,\H)$;
\label{eq:DPy:i}
\item $c^{\H}_g$ for each $g\in\atoms(\LM)$.
\label{eq:DPy:ii}
\end{enumerate}
\end{definition}
We remark that the definition is independent of the choice of $S$, since
if $\bigvee S=\bigvee S'$, then $S\cup T\in \N(\LM,\H)$ if and only if
$S'\cup T\in \N(\LM,\H)$ (in view of Theorem~\ref{thm:nested}).
We also note that the monomials $x_T$ in \eqref{eq:DPy:i} are just the
Stanley-Reisner relations for $\clst_{\N(\LM,\H)}(S)$.

Just as for Orlik--Solomon algebras (in Definition \ref{def:OS sheaf}), we
obtain a sheaf, but this time on $\L(\LM,\H)^{\op}$.  We note that, if $y\leq z$
in $\L(\LM,\H)$, then $J_y(\H)\subseteq
J_z(\H)$, so there is an obvious surjection $\Dp_y(\LM,\H)\to\Dp_z(\LM,\H)$.
\begin{definition}[The sheaf $\DP$]\label{def:Dsheaf}
Let $\H$ be a partial building set in a geometric lattice $\LM$, and let
$\L=\L(\LM,\H)$. Define a sheaf of graded rings 
$\DP(\LM,\H)$ on $\L^{\op}$ by setting, for each $y\in\L^{\op}$,
$\DP(\LM,\H)(y)=\Dp^\cdot_y(\LM,\H)$, with restrictions
given by the obvious maps.
\end{definition}

\begin{remark}\label{rem:forget atoms}
In Definition \ref{def:FYalg}, we may
use the relations \eqref{eq:FYalg:ii} to eliminate the 
variables $x_h$ for $h\in \atoms(\LM)$.  Then $\Dp(\LM,\H)\cong 
\Q[x_g\colon g\in \H^\circ]/J(\H^\circ)$, where the ideal 
$J(\H^\circ)$ is generated by 
(non-monomial) elements of type \eqref{eq:FYalg:i}.  
\end{remark}

\begin{remark}\label{rmk:D-basis}
Later, in
Corollary \ref{cor:D-basis}, we will see that the set of monomials $x_T^b$ for
which $T\in\N(\LM,\H)$, and $0<b(g)<d\left(\bigvee_{f\in T_{<g}} f,g\right)$ for
$g\in T$, is a monomial basis for $\Dp(\LM,\H)$. This is an easy generalization
of the monomial basis for full building sets given in \cite{FY04}.
In fact, we have everything we need to prove this now, but we defer the proof
until we are in a more general setting in \S\ref{sec:B}.
We will, however, use this monomial basis in the proof of Theorem \ref{thm:D(X)}
next.
\end{remark}

The algebras $\Dp_y(\LM,\H)$ can be 
further decomposed as a tensor product of the algebras derived from intervals
in the geometric lattice $\LM$.  
In the geometric setting, this corresponds to the fact that the closed strata in
the wonderful compactification are themselves products of wonderful
compactifications.
We will be interested in $y\in \L$ 
whose blow-down $\pi(y)\neq\onehat$, and we refer back to the definitions and
notation introduced in  \S\ref{ss:H decomp}, namely
Definitions \ref{def:factors plus} and \ref{def:restricted H}.
We need one more lemma before proceeding to the tensor decomposition in Theorem
\ref{thm:D(X)}.

\begin{lemma}\label{lem:Dpy relations}
Let $y\in\L(\LM,\H)$, and recall $\zeta=\zeta_{y,\H}$ from Proposition
\ref{prop:restrict H}. Suppose that $a,b\in\H$ such that $\zeta(a)=\zeta(b)$.
\begin{enumerate}
\item\label{eq:ca=cb}
If $a\in\atoms(\LM)$, then $c_a^\H=c_b^\H$ in $\Dp_y(\LM,\H)$.
\item\label{eq:xa=xb}
If $a,b\in\atoms(\LM)$, then $x_a=x_b$ in $\Dp_y(\LM,\H)$.
\end{enumerate}
\end{lemma}
\begin{proof}
We first claim that for
$p\in\clst_{\N(\LM,\H)}(\supp(y))_0$, if $p>a$ then $p\geq b$.
By our hypotheses, $b\leq b\vee z_y(\hat{b}) = a\vee z_y(\hat{b}) \leq p\vee
z_y(\hat{b})$. 
Similar to the proof of Proposition \ref{prop:restrict H}, we have
$F(\LM,\G;p\vee z_y(\hat{b})) = \{p\}\cup F(\LM,\G;z_y(\hat{b}))_{\not\leq p}$,
and so minimality of $\hat{b}$ implies $b\leq p$.

It follows that $c_a^\H=c_b^\H$ in $\Dp_y(\LM,\H)$. Moreover, if
$a,b\in\atoms(\LM)$, then $c_a^\H-x_a = c_b^\H-x_b$ and hence
$x_a=x_b$ in $\Dp_y(\LM,\H)$.
\end{proof}

\begin{theorem}[Local tensor decompositions]\label{thm:D(X)}
Let $\H$ be a partial building set for a geometric lattice $\LM$, and let
$y\in\L(\LM,\H)$ such that $\pi(y)\neq\onehat$ in $\LM$.  As graded
algebras,
\[
\Dp^\cdot_{y}(\LM,\H)\cong \bigotimes_{g\in F^+(y)}\Dp^\cdot(\LM_{y,g},\H_{y,g})
\cong \bigotimes_{g\in F^+(y)\cap\H^\circ}\Dp^\cdot(\LM_{y,g},\H_{y,g}).
\]
\end{theorem}
\begin{proof}
First, note that the two tensor decompositions are equivalent, since for $g\in
F^+(y)\setminus\H^\circ$ we have $\H_{y,g}^\circ=\emptyset$ and hence
$\Dp^\cdot(\LM_{y,g},\H_{y,g})\cong\Q$.

We will define a map $\psi_g:\Dp(\LM_{y,g},\H_{y,g})\to\Dp_y(\LM,\H)$ for each
$g\in F^+(y)$. Then we will prove that $\psi:=\bigotimes\psi_g$ is
an isomorphism.

Write $S=\supp(y)$ and recall $\zeta=\zeta_{y,\H}$ from Proposition
\ref{prop:restrict H}.
Fix a section \[\sigma:\bigsqcup_{g\in F^+(y)}\H_{y,g}\to \H\] of $\zeta$ such that
$\sigma(q)\in\max\zeta^{-1}(q)$.
Then for $g\in F^+(y)$ and $q\in\H_{y,g}$, define 
\[\psi_g(x_q) = \begin{cases}
x_{\sigma(q)} & \text{if } q\neq g\\
\sum_{h\geq g} x_h & \text{if } q=g.
\end{cases}\]

To see that $\psi_g$ does not depend on the choice of $\sigma$, consider
$q\in\H_{y,g}$. By Proposition \ref{prop:restrict H}, the choice of $\sigma(q)$
is unique unless $\zeta^{-1}(q) \subseteq\atoms(\LM)$. In this case,
the fact that $\psi_g(x_q)$ does not depend on choice follows from Lemma
\ref{lem:Dpy relations}\eqref{eq:xa=xb}. 
Proposition \ref{prop:local N(H)} implies that $\psi_g$ is well-defined on the
monomial relations. For the linear relations, let $q\in\atoms(\LM_{y,g})$ and
pick $a\in\atoms(\LM)$ such that $\zeta(a)=q$. Then by Lemma \ref{lem:Dpy
relations}\eqref{eq:ca=cb},
$\psi_g(c_q^{\H_{y,g}}) = c^\H_{\sigma(q)} = c^\H_a$ in $\Dp_y(\LM,\H)$.
Therefore, $\psi_g$ is well-defined.

Now, to show that $\psi=\otimes\psi_g$ is an isomorphism, we first check
surjectivity.
Here, it suffices to show that $x_h\in\im(\psi)$ for all
$h\in\H^\circ\cap\clst_{\N(\LM,\H)}(S)_0$.
If $h\in\H^\circ\setminus F^+(y)$ then $\psi(x_{\zeta(h)})=x_h$. 
Note that $\psi(x_{\onehat})=x_{\onehat}$, and for the remaining
$g\in\H^\circ\cap F^+(y)$, we proceed inductively.
If $x_h\in\im(\psi)$ for
each $h>g$, then $x_g=\psi(x_g)-\sum_{h>g} x_h\in\im(\psi)$.

It remains to show that $\psi$ is injective, which we do next by analyzing
basic monomials.
Suppose that, for each $g\in F^+(y)$, $x_{T_g}^{b_g}$ is a basic
monomial in
$\Dp(\LM_{y,g},\H_{y,g})$. This means, by Corollary \ref{cor:D-basis}
(see also Remark \ref{rmk:D-basis}), that
$T_g\subseteq \H^\circ_{y,g}$ and $0<b_g(p)<d(\bigvee(T_g)_{<p},p)$ for
each $p\in T_g$. Let $T=\cup_{g}\sigma(T_g)\subseteq\H^\circ$
and $b(h)=b_{\hat{h}}(\zeta(h))$ for $h\in T$; we
prove that $x_T^{b}$ is a basic monomial in
$\Dp_{y}(\LM,\H)$. For this, we claim that 
for each $g\in F^+(y)$ and $h\in \sigma(T_g)$,
\begin{equation}\label{eq:tensor}
z\vee \bigvee_{\substack{
f\in S\cup T\colon \\
f<h}} f = z\vee \bigvee_{
\substack{
f\in \sigma(T_g)\colon \\
f<h}} f
\end{equation}
where $S=\supp_{\N(\LM,\H)}(y)$ and 
$z=z_y(g)$.
This then implies
\begin{align*}
 d(\bigvee_{\substack{f\in S\cup T\colon\\ f<h}} f, h)
&\geq d(\bigvee_{\substack{f\in S\cup T\colon\\ f<h}} f\vee z, h\vee z)
\,=\, d(\bigvee_{\substack{f\in \sigma(T_g)\colon\\ f<h}} f\vee z, h\vee z)\\
&= d(\bigvee_{\substack{f\in \sigma(T_g)\colon\\ f<h}}
\zeta(f),\zeta(h))
\,=\, d(\bigvee_{\substack{f\in T_g\colon\\ \sigma(f)<h}} f,\zeta(h))\\
&= d(\bigvee_{\substack{f\in T_g\colon\\ f<\zeta(h)}} f,\zeta(h)),\\
\end{align*}
so that for every $h\in T$, $0<b(h)<d(\bigvee(S\cup T)_{<h},h)$.
Since $S\cup T$ is $\H$-nested by Proposition \ref{prop:local N(H)}, the monomial
$x_T^{b}$ is basic in $\Dp_y(\LM,\H)$.
Since $\sigma$ is injective, the assignment $\otimes x_{T_g}^{b_g}\mapsto x_T^b$
is an injection on basic monomials.

Now to show \eqref{eq:tensor}, we first observe that each join factor on the
right is clearly also a factor on the left. In the other direction, there
are two cases to consider.  If $f\in S$, then $f<h\leq g$, so $f\leq z$.  
Otherwise, $\zeta(f)\in T_{g'}$ for some $g'\in S$. This means, in particular,
that $g'=\hat{f}$, the minimum element of $F^+(y)$ above $f$ (Proposition
\ref{prop:restrict H}). Since $f<h\leq g$, this implies $g'\leq g$.
If $g'=g$, then $f$ is also a factor on the right; otherwise, $g'<g$ implies
that $f\leq g'\leq z$.
\end{proof}

\subsection{Blowups and the De Concini--Procesi algebra}\label{ss:blowup DP}
The effect of a blowup on the algebras $\Dp(\LM,\H)$
is analogous to what happens
for Orlik--Solomon algebras.  
\begin{proposition}\label{prop:D_phi}
Suppose $\H$ and $\H'=\H\cup\{p\}$ are partial building sets associated to a
building set $\G$ in a geometric lattice $\LM$. Then
there is a well-defined ring homomorphism $\phi_{\Dp}\colon \Dp(\LM,\H)\to 
\Dp(\LM,\H')$ defined by letting
\begin{equation}\label{eq:D_phi}
\phi_{\Dp}(x_g)=\begin{cases}
x_g & \text{if $g\not\leq p$;}\\
x_g+x_p & \text{if $g\leq p$,}
\end{cases}
\end{equation}
for each $g\in \H$.
\end{proposition}
\begin{proof}
The argument for relations
of type \eqref{eq:FYalg:i} is the same as the monomial relations for the
Orlik--Solomon algebra in the proof of Lemma \ref{lem:OS_phi}. 
Now let $g\in\atoms(\LM)$ and consider $c_g^{\H}$ in \eqref{eq:cg}.
If $h\in\H$ and $h>g$, then $h\in\H^\circ$ and hence $h\not\leq p$ (since
$\H^\circ$ is an order filter in $\G$). This implies that $c_g^{\H}$ picks up an
$x_p$ exactly when $g\leq p$, thus $\phi_{\Dp}(c^{\H}_g)=c^{\H'}_g$. 
\end{proof}

\begin{remark}
By eliminating the generators $x_h$ for $h\in\atoms(\LM)$, as in Remark
\ref{rem:forget atoms}, the map $\phi_{\Dp}$ has the simple formula
$\phi_{\Dp}(x_g) = x_g$ for $g\in\H^\circ$. This is because $g\in\H^\circ$
implies that $g\not\leq p$ in $\L(\LM,\H)$.
\end{remark}

We will now show that
this map extends to sheaves, and then fit the sheaves in a short exact
sequence.
For ease of notation, we will write $\Dp$ in place of $\Dp(\LM,\H)$ and 
$\Dp'$ for $\Dp(\LM,\H')$ as well as $\L:=\L(\LM,\H)$ and $\L':=\L(\LM,\H')$.

\begin{lemma}\label{lem:local phi}
For each $y\in \L'$, the composite of $\phi_{\Dp}$ with the restriction
$\Dp'\to \Dp'_y$ factors through the restriction $\Dp\to \Dp_{\pi(y)}$:
\[
\begin{tikzcd}
\Dp\ar[r,"\phi_{\Dp}"]\ar[d] & \Dp'\ar[d] &\\
\Dp_{\pi(y)}\ar[r,dashed] & \Dp'_y. \\
\end{tikzcd}
\]
\end{lemma}
We let $\phi_{\Dp,y}$ denote the map $\Dp_{\pi(y)}\to \Dp'_y$.
\begin{proof}
We lift the generators of $\Dp_{\pi(y)}$ to $\Dp$ and check that the
relations map to zero in $\Dp'_y$.  We saw in Proposition~\ref{prop:D_phi}
that $\phi_{\Dp}(c_g^{\H})=c_g^{\H'}$ for all $g\in\atoms(\LM)$, 
so it remains to check the Stanley-Reisner relations.  For this, we express
$y=\bigvee S$ for a minimal nested set $S\in \N(\LM,\H')$.  Let
$\Pi=\Pi_{\H}^{\H'}$.  Then
$\pi(y)=\bigvee\Pi(S)$ (see, e.g., \eqref{eq:onestep Pi}), so we show
simplices of $\st_{\N(\LM,\H')}(S)$ map to simplices of
$\st_{\N(\LM,\H)}(\Pi(S))$.  So suppose $T\in\st_{\N(\LM,\H')}(S)$.  Then 
\begin{align*}
S\cup T\in \N(\LM,\H') & \Rightarrow \Pi(S\cup T)\in \N(\LM,\H) \quad
\text{by Lemma~\ref{lem:Pi:a}};\\
& \Leftrightarrow \Pi(S)\cup\Pi(T)\in \N(\LM,\H) \\
& \Leftrightarrow \Pi(T)\in \st_{\N(\LM,\H)}(\Pi(S)),
\end{align*}
as required.
\end{proof}
\begin{lemma}\label{lem:D sheaf}
The map $\phi_{\Dp}$ induces a map of sheaves $\pi^*\DP(\L)\to \DP(\L')$
on $(\L')^{\op}$.
\end{lemma}
\begin{proof}
The claim amounts to showing that, for any $x\leq y$ in $\L'$, the square
\[
\begin{tikzcd}
\Dp_{\pi(x)}\ar[r,"\phi_{\Dp,x}"]\ar[d] & \Dp'_x\ar[d] &\\
\Dp_{\pi(y)}\ar[r,"\phi_{\Dp,y}"] & \Dp'_y. \\
\end{tikzcd}
\]
commutes.  For $x=\bottom$, this is Lemma~\ref{lem:local phi}.  In general,
it follows from Lemma~\ref{lem:local phi} because the map $\Dp\to\Dp_{\pi(x)}$
is surjective.
\end{proof}

Recall the notation $\L_{(p)}$ from Definition \ref{def:blowup}. 
We define a sheaf $\QS$ on $\L^{\op}_{(p)}$ as follows: for $y\in\L_{(p)}$,
$\QS(y)$ is the ideal generated by $x$ in the ring 
$\Q[x]/(x^d)$, where $d=d(z_y(\hat{p}),z_y(\hat{p})\vee p)$ and $\hat{p}$ is as
in Proposition \ref{prop:restrict H}.
We recall from Lemma \ref{lem:alpha}\eqref{eq:def_alpha}
the embedding $\alpha\colon \L_{(p)}^{\op}\times\set{0<1}\to
\L'^{\op}$.
\begin{theorem}\label{thm:sheaf keel}
If $\L'=\Bl_p(\L)$, there is a short exact sequence of sheaves on $(\L')^{\op}$
\[
\begin{tikzcd}
0\ar[r] & \pi^*\DP(\L)\ar[r,"\phi_{\Dp}"] & \DP(\L')\ar[r] & \alpha_!
(\QS\boxtimes \Q)\ar[r] & 0,
\end{tikzcd}
\]
where $\Q$ denotes the constant sheaf on the two-element poset $\set{0<1}$.
\end{theorem}
\begin{proof}
Lemma \ref{lem:D sheaf} says that $\phi_{\Dp}$ is a map of sheaves, and so we
check that it is injective and has the stated cokernel. Let us examine this on
the stalk at $y\in\L'$, $\phi_{\Dp,y}\colon\Dp_{\pi(y)}\to \Dp'_y$, 
using our tensor decomposition from Theorem \ref{thm:D(X)}:
\[
\Dp_{\pi(y)}\cong \bigotimes_{f\in F^+(\pi(y))} 
\Dp(\LM_{\pi(y),f},\H_{\pi(y),f}) 
\to \bigotimes_{f\in F^+(y)} \Dp(\LM_{y,f},\H'_{y,f})\cong \Dp'_y.
\]
If $y\notin \L'_{(p)}$, then by Proposition \ref{prop:blowup local
H}, the local building sets are unchanged by the blowup and hence $\phi_{\Dp,y}$
is an isomorphism with 
$\coker(\phi_{\Dp,y}) = 0 = (\alpha_!(\QS\boxtimes\Q))_y$.

Now suppose that $y\in\L'_{(p)}$, and  let $\hat{p}=\min\{g\in
F^+(y)\colon g\leq p\}$ as in Proposition \ref{prop:restrict H}.
By Proposition \ref{prop:blowup
local H}, the map $\phi_{\Dp,y}$ is an isomorphism on each tensor factor of
$\Dp_{\pi(y)}$ except the one indexed by $\hat{p}$. Furthermore, the tensor
factors of $\Dp_{y}$ indexed by $f\in F^+(x)_{<p}$ are simply
$\Dp(\LM_{y,f},\H'_{y,f})\cong\Q$ since $(\H'_{y,f})^\circ=\emptyset$.
So it suffices to understand the map
$\Dp(\LM_{\pi(y),\hat{p}},\H_{\pi(y),\hat{p}})\to\Dp(\LM_{y,\hat{p}},\H'_{y,\hat{p}})$.
When $y=(p,x)$ for some $x\in\L_{(p)}$, we have $\LM_{y,p}=[z_y(p), p]$ with
$\H_{y,p}^\circ=\emptyset$ and $\LM_{\pi(y),p}=[\bottom,p]$
with $(\H'_{\pi(y),p})^\circ = \{p\}$. In this case, the map in question is the
inclusion of $\Q$ into $\Q[x_p]/(x_p^d)$, where $d=d(z_y(p),z_y(p)\vee p)$, with
cokernel $(\alpha_!(\QS\otimes\Q))_y$ as desired. The case $y\in\L_{(p)}$ is
similar.
\end{proof}


\section{Combining the two models}
\label{sec:B}
In this section, we blend together the Orlik--Solomon and De Concini--Procesi
algebras into a \cdga. The ultimate goal is to show that this is a model for the
Orlik--Solomon algebra, agreeing with the Leray model when realizable.

Throughout this section, we continue to let $\LM$ be a geometric lattice 
containing a fixed, building set $\G$.  Suppose 
$\H\subseteq \G$ is a partial building set, and $\L(\LM,\H)$ is the 
semilattice obtained from $\LM$ by blowing up $\H^\circ$.

\subsection{Defining the algebra}
Let $R(\H)=\Q[e_g,x_g\colon g\in \H]$ be the graded-commutative algebra
with generators $e_g$ in bidegree $(0,1)$ and $x_g$ in bidegree $(2,0)$.  
The algebra $R(\H)$ is equipped with a differential $\dd$ of bidegree $(2,-1)$,
defined on generators by $\dd(e_g)=x_g$ and $\dd(x_g)=0$, giving it the 
structure of a \cdga.  
Fixing a (reverse) linear extension of the order on $\H$ as in Definition
\ref{def:partialblowups} gives an order among the
$e$ variables and among the $x$ variables; we also require $x_g\prec e_h$ for
 each $g,h$. 

$R(\H)$ has a monomial basis which we denote by
\[
e_Sx_T^b := e_{g_1}\dots e_{g_s} x_{h_1}^{b_1} \dots x_{h_t}^{b_t}
\]
where $S=\{g_1,\dots,g_s\}$ with $g_1\prec\dots\prec g_s$ and 
$T=\{h_1,\dots,h_t\}$ with $h_1\prec \dots \prec h_t$.
Recall the derivation $\partial$ from Definition~\ref{def:OS}:
\[
\partial e_S=\sum_{j=1}^k (-1)^{j-1} e_{g_1}\dots \hat{e}_{g_j}\dots e_{g_k}.
\]
Also recall the following defined in \eqref{eq:cg}:
\[
c_g = \sum_{\substack{h\in\H\colon\\g\leq h}} x_h.
\]

\begin{definition}[The Leray model of a matroid]
\label{def:B}
The algebra $\hB(\LM,\H)$ is defined as the quotient of $R(\H)$ 
by the ideal $I(\LM,\H)$
generated by:
\begin{enumerate}[(i)]
\item $e_Sx_T$ whenever $S\cup T\notin \N(\LM,\H)$,\label{def:B:i}
\item 
$\partial e_S$ whenever $S\in \N(\LM,\H)$ is a circuit, \label{def:B:ii}
\item $c_g$ for each $g\in \atoms(\LM)$.
\label{def:B:iii}
\end{enumerate}
If $\onehat\in\H$, we also define $B(\LM,\H):=\hB(\LM,\H)/(e_{\onehat})$.  The
algebras $\hB(\LM,\H)$ and $B(\LM,\H)$ 
inherit the grading and differential $\dd$ from $R(\H)$.
\end{definition}

Relation \eqref{def:B:i} is a Stanley-Reisner relation for our simplicial
complex, and relation  \eqref{def:B:ii} is an Orlik--Solomon relation associated
to $\L_{\leq \vee S}$.
Note that, if $\LM$ is realizable and $\G=\H$ is a full building set, 
this agrees with the presentation of the Morgan model given in 
\cite[\S 5]{DP95}.  It follows from Proposition~\ref{prop:onehat} that
the generator $e_{\onehat}$ does not appear in any of the relations
defining $I(\LM,\H)$; therefore 
$\hB(\LM,\H)\cong B(\LM,\H)\otimes \Q[e_{\onehat}]$.  The choice of 
whether or not to include $e_{\onehat}$ models the difference between
blowing up a central arrangement versus a projective arrangement, and
we will allow both possibilities: see Remark~\ref{rmk:proj OS}.

We note that the bigraded algebra $\hB(\LM,\H)$ contains the two algebras
of \S\ref{sec:OS} and \S\ref{sec:DP}.  
Because of this, much of the work done here is an extension of that in
\S\ref{sec:OS}, \S\ref{sec:DP}, and \cite{FY04}. Explicitly, we have:
\begin{proposition}\label{prop:subalgs}
Let $\H$ be a partial building set in a geometric lattice $\LM$.
\begin{enumerate}[(a)]
\item
The subalgebra of $\hB(\LM,\H)$ in degree $(0,-)$ is
$\OS(\L(\LM,\H))$.
\item
The subalgebra of $\hB(\LM,\H)$ in degree $(-,0)$ is $\Dp(\LM,\H)$.
\item
$(\hB(\LM,\H),\dd)$ is a chain complex of $\Dp(\LM,\H)$-modules.
\label{eq:subalgs3}
\end{enumerate}
\end{proposition}
\begin{proof}
The first statement is a matter of comparing
relations in $B(\LM,\H)$ with \eqref{eq:OSrelations}, \eqref{eq:extSRrelns}.
Similarly, the relations in the $x$ variables agree with 
Definition \ref{def:FYalg}, establishing the second statement.
The last follows by checking $\dd(x_gf)=x_g\cdot \dd(f)$ by the
Leibniz rule for each $x_g$, since $\dd(x_g)=0$, so $\dd$ is a 
$\Dp(\LM,\H)$-module homomorphism.
\end{proof}

It should be clear that the last two statements hold for $B(\LM,\H)$ as
well.  The analogue of the first statement is the following.

\begin{proposition}\label{prop:OSisB}
Suppose $\LM$ is a geometric lattice.  Let $\H=\atoms(\LM)$ and $\H'=\H\cup
\set{\onehat}$.  Then $\OS(\LM)\cong \hB(\LM,\H)$
via the map $e_g\mapsto e_g$
for each $g\in \atoms(\LM)$, and 
\[
\POS(\LM)\cong \ker d\colon B^{0,\cdot}(\LM,\H')\to\ B^{2,\cdot}(\LM,\H').
\]
\end{proposition}
\begin{proof}
For the first statement, we see from \eqref{eq:cg} 
that $c_g=x_g$, so the defining ideal $I(\LM,\H)$
is generated by Orlik--Solomon relations 
and the variables $x_g$ for each $g\in\H$.  For the second statement, 
the kernel of the differential can be seen to be generated by differences
$e_g-e_h$ for atoms $g,h$, and the result follows from 
Definition~\ref{def:proj OS}.
\end{proof}

We conclude this section by noting
 that the algebra $B(\LM,\H)$ has a fine grading
indexed by the semilattice $\L(\LM,\H)$ which is induced by Brieskorn's
decomposition of the Orlik--Solomon algebra (Corollary~\ref{cor:Brieskorn}).
In the geometric setting, the summands are indexed by strata and obtained
as tensor products of the cohomology of the divisor complement near the stratum
(an Orlik--Solomon algebra) and the cohomology of the stratum itself (a
De Concini--Procesi algebra): see \cite[\S5.2]{DP95} as well as \cite{du15,Bi16}.

\begin{proposition}\label{prop:decomp}
Let $\H\subseteq\LM_+$ be a partial building set, and let $\L=\L(\LM,\H)$.
We have a decomposition of $B(\LM,\H)$ as $\Dp(\LM,\H)$-modules: 
for all $i,j\geq 0$, 
\begin{equation}\label{eq:B-decomp}
\hB^{ij}(\LM,\H)\cong\bigoplus_{y\in\L_j}\OS^j(\L_{\leq y})\otimes_\Q 
\Dp^i_y(\LM,\H),
\end{equation}
where $\Dp^\cdot_y(\LM,\H)$ is defined in Definition~\ref{def:DPy}.
Similarly,
\begin{equation}\label{eq:oB-decomp}
B^{ij}(\LM,\H)\cong\bigoplus_{
\substack{y\in\L_j\colon\\ 
y\not\geq (\hat{1},\hat{0})}}
\OS^j(\L_{\leq y})\otimes_\Q \Dp^i_y(\LM,\H).
\end{equation}

\end{proposition}
\begin{proof}
Since the relations defining $\hB(\H)$ are homogeneous in both $e$ and $x$ 
variables, Brieskorn's Lemma~\ref{cor:Brieskorn} gives a decomposition
\[
\hB^{ij}(\LM,\H)\cong\bigoplus_{y\in\L_j}\bigoplus_{
\substack{S\in\nbc(\L)\colon\\
\bigvee S=y}}e_S \Dp^i_S,
\]
for some $\Dp(\LM,\H)$-module $\Dp^\cdot_S$.  We note that if $S$ and $S'$
are independent sets with $\bigvee S=\bigvee S'$, the relations of 
Definition~\ref{def:B} are unchanged by replacing $S$ by $S'$, so 
$\Dp^\cdot_S=\Dp^\cdot_{S'}$.  Using the homogeneity of the presentation,
we compare it with Definition~\ref{def:DPy} and find $\Dp^\cdot_S\cong
\Dp^\cdot_y(\LM,\H)$, where $y=\bigvee S$.
The version \eqref{eq:oB-decomp} is analogous.
\end{proof}

\subsection{An equivalent presentation}

Just as in \cite[Thm.~1]{FY04}, we must add to relation \eqref{def:B:iii} from
Definition \ref{def:B} in order to obtain a Gr\"obner basis for the ideal
$I(\LM,\H)$ and hence a monomial basis for our algebra $B(\LM,\H)$.

\begin{theorem}
\label{thm:B2}
The ideal $I(\LM,\H)$ in $R(\H)$, from 
Definition \ref{def:B}, is equal to the ideal generated by:
\begin{enumerate}
\item[\eqref{def:B:i}] $e_Sx_T$ whenever $S\cup T\notin \N(\LM,\H)$,
\item[\eqref{def:B:ii}] 
$\partial e_S$ whenever $S\in \N(\LM,\H)$ is a circuit, and
\setcounter{enumi}{2}
\item[\mylabel{def:B:iii'}{iii$^\prime$}]
$e_Sx_Tc_g^d$ whenever $S\cup T\in\N(\LM,\H)$ and $g\in\H$ for
which: $S\cap T=\emptyset$, $S\cup T$ is an antichain, 
$\bigvee(S\cup T)< g$, and $d=d(\bigvee(S\cup T),g)$.
\end{enumerate}
\end{theorem}
\begin{proof}
Let $I'$ denote the ideal generated by relations \eqref{def:B:i},
\eqref{def:B:ii}, \eqref{def:B:iii'}.
Every relation of type \eqref{def:B:iii} is also of type \eqref{def:B:iii'},
since for $g\in\atoms(\LM)$, we can take $S=T=\emptyset$ and have
$d(\bottom,g)=1$. To show $I'=I(\LM,\H)$, then it remains to show 
every relation of type \eqref{def:B:iii'} is in $I(\LM,\H)$, which we do by
induction on $d$. 

\underline{\textit{Case} $d=1$:}
Consider an element $e_Sx_Tc_g$ of type \eqref{def:B:iii'}, and let
$p:=\bigvee(S\cup T)$. Then $S\cup T$ is a nested set, 
$g\in\H$, $p< g$, and $d(p,g)=1$.
Since $d(p,g)=1$, we can pick $h\in \atoms(\LM)$ for which $p\vee h = g$.
We want to show that if $y\in \H$ has $y\geq h$ but $y\not\geq g$, 
then
$\set{y}\cup S\cup T\notin \N(\LM,\H)$, 
because this would imply that, modulo type \eqref{def:B:i}
relations, we have $e_Sx_Tc_g \equiv e_Sx_Tc_h\in I(\LM,\H)$. 

Accordingly, 
assume that $y\geq h$ and $\set{y}\cup S\cup T\in\N(\LM,\H)$, 
and we will show that
$y\geq g$.  By construction, $h\not\leq p$, so $y\not\leq z$ for any 
$z\in S\cup T$.
Let $Z\subseteq S\cup T$ consist of those elements of $S\cup T$
which are incomparable with $y$. 
If $Z\neq\emptyset$, then $Z\cup\set{y}$ is a
nontrivial antichain that does not contain $g$, and $g>z$ for all $z\in Z$.
Since
\[
\bigvee(Z\cup\set{y}) = \bigvee(S\cup T\cup\set{y}) = p\vee y \geq p\vee h = g,
\] 
Proposition \ref{prop:factors2}, shows that 
$Z\cup\set{y}\notin \N(\LM,\H)$. However,
we assumed $S\cup T\cup\{y\}\in \N(\LM,\H)$ and
$\N(\LM,\H)$ is a simplicial complex, which yields a contradiction.
Thus, $Z=\emptyset$, and we must have $y\geq z$ for all $z\in S\cup T$. 
This gives us our desired conclusion that $y\geq p\vee h = g$.

\underline{\textit{Case} $d>1$:}
Take $e_Sx_Tc_g^d$ of type \eqref{def:B:iii'} with $d>1$, and let
$p:=\bigvee(S\cup T)$ once again. Pick $h\in\atoms(\LM)$ so that $p<p\vee h<g$.
We want to show that if $y\geq h$ but $y\not\geq g$, we have
$e_Sx_Tx_yc_g^{d-1}\in I(\LM,\H)$: then, modulo $I(\LM,\H)$, we have
\[
e_Sx_Tc_g^d \equiv e_Sx_Tc_hc_g^{d-1},
\]
which is an element of $I(\LM,\H)$ since $c_h$ is of type \eqref{def:B:iii}.

Assume that $y\geq h$ but $y\not\geq g$. 
We may assume that $\set{y}\cup S\cup T\in
\N(\LM,\H)$, since otherwise we'd be done by using type \eqref{def:B:i} relations.
Once again, our choice of $h$ implies $y\not\leq z$ for any $z\in S\cup T$.
Consider the (possibly empty) subset 
$Z\subseteq S\cup T$ consisting of elements which are
incomparable to $y$. 
We will first show that $e_{S\cap Z}x_{T\cap Z}x_yc_{y\vee g}^{d-1}$ is in the
ideal $I(\LM,\H)$ using induction, and then show that it is equal to $e_{S\cap
Z}x_{T\cap Z}x_yc_g^{d-1}$ modulo $I(\LM,\H)$, which will conclude the proof.

Let $\tilde{y}=\bigvee(\{y\}\cup Z) = y\vee p$. Note that $\tilde{y}\vee g =
y\vee g$. Also, since $\bottom<h\leq y\wedge g$, we have $y\vee g\in\H$ by
Proposition \ref{prop:BICO}\eqref{prop:BICO2}. 
Moreover,
\[ 
d(\tilde{y},\tilde{y}\vee g) \leq d(\tilde{y}\wedge g,g) \leq d(p\vee h,g) < d
\]
where the first inequality follows by \cite[\S3(iv)]{FY04}, and the last
two by \cite[\S3(i)]{FY04}, together with the observation that
$\tilde{y}\wedge g\geq (y\wedge g)\vee (p\wedge g) \geq h\vee p$.
Thus, our first claim follows from the induction hypothesis.

Now we argue that if $q\geq g$ but $q\not\geq y\vee g$, we have $\set{q,y}
\notin\N(\LM,\H)$, so that relations of type \eqref{def:B:i} give us our last claim.  
So assume that
$q\geq g$ and $\{q,y\}\in\N(\LM,\H)$, and we will show that this implies $q\geq
y\vee g$. If $q\leq y$, then $g\leq y$, contradicting our choice of $y$. If $q$
and $y$ are incomparable, then by nestedness we have $q\vee y\notin \H$. But
since $h\leq g\leq q$ and $h\leq y$, we obtain
$\bottom<h\leq q\wedge y$, from which Proposition
\ref{prop:BICO}\eqref{prop:BICO2} yields a contradiction.
Thus, $q\geq y$, which implies that $q\geq y\vee g$.
\end{proof}

\begin{remark}\label{rmk:iii'}
Observe that the relation $e_Sx_Tc_g^d$ holds in $\Dp(\LM,\H)$ even without the
hypotheses $S\cap T=\emptyset$ and $S\cup T$ an antichain.  Indeed, if $S\cup
T\in\N(\LM,\H)$ and $g\in\H$ for which $\bigvee(S\cup T)<g$ and
$d=d(\bigvee(S\cup T),g)$, then $e_Sx_Tc_g^d$ is divisible by the type
\eqref{def:B:iii'} relation $e_{S'}x_{T'}c_g^d$ where $S'=(S\setminus
T)\cap\max(S\cup T)$ and $T'=T\cap\max(S\cup T)$.
\end{remark}

\subsection{A Gr\"obner basis}\label{ss:groebner}

The construction we describe next 
is modelled after the Gr\"obner basis \cite[Thm.~2]{FY04}.  In fact,
in the case where $\H=\G$, that Gr\"obner basis is the subset of ours
obtained by restricting to the $x$ variables.
On the other hand, by restricting to the $e$ variables, 
we recover the $\nbc$ basis for
$\OS(\L(\LM,\H))$ from Theorem~\ref{thm:general nbc}.

The corresponding additive basis for $B(\LM,\H)$ will play an essential
role in our proof that blowups induce injective quasi-isomorphisms of \cdga s
(Theorem~\ref{thm:injective} and Theorem~\ref{thm:model}.)

\begin{theorem}
\label{thm:grobner}
Recall that a fixed linear extension of the order on $\H$ induces an order
among the $e$ variables and among the $x$ variables, and we further require 
$x_g\prec e_h$ for each $g,h$.
With this, consider the deg-lex monomial order on $R(\H)$.

The relations \eqref{def:B:i}, \eqref{def:B:ii}, and \eqref{def:B:iii'}
from Theorem \ref{thm:B2} form a Gr\"obner basis for the ideal $I(\LM,\H)$ in
$R(\H)$.
\end{theorem}
\begin{proof}
This proof is not very enlightening: we follow the method used by Feichtner and
Yuzvinsky and explicitly compute syzygies. We have several cases, depending on
the different types of relations.

\underline{\textit{Case} \eqref{def:B:i}--\eqref{def:B:i}:}
Since type \eqref{def:B:i} relations are monomial, the syzygy for two
of these will be zero.

\underline{\textit{Case} \eqref{def:B:ii}--\eqref{def:B:ii}:}
The syzygy for two type \eqref{def:B:ii} relations is zero by 
Theorem \ref{thm:general nbc}. 

\underline{\textit{Case} \eqref{def:B:i}--\eqref{def:B:ii}:}
If we have $e_Rx_S$ of type \eqref{def:B:i} and $\partial e_T$ of type
\eqref{def:B:ii}, 
we note that $\bigvee (T-\set{g})=\bigvee T$ for each $g\in T$, since $T$
is a circuit.  Since $R\cup S$ and hence $R\cup S\cup T$ is not nested, 
then neither is $R\cup S\cup T-\set{g}$ for any $g\in T$.  
It follows that each monomial in the syzygy between 
$e_Rx_S$ and $\partial e_T$ is a relation of type \eqref{def:B:i}

\underline{\textit{Case} \eqref{def:B:i}--\eqref{def:B:iii'}:}
Now consider $e_Sx_T$ of type \eqref{def:B:i} and $e_Ax_Bc_g^d$ of type
\eqref{def:B:iii'}. 
Let $U=S\cup A$ and $V=B\cup T-\{g\}$, so that the syzygy is
\[
z = e_{U}x_{V}(x_g^d-c_g^d).
\]
If $g\notin T$, then the syzygy $z$ is divisible by the type \eqref{def:B:i}
relation $e_Sx_T$. So assume that $g\in T$. 
Since $S\cup T$ is not nested, then neither is $U\cup V\cup\{g\}$.
If $U\cup V$ were not nested, then the syzygy $z$ would be divisible by
the type \eqref{def:B:i} relation $e_{U}x_{V}$, so assume that $U\cup V$ is nested.

Modulo $e_{U}x_{V\cup\{g\}}$, then
\[z \equiv e_{U}x_{V}\Big(\sum_{f>g}x_f\Big)^d.\]
Since $U\cup V\cup\{g\}$ is not nested, it contains a nontrivial antichain $Y$
whose join $y=\bigvee Y$ is in $\H$. Moreover, $Y$ must contain $g$ since $U\cup
V$ is nested; let $y'=\bigvee(Y-\{g\})$. Since 
\[d=d(\bigvee_{\substack{f\in A\cup B\\f<g}} f,g)
\geq d(\bigvee_{\substack{f\in A\cup B\\f<g}} f \vee y',g\vee y')
\geq d(\bigvee_{\substack{f\in U\cup V\\f<y}} f, y),\]
we have that $e_Ux_{V}c_y^d$ is divisible by a type \eqref{def:B:iii'} relation.
We claim that modulo type \eqref{def:B:i} relations,
\[ z\equiv e_Ux_Vc_y^d,\]
 which, along with Remark \ref{rmk:iii'}, will finish the proof of this case.

To prove this last claim, we will show that if $f\in\H$ with $f>g$ and
$f\not\geq y$, then $U\cup V\cup\{f\}$ is not nested.
Suppose that $f>g$ and $f\leq y$. Then  $y'\vee f = y'\vee g = y\in\H$, 
which implies that $U\cup V\cup\{f\}$ is not nested.
Now suppose that $f>g$ such that $f$ and $y$ are incomparable.
Then $f\vee y'\geq g\vee y'=y$, and by Proposition \ref{prop:factors2}, this
means that $U\cup V\cup\{f\}$ is not nested.

\underline{\textit{Case} \eqref{def:B:ii}--\eqref{def:B:iii'}:}
Suppose that we have $\partial e_S$ of type \eqref{def:B:ii} and
$e_Ax_Bc_g^d$ of type \eqref{def:B:iii'}.  Let $h=\min_{\prec} S$, and
$S'=S-\set{h}$. Let $L:=e^{}_{S'\cup A}x^{}_B x_g^d$, the lead monomial in the
syzygy.  Cancelling $L$, we obtain
\begin{align*}
& e_{A-S'}x_Bx_g^d(\partial e_S-e_{S'})-e_{S'\cup A}x_B(c_g^d-x_g^d)  
\equiv \\
& \equiv \quad  e_{A-S'}x_Bx_g^d(\partial e_S-e_{S'})
+ e_{A-S'}(\partial e_S-e_{S'})x_B(c_g^d-x_g^d)\\
\intertext{by adding a multiple of $\partial e_S$ with initial term less than
$L$;}
& = e_{A-S'}(\partial e_S-e_{S'})x_B c_g^d,\\
& \equiv \sum_{k\in S'} \pm e^{}_{A\cup S-\set{k}} x_B c_g^d.
\end{align*}
If $k\notin A$, then $d(\bigvee(A\cup B\cup S-\{k\}),g)\leq d(\bigvee(A\cup
B),g)$ and hence this summand is divisible by a relation of type
\eqref{def:B:iii'}. Now assume $k\in A$, which implies that $k\leq g$. Since $S$
is a circuit, we have $v:=\bigvee S\in \G$. 
Then for any $f\in\H$ with $f\geq g$, we have $\bottom<k\leq f\wedge v$ and
hence $f\vee v\in\H$ by Proposition
\ref{prop:BICO}\eqref{prop:BICO2}. This means that if $f\not\geq v$, the
set $S\cup\{f\}\setminus\{k\}$ is not nested, and
type \eqref{def:B:ii} relations then reduce the above expression to
\[\equiv\sum_{k\in A\cap S'}\pm e_{A\cup S-\{k\}} x_B c_{g\vee v}^{d}.\]
Since $d(\bigvee(A\cup B)\vee v,g\vee v)\leq
d(\bigvee(A\cup B),g)$, this is divisible by a type \eqref{def:B:iii'} relation,
which completes the argument.

\underline{\textit{Case} \eqref{def:B:iii'}--\eqref{def:B:iii'}:}
Let $e_Sx_Tc_g^d$ and $e_Ax_Bc_h^f$ be two
\eqref{def:B:iii'} relations. We have different scenarios here:

First, if $g=h$ and $d\leq f$, then the syzygy is
\[
e_{S\cup A}x_{T\cup B}x_h^fc_g^d - e_{S\cup A}x_{T\cup B} x_g^dc_h^f = 
e_{S\cup A}x_{T\cup B}x_g^dc_g^d(x_g^{f-d}-c_g^{f-d}),
\]
which is divisible by the type \eqref{def:B:iii'} relation $e_Sx_Tc_g^d$.

Second, 
if $g\neq h$, $g\notin B$, and $h\notin T$, then assume (without loss of
generality) that $g\succ h$. The syzygy is then
\[
z := e_{S\cup A}x_{T\cup B}(x_h^fc_g^d - x_g^dc_h^f)
\]
Let $y = e_{S\cup A}x_{T\cup B}c_g^d(c_h^f-x_h^f)$, which is divisible by the
type \eqref{def:B:iii'} relation $e_Sx_Tc_g^d$ and satisfies $\In(y)\leq
\In(z)$. 
It suffices to check that $z+y$ reduces to zero, and it does since 
\[
z+y = e_{S\cup A}x_{T\cup B}(c_g^d-x_g^d)c_h^f
\]
is divisible by the type \eqref{def:B:iii'} relation $e_Ax_Bc_h^f$.

Finally, assume  $g\neq h$ and $g\in B$, 
and note that we must also have $g<h$ (so $h\prec
g$) and $h\notin T$. Let $U=S\cup A$ and $V=T\cup B\setminus\{g\}$. The syzygy
is then
\[
z := e_Ux_V(x_h^fc_g^d-x_g^dc_h^f).
\]
Let $y = e_Ux_Vc_g^d(c_h^f-x_h^f)$, which is divisible by the
type \eqref{def:B:iii'} relation $e_Sx_Tc_g^d$ and has a 
leading term smaller than or equal to that of $z$. 
It suffices to check that 
\[
z+y = e_Ux_Vc_h^f(c_g^d-x_g^d)
\]
reduces to zero. First, through division by the type \eqref{def:B:iii'}
relation $e_Ax_Bc_h^f$, since $g\in B$, we obtain
\[
z+y \equiv e_Ux_V\left(\sum_{k>g}x_k\right)^dc_h^f.
\]
Then, for $g<k<h$, since $d(\bigvee(U\cup V\cup\{k\}),h)\leq f$, we may divide
by $e_Ux_Vx_kc_h^f$ from which we obtain
\[ 
z+y \equiv e_Ux_V\left(\sum_{\substack{k>g\\k\not<h}} x_k\right)^dc_h^f
\]
Then, since $d(\bigvee(U\cup V),h)\leq d+f$, we may divide by $e_Ux_Vc_h^{d+f}$.
This leaves us with a sum of monomials, each of which is divisible by some
$e_Ux_Vx_kc_h^f$ where $k>g$ and $k$ is incomparable to $h$.
Thus, it remains to show that when $U\cup V\cup\{k\}\in \N(\LM,\H)$, 
$k>g$, and $k$ incomparable to $h$, we get $e_Ux_Vx_kc_h^f\equiv 0$.

For this, we claim that modulo type \eqref{def:B:i} relations, 
\[e_Ux_Vx_kc_h^f \equiv e_Ux_Vx_kc_{h\vee k}^f\]
and that the right hand side is divisible by a type \eqref{def:B:iii'} relation.
The latter claim follows since $h\wedge k\geq g$ implies $h\vee k\in\H$, and
also $d(\bigvee(U\cup V\cup\{k\}),h\vee k)\leq f$. For the first claim, which
will finish our proof, we show that if $p\geq h$ but $p\not\geq h\vee k$ then
$\{p,k\}$ is not nested so that $x_kx_p$ is a type \eqref{def:B:i} relation.
Now, if $p\geq h$ with $p\not\geq h\vee k$, then $p$ and $k$ are incomparable. 
In this case, we would also have $p\wedge k \geq g$ implying $p\vee k\in\H$ by
Proposition
\ref{prop:BICO}\eqref{prop:BICO2}. Therefore, $\{p,k\}\notin\N(\LM,\H)$.
\end{proof}

Since a monomial basis of the quotient $R(\H)/I(\LM,\H)$ 
is given by the monomials
which are not divisible by initial monomials of elements of the Gr\"obner basis,
and since $\hB(\LM,\H)\cong B(\LM,\H)\otimes\Q[e_{\onehat}]$,
we obtain the following corollary. 

\begin{corollary}
\label{cor:basis}
The algebra $\hB(\LM,\H)$ has an additive basis given by the monomials $e_Sx_T^b$ 
for which $S\cup T\in \N(\LM,\H)$, $S\in\nbc(\L(\LM,\H))$, and for each $g\in T$
we have $0<b(g)<d\left(\bigvee{(S\cup T)_{<g}} , g\right)$.
If $\onehat\in\H$, then the subset of monomials of this form which do not
contain $e_{\onehat}$ is an additive basis for the algebra $B(\LM,\H)$.
\end{corollary}

Recalling Remark~\ref{rem:forget atoms}, we note that, if $e_Sx_T^b$ is a 
monomial in the basis above, then $T\subseteq\H^\circ$.
In the special case where $S=\emptyset$, we obtain a straightforward
generalization of the additive basis of \cite{FY04}.
\begin{corollary}
\label{cor:D-basis}
The algebra $\Dp(\LM,\H)$ has an additive basis given by monomials 
$x_T^b$, indexed by $T\in \N(\LM,\H)$ for which, if $g\in T$, then
$0<b(g)<d(z(g), g)$.
\end{corollary}
Using this and the tensor decomposition of Theorem \ref{thm:D(X)}, one could
also obtain an explicit monomial basis for the local algebras $\Dp_y(\LM,\H)$.

\subsection{Poincar\'e duality} 
We conclude this section with a discussion of Poincar\'e duality in the
sheaf of algebras $\DP(\LM,\H)$.  First, the existence of Poincar\'e duality
for each of the algebras $\Dp_y(\LM,\H)$ is not surprising but, we feel,
requires a bit of justification.  Using the decomposition of 
Theorem~\ref{thm:D(X)}, it is enough to show that $\DP(\LM,\H)$ itself
possesses Poincar\'e duality.  In the case where $\G=\LM_+$ is the maximal
building set, Adiprasito, Huh and Katz accomplished this for 
any matroid and any partial building set~\cite[\S6]{AHK15}.
In \cite[\S 3]{Yuz97}, Yuzvinsky explicitly gave an
isomorphism $\Dp^{2p}(\LM,\G)\cong \Dp^{2(r-p)}(\LM,\G)$, for $0\leq p\leq r$.
Although he assumes that both $\G$ is a (full) building set and that
$\LM$ is the intersection lattice of a complex arrangement, it is 
straightforward to extend his approach to any geometric lattice and
partial building set $\H$.  We will do so here using our Gr\"obner basis
from \S\ref{ss:groebner}.

We will want to use Poincar\'e duality because the
$\Q$-dual of $(B(\LM,\H),\dd)$, as a cochain complex, has a technical
advantage: its differential is obtained from the differential on 
the flag complex of \S\ref{ss:flag complex} by extension of scalars.

As usual, let $\LM$ be a geometric lattice
of rank $r+1$, and $\H$ a partial building set for $\LM$.  
For a nested set $T\subseteq\H^\circ$, write $T^+:=T\cup\set{\onehat}$ and for
$g\in T^+$,  write
\[z_T(g) := \bigvee_{\substack{f\in T\colon\\f<g}} f.\]
Since $T\subseteq\H^\circ$, this is consistent with our notation from
\S\ref{ss:H decomp}: by considering $y=\bigvee T\in\L(\LM,\H)$ we have
$T^+=F^+(y)$ and $z_T(g)=z_y(g)$.
 Recall from Corollary \ref{cor:D-basis}
that a monomial basis for $\Dp(\LM,\H)$ is given by $x_T^b$, where
$T\in\N(\LM,\H)$ and $0<b(g)<d(z_T(g),g)$.
We define a $\Q$-linear map $\varepsilon$ on this monomial basis by letting
\[
\varepsilon(x_T^b)=(-1)^{\abs{T-\set{\onehat}}-r}
\prod_{g\in T^+} x_g^{d_T(g)-b(g)},
\]
where $d_T(g)=d(z_T(g),g)$ if $g\neq \onehat$, and $d_T(\onehat)=
d(z_T(\onehat),\onehat)-1$.
Up to sign, this is simply Yuzvinsky's basis involution.
\begin{lemma}\label{lem:duality}
Let $S\in\N(\LM,\H)$ be such that $S\subseteq\H^\circ$ and $\onehat\in S$.
Let $g\in S$ and suppose that $x_S^b$ is a monomial such that for all $h\in S$
with $h>g$, we have $b(h)=d_S(h)$.
\begin{enumerate}[(a)]
\item\label{lem:duality:i}
If $b(g)=d_S(g)$, then we have 
$x_S^b = (-1)^{|S-(S')^+|}x_{S'}^{b}x_{\onehat}^{d_{S'}(\onehat)}$ 
in $\Dp(\LM,\H)$, where $S'=\{h\in S\colon h\not\geq g\}.$
\item\label{lem:duality:ii}
If $b(g)>d_S(g)$, then we have $x_S^b=0$ in $\Dp(\LM,\H)$.
\end{enumerate}
\end{lemma}
\begin{proof}
We argue by lexicographic induction, the base case being $x_{\onehat}^b$,
which is zero if $b>r$.  
Now assume that $x_S^b$ and $g\in S$ are such that $b(h)=d_S(h)$ for all $h>g$,
and $b(g)\geq d_S(g)$, and assume that the statements are true for all earlier
such monomials.
First, we note that $x_{S-\set{g}}^bc_g^{b(g)}$ is a multiple of a 
type \eqref{def:B:iii'} relation.
Thus we can write $x_S^b=x^b_{S-\set{g}}x_g^{b(g)}$ 
as a sum of terms of the form $-x_{S-\set{g}}^bx_T^a$, where
$T\subseteq\set{h\in\H\colon h\geq g}$ and $T\neq\set{g}$, the set 
$(S-\set{g})\cup T$ is nested, and $\sum_{h\in T} a(h)=b(g)$. 
We will now examine these monomials more closely.

Write $S_{\geq g}=\{h\in S\colon h\geq g\}$ and $S'=S-S_{\geq g}$.
Consider one of the above monomials, which can be written in the form
$m=-x_{S'}^bx_{S_{>g}}^{d_S+a}x_{T-S_{>g}}^a$.
Note that since $T$ is nested and $\bottom<g\leq\bigvee T$, it cannot
contain a nontrivial antichain (by Proposition \ref{prop:factors1}).
Hence $T$ is a nonempty chain. 
Similarly, $S_{>g}$ is also a nonempty chain.
Let $h$ be the minimum element of $S_{>g}$, and let $\max T$ be the
maximum element of $T$ (both with respect to $\leq$). 
If $\max T\neq h$, then the monomial is zero by the inductive hypothesis
\eqref{lem:duality:ii} with the minimum element of $S_{\geq\max T}$ playing the
role of $g$. Thus, the monomial $m$ could only be nonzero if $\max T=h$.
Now suppose that $|T|>1$; i.e. there exists some $f\in T$ for which $f<h$.
Taking $f$ to the maximum such, the monomial will be zero by the inductive
hypothesis \eqref{lem:duality:ii} with $h$ playing the role of $g$.

Thus, the only possibly nonzero monomial in our expansion of $x_S^b$ is
$m=-x_{S'}^bx_{S_{>h}}^{d_S}x_h^{b(g)+d_S(h)}$, 
where $h$ is the minimum element of $S_{>g}$. 
Since $d_{S-\set{g}}(h)=d_S(g)+d_S(h)$,
it follows by induction that if $b(g)>d_S(g)$ then $x_S^b=0$, and
if $b(g)=d_S(g)$ then
\begin{align*}
x_S^b&=-x_{S'}^bx_{S_{>h}}^{d_S}x_h^{d_S(g)+d_S(h)}\\
&=-x_{S'}^bx_{S_{\geq h}}^{d_{S-\set{g}}}\\
&=-(-1)^{|S-\set{g}-(S')^+|}x_{S'}^bx_{\onehat}^{d_{S'}(\onehat)}\\
&=(-1)^{|S-(S')^+|}x_{S'}^bx_{\onehat}^{d_{S'}(\onehat)}
\end{align*}
\end{proof}

\begin{proposition}\label{prop:duality}
Let $\LM$ be a geometric lattice of rank $r+1$ and $\H$ a partial building set
for $\LM$. 
For each basic monomial $x_T^b$ in $\Dp(\LM,\H)$, 
\[
x_T^b \varepsilon(x_T^b)= (-1)^r x_{\onehat}^r.
\]
\end{proposition}
\begin{proof}
By definition of $\varepsilon$, we have
\[x_T^b\varepsilon(x_T^b) = (-1)^{\abs{T-\set{\onehat}}+r} 
\prod_{h\in T^+}x_h^{d_T(h)}.\]
The statement is clearly true when $T^+=\{\onehat\}$, and otherwise 
the result follows by applying Lemma \ref{lem:duality}\eqref{lem:duality:i} to
each minimal element of $T$.
\end{proof}
Accordingly, for homogeneous elements $u\in \Dp^{2i}(\LM,\H)$ and
$v\in \Dp^{2(r-i)}(\LM,\H)$, we define $\angl{u,v}\in \Q$
to be the coefficient of $\mu:=(-1)^rx_{\onehat}^r$ in the product 
$uv\in \Dp^{2r}(\LM,\H)\cong\Q$.
Proposition~\ref{prop:duality} states that this is a perfect pairing, and 
the monomials $\varepsilon(x_T^b)$ form a dual basis for 
$\Dp^{2i}(\LM,\H)^\vee$.  

Using the decomposition from Theorem~\ref{thm:D(X)}, we can describe
Poincar\'e duality on the level of monomials in each quotient algebra
$\Dp_y(\LM,\H)$, for each $y\in \L(\LM,\H)$ as well: let
\begin{equation}\label{eq:dualizing class}
\mu_y =\prod_{g\in F^+(y)\cap\H} (-c_g)^{d(z_y(g),g)-1}
 = \prod_{g\in F^+(y)\cap\H^\circ} (-c_g)^{d(z_y(g),g)-1}
\end{equation}
The two expressions are equivalent since $d(z_y(g),g)=1$ when $g\in\atoms(\LM)$.
We let $\angl{u,v}_y$ be 
the coefficient of $\mu_y$ in the product $uv$, for $u,v\in \Dp_y(\LM,\H)$.

For the rest of the section, we fix $\L=\L(\LM,\H)$ and $\Dp=\Dp(\LM,\H)$.
For a graded module $M$, let $M[r]^i=M^{i+r}$ for all integers $i,r$.
\begin{proposition}\label{prop:duality2}
For each $y\in \L$, the pairing $\angl{-,-}_y$ gives a graded
isomorphism of graded $\Dp$-modules $\Dp_y(\LM,\H)^\vee\cong
\Dp_y(\LM,\H)[2r-2\rank(y)]$.  
\end{proposition}
\begin{proof}
The Poincar\'e duality pairing of Proposition~\ref{prop:duality} gives
an additive isomorphism $\Dp^\vee\cong \Dp[2r]$.
Since $\angl{u,vw}=\angl{uv,w}$ for all
$u,v,w\in \Dp$, the map is an isomorphism of $\Dp$-modules.  The corresponding
claim for $\Dp_y$ follows from Theorem~\ref{thm:D(X)}.
\end{proof}

\begin{lemma}\label{lem:mu identity}  
  Suppose that $w=y\vee h$ for some $y,w\in \L$ and atom
$h\in\H\setminus\{\onehat\}$, and let
  $\rho\colon \Dp_y(\LM,\H)\to\Dp_w(\LM,\H)$ denote the restriction map.
  There is a $\Dp_y(\LM,\H)$-module homomorphism
  \[
  s\colon \Dp_w(\LM,\H)\to\Dp_y(\LM,\H)
  \]
  given by letting $s(u)=x_h \tilde{u}$, for any element $\tilde{u}$
  satisfying $\rho(\tilde{u})=u$, and an equality
\begin{equation}\label{eq:mu identity}
\angl{s(u),v}_y=\angl{u,\rho(v)}_w
\end{equation}
for all $u\in \Dp_w(\LM,\H)$ and $v\in \Dp_y(\LM,\H)$.
\end{lemma}
\begin{proof}
  Let $K=\ker(\rho)$, and $S$ a nested set for which $y=\bigvee S$.
  By comparing presentations (Definition~\ref{def:DPy}),
  we see that $K$ is generated by monomials $x_T$ for which $S\cup T$
  is a nested set, and for which $S\cup\set{h}\cup T$ is not.  For
  such monomials, $x_T\cdot x_h\in J_y(\H)$, which is to say that
  $x_h\in \ann(K)$.  It follows that, if $\rho(\tilde{u})=\rho(\tilde{u}')$,
  then $x_h\tilde{u}=x_h\tilde{u}'$.  That is, the map $s$ is well-defined
  (and clearly a $\Dp_y(\LM,\H)$-module homomorphism.)

  Using the bilinearity of the pairing, it is enough to check
  \eqref{eq:mu identity} for $u=1$ and all $v$ for which $\rho(v)=\mu_w$,
  which is to say that $x_hv=\mu_y$.  If $\rho(v')=\mu_w$ as well, then
  $v-v'\in K$, so $x_h(v-v')=0$.  It suffices, then, to show $x_hv=\mu_y$
  for a single choice of $v$.

Furthermore, in view of the tensor decomposition of $\Dp_y(\LM,\H)$ (Theorem
\ref{thm:D(X)}), it suffices to consider $y=\bottom$ and $w=h\in\H$.
Let $v=(-c_h)^{d(\bottom,h)-1}(-x_{\onehat})^{d(h,\onehat)-1}$.
Using the \eqref{def:B:iii'} relations 
$c_h^{d(\bottom,h)}=0$ and $x_px_{\onehat}^{d(p,\onehat)}=0$
for $p<\onehat$, we have
\begin{align*}
x_hv &= x_h (-c_h)^{d(\bottom,h)-1}(-x_{\onehat})^{d(h,\onehat)-1}\\
&= \left(-\sum_{p>h} x_p\right)
(-c_h)^{d(\bottom,h)-1}(-x_{\onehat})^{d(h,\onehat)-1}\\
&= (-x_{\onehat})^{1+(d(\bottom,h)-1)+(d(h,\onehat)-1)}\\
&= (-x_{\onehat})^{d(\bottom,\onehat)-1}\\
&=\mu.
\end{align*}
\end{proof}

In the next section, it will be more convenient to work with the $\Q$-dual
of $B$, rather than $B$.  We note that, for any graded $\Dp$-module $M$, 
$\Hom_{\Dp}(M,\Dp)=\Hom_{\Q}(M,\Q)[-2r]$.
So for each $j\geq0$ we let  
\begin{align}
  C^{\cdot j}(\LM,\H)&=\Hom_{\Q}(B(\LM,\H),\Q)[2j-2r]\label{eq:def C}\\
&=\Hom_{\Dp}(B(\LM,\H),\Dp)[2j]\nonumber
\end{align}
as a graded $\Dp$-module.  
Using the direct sum decomposition of Proposition~\ref{prop:decomp}, we may
write it as 
\begin{align}
C^{ij}(\LM,\H)&=\bigoplus_{\substack{y\in\L_j\colon\\  \nonumber
y\not\geq (\hat{1},\hat{0})}}
\OS^j(\L_{\leq y})^\vee\otimes_\Q (\Dp_y^\vee)^{i+2j-2r}\\ 
&=\bigoplus_{y}
\Fl^j(\L_{\leq y})\otimes_\Q \Dp^i_y. \label{eq:dual decomp}
\end{align}

Our pairing from above extends to one for all 
$0\leq j\leq r$ and $0\leq i\leq r-j$:
We will also denote it by $\angl{-,-}\colon B^{2i,j}\otimes 
C^{2(i+j),j}\to \Q$.  We will let 
$\dd^\vee\colon C^{ij}\to C^{i,j+1}$ denote the dual to the differential
$\dd$.  Then $\dd^\vee$ makes $C(\LM,\H)$ a complex of $\Dp$-modules.

Recall from \S\ref{ss:flag complex} that the flag complex $\Fl(\L)$ has a
differential $\delta$, defined in \eqref{eq:def delta}.
\begin{proposition}\label{prop:delta}
We have $\dd^\vee=\delta\otimes_{\Q}\Dp(\LM,\H)$.
\end{proposition}
\begin{proof}
We will simply write $\delta$ in place of $\delta\otimes_{\Q}\Dp(\LM,\H)$.
Suppose $S$ is an independent set of size $j\geq1$ in $\L(\LM,\H)$ with
$z=\bigvee S$, and $Y\in \Fl^{j-1}(\L)$ with top element $y<z$.  
It suffices to check that, for all $f,f'\in \Dp_y$ for which
$\deg(f)+\deg(f')=2r-2\rank(z)$, we have  
\[
\angl{\dd(e_S\otimes \rho(f)),Y\otimes f'}_y
=
\angl{e_S\otimes \rho(f),\delta(Y\otimes f')}_z.
\]

Both sides are zero unless
$Y=Y(g_1,g_2,\dots,g_{j-1})$, where $S=\set{g_1,g_2,\ldots,g_{j-1},g_j}$,
where $z=y \vee g$, and we abbreviate $g=g_j$.

Then we have
\begin{align*}
\angl {\dd(e_S\otimes \rho(f)),Y\otimes f'}_y
&=  \angl{\sum_{i=1}^j(-1)^{i-1}e_{S-\set{g_i}}\otimes x_{g_i}f,
Y\otimes f'}_y\\
&=(-1)^{j-1} \angl{e_{S-\set{g}}\otimes x_{g}f, Y\otimes f'}_y\\
&=(-1)^{j-1} \angl{e_{S-\set{g}}\otimes f, Y\otimes x_{g}f'}_y\\
&=(-1)^{j-1} \angl{e_{S-\set{g}}\otimes \rho(f), Y\otimes \rho(f')}_z
\text{~by Lemma~\ref{lem:mu identity};}\\
&=\angl{e_S\otimes\rho(f),\delta(Y\otimes f')}_z,
\end{align*}
as required.
\end{proof}




\subsection{Blowups and the \cdga}\label{ss:injective}
Our last objective is to show that combinatorial blowups induce
injective quasi-isomorphisms between the \cdga s we have constructed, hence
establishing a model for the Orlik--Solomon algebra.
It turns out to be relatively straightforward to verify that the map predicted
by topology is indeed well-defined and injective in our more general setting.
Our argument that the maps are quasi-isomorphisms requires some additional
ideas which we develop in \S\ref{ss:quasiiso}.  
We state our main result now, and the rest of this paper is devoted to
completing the proof. 

\begin{theorem}[Cohomology of the model]
\label{thm:model}
For each partial building set $\H$ containing $\onehat$, there are isomorphisms
$\OS(\LM) \simeq H^\cdot \hB(\LM,\H)$ and $\POS(\LM)\simeq H^\cdot B(\LM,\H)$
induced by
\begin{equation}\label{eq:OS_embeds}
e_g\mapsto\sum_{\substack{ h\in \H\colon\\ g\leq h}}e_h.
\end{equation}
\end{theorem}
\begin{proof}
In Lemma \ref{lem:phi}, we construct for any partial building sets $\H$ 
and $\H'=\H\cup\{p\}$, a \cdga\
map $\phi_{\hB}:\hB(\LM,\H)\to\hB(\LM,\H')$, motivated by topology and defined
as in \eqref{eq:OS_embeds}. 
In Theorem \ref{thm:injective}, we prove that $\phi_{\hB}$ is always injective.
When $\onehat\in\H$, we have $\hB(\LM,\H)\cong B(\LM,\H)\otimes\Q[e_{\onehat}]$,
and hence $\phi_{\hB}$ induces an injective \cdga\ map 
$\phi_B:B(\LM,\H)\to B(\LM,\H')$.

We will argue by induction, needing separate base cases for $B$ and $\hB$.
For the latter, Proposition
\ref{prop:OSisB} establishes an isomorphism between $\OS(\LM)$ and
$\hB(\LM,\atoms(\LM))$, and the differential is zero.  For the former,
let $\H = \atoms(\LM)\cup\set{\onehat}$, and we calculate directly.
Since each $\Dp_y(\LM,\H)\cong\Q[x_{\onehat}]/(x_{\onehat}^{r+1-i})$ for
$y\in \LM_i$, using \eqref{eq:oB-decomp}, we may identify
$
B^{2i,j}(\LM,\H)\cong \OS^j(\LM_{\leq r-i}),
$
where $\LM_{\leq r-i}$ denotes the truncation of $\LM$ to degrees $j\leq r-i$.
Under this identification,
\[
(B(\LM,\H),\dd)\cong\bigoplus_{i=0}^r\big(\OS(\LM_{\leq r-i}),\partial\big),
\]
so $H^0(B(\LM,\H))=\ker\partial=\POS(\LM)$, and higher cohomology vanishes
by \cite[Lem.\ 3.13]{OTbook} as in Proposition~\ref{prop:OS exact}.

Now we apply Theorems \ref{thm:injective} and
\ref{thm:quasiiso} to see that the composition of maps $\phi_B$ gives an
isomorphism $\POS(\LM)\to H^0(B(\LM,\H),\dd)$, for any partial building set
$\H$ containing $\onehat$, and $H^p(B(\LM,\H),\dd)=0$ for $p>0$.
By induction, we observe that this composition agrees with the formula
\eqref{eq:OS_embeds}.  The analogous result for $\hB$ follows similarly. 
\end{proof}

Before proceeding, we single out an interesting consequence of our theorem, as
well as a question for future work.
At the two extremes, we have $\L=\LM$ is a geometric lattice and 
$\L=\L(\LM,\G)$ is the face poset of the (classical) nested set complex,
respectively. When $\G$ is a full building set, the poset $\L(\LM,\G)$ is
simplicial, and $\OS(\L(\LM,\G))$ is the exterior Stanley-Reisner algebra of the
nested set complex (see Example \ref{ex:newOS2}). We immediately obtain the
following.
\begin{corollary}
The map \eqref{eq:OS_embeds} gives an inclusion
$\OS(\LM)\hookrightarrow \OS(\L(\LM,\G))$, where
$\OS(\L(\LM,\G))$ is the exterior face ring of the nested set
complex $\N(\LM,\G)$.
\end{corollary}

\begin{question}
If the matroid comes from a complex hyperplane arrangement, both
$\OS(\LM)$ and $\OS(\L(\LM,\G))$ are cohomology algebras of spaces determined
by poset combinatorics.  
The algebra $\OS(\L(\LM,\G))$ is the cohomology
algebra of a complex made up of unions of coordinate subtori in a torus,
indexed by $\N(\LM,\G)$, 
which is in some cases a classifying space for a right-angled Artin group:
see \cite{PS09} for details.

For other partial building sets $\H\subseteq \G$, is there a reasonable
space for which $\OS(\L(\LM,\H))$ is its cohomology algebra?
\end{question}

Now we turn back to the efforts of proving Theorem \ref{thm:model}.
In \S\ref{ss:blowup OS}, we found that there was an injective map
$\phi\colon \OS(\L(\LM,\H))\to\OS(\L(\LM,\H'))$.  
Here, we show that the map extends
to our \cdga.
\begin{lemma}
\label{lem:phi}
Suppose that $\H$ and $\H'=\H\cup\{p\}$ are partial building sets in a geometric
lattice $\LM$.
There is a \cdga\ map $\phi_{\hB}:\hB(\LM,\H)\to \hB(\LM,\H')$ defined as follows:
\[
\phi_{\hB}(e_g) = \begin{cases} e_g & \text{ if } g\not\leq p\\ e_g+e_p & \text{ if
} g\leq p \end{cases} 
\hspace{1cm}
\phi_{\hB}(x_g) = \begin{cases} x_g & \text{ if } g\not\leq p\\ x_g+x_p & \text{ if
} g\leq p \end{cases}
\]
\end{lemma}
\begin{proof}
We check that $\phi$ preserves the 
relations from Definition~\ref{def:B}.  The argument for relations
of type \eqref{def:B:i} and \eqref{def:B:ii} is the same as the one given
for the Orlik--Solomon algebra in Theorem~\ref{thm:OS_phi}.  Now for
$g\in\atoms(\LM)$, the only term $x_h$ in the definition of $c^{\H}_g$ 
\eqref{eq:cg}  that could possibly correspond to $h\leq p$ is $g$ itself. 
That implies  $c^{\H}_g$ picks up an $x_p$ exactly when $g\leq p$, hence
$\phi(c^{H}_g)=c^{\H'}_g$. 

It follows that $\phi$ is an algebra homomorphism.  To see $\phi$ is a
\cdga\ homomorphism, it is enough to verify that $\phi\circ \dd=\dd\circ \phi$
for the generators $e_g$ and $x_g$, where the claim is obvious.
\end{proof}

To establish the injectivity of $\phi$, we describe how it behaves on our
monomial basis from Corollary~\ref{cor:basis}.
For an expression $f\in \hB(\LM,\H)$, 
we may write its standard representative in the
monomial basis and let $\In(f)$ denote the largest monomial which appears.

\begin{lemma}\label{lem:lead of phi}
For a monomial $e_Sx_T^b$ in the monomial basis of $\hB(\LM,\H)$ from 
Corollary~\ref{cor:basis}, we have
\begin{equation}\label{eq:lead of phi}
\In(\phi(e_Sx_T^b))=\begin{cases}
e_Sx_T^b & \text{if $\bigvee S_{<p}\neq p$;}\\
e_{S-{g}}e_px_T^b & \text{if $\bigvee S_{<p}=p$,}
\end{cases}
\end{equation}
where $g=\min_{\prec} S_{<p}$, and, as usual, $S_{<p}:=\set{h\in S\colon h<p}$.
\end{lemma}
\begin{proof}
First recall from Lemma \ref{lem:OS lead of phi} that $\In(\phi(e_S))$ 
is either $e_S$ if $S\in \N(\LM,\H')$ (equivalently $\bigvee S_{<p}\neq p$) or 
$e_{S-g}e_p$ if $S\notin \N(\LM,\H')$ (equivalently $\bigvee S_{<p}=p$).

Also note that $\phi(x_T^b)$ is a sum of monomials where some
of the $x_h$'s with $h<p$ are replaced by $x_p$. Since $h<p$ implies that
$x_p\prec x_h$, we have $\In(\phi(x_T^b))=x_T^b$.

Next, we take the product of the standard representatives for $\phi(e_S)$ and
$\phi(x_T^b)$, and then rewrite its expansion in terms of the monomial basis
in order to get the standard representative of $\phi(e_Sx_T^b)$.
The largest monomial that could appear is $\In(\phi(e_S))\In(\phi(x_T^b))$, and
so it remains to check that this monomial is indeed in the basis. 
This check is similar to that in the proof of Lemma \ref{lem:OS lead of phi}.
\end{proof}

\begin{theorem}
\label{thm:injective}
Suppose that $\H$ and $\H'=\H\cup\{p\}$ are partial building sets in a geometric
lattice $\LM$.
There is an injective \cdga\ map $\phi_{\hB}:\hB(\LM,\H)\to\hB(\LM,\H')$.
Furthermore, if $\onehat\in\H$,
there is an injective \cdga\ map $\phi_B:B(\LM,\H)\to B(\LM,\H')$.
\end{theorem}
\begin{proof}
The initial monomials in Lemma \ref{lem:lead of phi} are distinct, by the same
argument in the proof of Lemma \ref{lem:OS injective}.
This implies that the map $\phi_{\hB}$ is an injective \cdga\ map, and it
induces the injective map on $B$.
\end{proof}

\section{The combinatorial Leray model and formality}
\label{sec:model}
The main objective of this section is to show that combinatorial blowups
induce quasi-isomorphisms of \cdga s, completing the proof of Theorem 
\ref{thm:model}.  For this, we make use of some 
sheaf cohomology on the semilattice $\L=\L(\LM,\H)$ and, more interestingly,
on its associated poset of intervals $\Int(\L)$, whose definition was
given in \S\ref{ss:I(L)}.
In fact, we find that the algebras $\Dp_y(\LM,\H)$ give the poset
$\L$ the structure of a ringed space, by equipping the poset with the order
topology.  Then (the dual of) $B(\LM,\H)$ is expressed as the 
global sections of a complex of locally free coherent sheaves on the 
poset $\Int(\L)$.


\subsection{A complex of sheaves}\label{ss:sheaves}

Inspired by Yuzvinsky's methods in \cite{Yuz95}, we will find that our
complex $(C(\LM,\H),d^\vee)$, defined in \eqref{eq:def C},
is a global version of a simpler, local construction.
To begin, we recall our notational conventions from
\S\ref{ss:basic sheaves}, and the poset of intervals together with coordinate
maps from \S\ref{ss:I(L)}:
\[
\begin{tikzcd}
  \L & \Int(\L)\ar[l,"\pr_1",swap]\ar[r,"\pr_2",bend left=10] &
  \L^{\op}.\ar[l,"\iota",bend left=10]
\end{tikzcd}
\]
It's easy to see that $\pr_2^*=\iota_*$; in particular, 
for $(x,y)\in \Int(\L)$ we have
\[
\iota_*\DP(x,y)=\Dp^\cdot_y(\H).
\]
Like $(\L^{\op},\DP)$, we see $(\Int(\L),\iota_*\DP)$ is also a ringed space.
Because $\L$ has a minimum element, sheaves coming from $\L^{\op}$ are acyclic:

\begin{lemma}\label{lem:D acyclic}
For any $\GS$ on $\L^{\op}$, we have $H^q(\L^{\op},\GS)=H^q(\Int(\L),\iota_*\GS)=0$
for all $q>0$.
\end{lemma}
\begin{proof}
Let $\GS$ be a sheaf on $\L^{\op}$, and consider the constant map 
$p\colon \L^{\op}\to \set{\bottom}$.  Since $\bottom$ is the (unique)
minimum element of $\L$, the constant map $p$ satisfies the hypothesis
of Lemma~\ref{lem:pointy fibres}, so $p_*$ is exact.  But $p_*=\Gamma$,
the global sections functor, so $H^i(\L^{\op},\GS)=0$ for all $i>0$.

The pullback $H^\cdot(\L^{\op},\GS)\to H^\cdot(\Int(\L),\pr_2^*\GS)$
is an isomorphism: to see this, we just note $\Int(\L^{\op})\cong\Int(\L)$
under the map $(x,y)\mapsto (y,x)$, and use Lemma~\ref{lem:pr1_iso}.
The argument is completed by noting $\pr_2^*=\iota_*$.  
\end{proof}

Now we are ready to introduce the local version of our Leray model,
using the decomposition \eqref{eq:dual decomp} as a guide.  For a partial
building set $\H$, let $\L=\L(\LM,\H)$, and let $\DP=\DP(\LM,\H)$.
\begin{definition}
Let $\CS(\LM,\H)=\pr_1^*\FS(\L)\otimes_\Q \pr_2^*\DP$, with differential
$\delta_\Q\otimes\pr_2^*\DP$.
That is, for all $(x,y)\in \Int(\L)$ and $i,j\geq0$, we have
\[
\CS^{ij}(\LM,\H)(x,y)=\Fl^j(\L_{\leq x})\otimes_{\Q}\Dp^i_y.
\]
As usual, we will write $\CS$ in place of $\CS(\LM,\H)$ when no ambiguity
arises.
\end{definition}
By construction, $(\CS,\delta)$ is a complex of locally free sheaves of 
$\iota_*\DP$-modules on $\Int(\L)$.  Our motivation is to study the double 
complex dual to $B(\LM,\H)$, which we recover here as global sections:
\begin{theorem}\label{thm:global sections}
We have an isomorphism of cochain complexes
$\Gamma(\CS^{\cdot,\cdot},\delta)=(C^{\cdot,\cdot},\delta)$.
\end{theorem}
\begin{proof}
  We will omit the cohomological indices for clarity, and let
  $[\L]$ denote the discrete topological space on $\L$.  The diagonal
  map $\Delta\colon  [\L]\to \Int(\L)$ is continuous, and we can regard
  the direct sum decomposition \eqref{eq:dual decomp} as (trivially) making
  $C^{\cdot,\cdot}$ a sheaf on $[\L]$.

  Since $\Delta_*C(x,y)=\CS(x,y)$ for $x=y$ and zero otherwise, the obvious
  map $\CS\to \Delta_*C$ has kernel $\ZS$, where $\ZS(x,y)=\CS(x,y)$ for
  $x<y$, and $\ZS(x,x)=0$.  Then
  \begin{align*}
    \Gamma(\ZS)&=\lim_{\substack{\longleftarrow \\
    _{(x,y)\in \Int(\L)}}} \ZS(x,y)\\
    &=0,
  \end{align*}
  since the initial objects in the diagram are all zero.  Applying global
  sections to the exact sequence $0\to\ZS\to\CS\to\Delta_*C$ then gives
  an injective map $\Gamma(\CS)\to\Gamma(\Delta_*C)=C$.

  In the other direction, for each $z \in \L$ we give a map $\psi_z\colon
  C(z)\to \CS(x,y)$.  Fix $i,j\geq0$.  

If $j\neq \rank(z)$, we let $\psi_z=0$.  Otherwise,
for $x_0\leq y_0$, then, let $\psi_z(x_0,y_0)
\colon C^{ij}\to \CS^{ij}(x_0,y_0)$
be given by $0$ unless $z\leq x_0$, and otherwise the tensor
product of the inclusion $\Fl^j(\L_{\leq z})\hookrightarrow 
\Fl^j(\L_{\leq x_0})$
with the surjection $\Dp^i_{z}\to \Dp^i_{y_0}$.  To check
that the maps $\psi_z$ are compatible with the restriction maps, consider
elements $x_0\leq x_1\leq y_1\leq y_0$, and consider the diagram
\[
\begin{tikzcd}
& \Fl^j(\L_{\leq z})\otimes \Dp^i_{z}
\ar[dl,"\psi_{z}{(x_1,y_1)}",swap]\ar[dr,"\psi_{z}{(x_0,y_0)}"] & \\
\CS^{ij}(x_1,y_1)\ar[rr] & & \CS^{ij}(x_0,y_0)
\end{tikzcd}
\]
where the horizontal map is the restriction map given by $(x_0,y_0)
\leq (x_1,y_1)$ in $\Int(\L)$.
Clearly the diagram commutes if 
$\psi_z(x_0,y_0)$ and $\psi_z(x_1,y_1)$ are both zero. 
The only remaining possibility is that $z\leq x_1$.
But then the composite 
\[
\Fl^j(\L_{\leq z})\hookrightarrow \Fl^j(\L_{\leq x_1})\twoheadrightarrow
\Fl^j(\L_{\leq x_0})
\]
is zero when $z\not\leq x_0$ and inclusion when $z\leq x_0$, using
Lemma~\ref{lem:flag brieskorn}, so again the diagram
is seen to commute.

Since 
\[
\Gamma(\CS)=\lim_{\substack{\longleftarrow \\ (x,y)\in \Int(\L)}} \CS(x,y),
\]
this induces a map $C=\Gamma\Delta_*C\to \Gamma(\CS)$.  The composite
$C\to \Gamma(\CS)\to \Gamma(\Delta_*C)=C$ is easily seen to be an 
isomorphism.  
We conclude that $\Gamma(\CS)\to 
\Gamma(\Delta_*C)=C$ is also surjective, hence an isomorphism.
\end{proof}
\begin{lemma}\label{lem:acyclic}
For all $i,j$, we have $H^q(\Int(\L),\CS^{ij})=0$ for all $q>0$.
\end{lemma}
\begin{proof}
In Lemma~\ref{lem:D acyclic} we saw $\pr_2^*\DP=\iota_*\DP(\L)$ 
is acyclic.  In Proposition~\ref{prop:flag sheaf complex}, we saw $\FS(\L)$ is 
flasque, hence acyclic.  By Lemma~\ref{lem:pr1_iso}, this implies
$\pr_1^*(\FS(\L))$ is also acyclic, and the result follows by the K\"unneth
formula.
\end{proof}

\begin{theorem}\label{thm:resolution}
The cochain complex
\[
\begin{tikzcd}
0\ar[r] & \iota_!\DP(\LM,\H)\ar[r] & \CS^{\cdot,0}\ar[r,"\delta"] & \cdots\ar[r,"\delta"] & \CS^{\cdot,i}
\ar[r,"\delta"] & \cdots\ar[r,"\delta"] & \CS^{\cdot,r}\ar[r] & 0
\end{tikzcd}
\]
is a $\Gamma$-acyclic resolution of $\iota_!\DP(\LM,\H)$.
\end{theorem}
\begin{proof}
  Lemma~\ref{lem:acyclic} showed that the sheaves in the complex are
  $\Gamma$-acyclic.  To see the complex is exact, we start with the
  exact complex $0\to\KS\to\FS^0\to\FS^1\cdots$ on $\L$ from 
Proposition~\ref{prop:flag sheaf complex}.
Applying $\pr_1^*$ preserves exactness.  
The sheaf $\pr_2^*\DP$ is free over $\Q$, hence flat.  We conclude
$\CS=\pr_1^*\FS\otimes\pr_2^*\DP$ has cohomology concentrated in degree 
zero, and
\begin{align*}
\ker(\delta^0) &= \KS\otimes \pr_2^*\DP\\
&=\KS\otimes \iota_*\DP\\
&=\iota_!\DP.
\end{align*}
\end{proof}

\subsection{Blowups induce quasi-isomorphisms}
\label{ss:quasiiso}
Now we combine the pieces above to prove the following theorem.
As usual, let $\L=\L(\LM,\H)$ and $\L'=\Bl_p(\L)$ be locally 
geometric semilattices as above, and $\pi\colon \L'\to \L$ the blow-down map.

\begin{theorem}\label{thm:quasiiso}
  For each partial building set $\H$ containing $\onehat$,
  there is a quasi-isomorphism $B(\LM,\H)\to B(\LM,\H')$.
\end{theorem}

The proof will make use of some preparatory results.  The first two lemmas are
of central importance.  Initial quotients were defined in
Definition~\ref{def:initial-quotient}.

\begin{lemma}\label{lem:pi connected}
  The map $\pi\colon \L'\to\L$ is an initial quotient with connected fibres.
\end{lemma}
\begin{proof}
By the construction of $\L'=\Bl_p(\L)$, we note that 
$\pi^{-1}(y)=\set{y}$ if $y\not\in \L_{\geq p}$.  On the other hand, 
if $y\geq p$, we have
\[
\pi^{-1}(y)=\set{(p,x)\colon p\vee x=y}.
\]
We claim that this set has a unique minimal element $(p,z_y)$.  If $y=p$,
clearly $z_y=\bottom$.  If $y>p$, by Proposition~\ref{prop:BICO}\eqref{prop:BICO1}, $y$ is
not irreducible, and $y=p\vee g_1\vee \cdots\vee g_k$ where
$F(\L,\G;y)=\{p,g_1,\dots,g_k\}$. Let $z_y=g_1\vee \cdots \vee g_k$, so that
$(p,z_y)\in \pi^{-1}(y)$.  Now consider any $(p,x)\in \pi^{-1}(y)$.
The join decomposition \eqref{eq:join decomp} gives
\[
[\bottom,y]\cong [\bottom,p]\times\prod_{i=1}^k[\bottom,g_i].
\]
Since $x\leq p\vee x=y$, the image of $x$ under this isomorphism has the
form $(q,g_1,\ldots,g_k)$ for some $q\in [\bottom,p]$.  It follows
$x\geq z_y$, and $(p,x)\geq (p,z_y)$, so $(p,z_y)$ is a cone vertex
in $\pi^{-1}(y)$, which proves the fibres are connected.

To show $\pi$ is an initial quotient map, we suppose
$\pi(x)\leq y$ and look for $x'\geq x$
for which $\pi(x')=y$.  Again if $y\not\geq p$, we may take $x'=y$.

Otherwise $y\geq p$, and we let $z_y$ be as above.
If $x\in\L_{\not\geq p}$, then take $x'=(p,(x\wedge p)\vee
z_y)\in\pi^{-1}(y)$.  In this case, the fact that $x=\pi(x)\leq y$ along with
the join decomposition above imply that $x=(x\wedge p)\vee(x\wedge z_y)\leq
(x\wedge p)\vee z_y$ and hence $x\leq x'$. Similarly, if $x=(p,w)$ for some
$w\in\L_{(p)}$, then take $x'=(p,(w\wedge p)\vee z_y)$.
\end{proof}

\begin{lemma}\label{lem:pi fibre}
  For any closed interval $[y_0, y_1]\subseteq \L$, the order complexes
  $\abs{\pi^{-1}(\L_{\geq y_0})}$ and $\abs{\pi^{-1}([y_0,y_1])}$
  are contractible.
\end{lemma}
\begin{proof}
  We fix a closed interval $[y_0,y_1]$.  Let $U=\pi^{-1}(\L_{\leq y_1})$,
  and consider the set
  $\pi^{-1}(\L_{\geq y_0})$.
  \begin{enumerate}
    \item[Case 1:]
  If $p\vee y_0$
does not exist, then $\L_{\geq y_0} \cap \L_{\geq p}=\emptyset$, and
once again, $\pi^{-1}(\L_{\geq y_0})\cong\L_{\geq y_0}$.  This order complex
has a cone point, $y_0$, and so does its intersection with the open set
$U$.  In this case, then, both order complexes are contractible.
\item[Case 2:]
Next suppose $p\leq y_0$. This time, there is a unique minimal element in
$\pi^{-1}(y_0)$, denoted $(p,z_{y_0})$ in the proof of
Lemma~\ref{lem:pi connected}.  It is a cone point in both order complexes.
\item[Case 3:]
Finally, if $p\vee y_0$ exists but $p\not\leq y_0$, 
it is easy to check that
\[
\pi^{-1}(\L_{\geq y_0})=\L'_{\geq y_0}
\cup \L'_{\geq (p,z_{y_0\vee p})}.
\]
Since $y_0\vee(p,z_{y_0\vee p})=(p,y_0)$, we have
$\L'_{\geq y_0}\cap \L'_{\geq(p,z_{y_0\vee p})}=\L'_{\geq(p,y_0)}$.

Thus 
$\pi^{-1}(\L_{\geq y_0})$ is a union of contractible simplicial complexes
which intersect along a contractible subcomplex, so again 
$\abs{\pi^{-1}(\L_{\geq y_0})}$ is contractible.  To conclude the
proof we look at the intersection of each of these three conical
sets with $U$.

Since $\pi(y_0)=y_0$ and $y_0\leq y_1$, we see
$\L'_{\geq y_0}\cap U$ is nonempty.  Now $\pi(p,z_{y_0\vee p})=y_0\vee p$.  So
if $y_0\vee p\leq y_1$, then all three cone points are in $U$, and
$\abs{\pi^{-1}([y_0,y_1])}$ is contractible by the same argument.  If
$y_0\vee p\not\leq y_1$, then $\L'_{(p,z_{y_0\vee p})}\cap U$ is empty,
and the preimage of $[y_0,y_1]$ equals $\L'_{\geq y_0}\cap U$, a cone.
\end{enumerate}
\end{proof}
\begin{remark}\label{rem:he semilattices}
  By Quillen's fibre lemma (see, e.g., \cite[p.\ 272]{Koz08}),
  the result above implies that the combinatorial blowdown $\pi\colon
  \L'_+\to\L_+$ induces a homotopy equivalence of order complexes.  If
  we compose these equivalences over the whole building set $\G$, we
  recover the matroidal case of the main result of \cite{FM05}.  When
  $\G=L_+$, this was also noted in \cite[Rem. 6.5]{AHK15}.
\end{remark}

By Proposition~\ref{prop:iota_natural}, we have a commuting square:
\[
\begin{tikzcd}
  (\L')^{\op}\ar[r,"\pi"]\ar[d,"\iota"] & \L^{\op}\ar[d,"\iota"] \\
  \Int(\L') \ar[r,"\Int(\pi)"] & \Int(\L)
\end{tikzcd}
\]

\begin{lemma}\label{lem:iota_vs_pi}
  For any sheaf $\GS$ on $\L^{\op}$, it is the case that 
  $\iota_!\pi^*\GS = \Int(\pi)^*\iota_!\GS$.  
\end{lemma}
\begin{proof}
  For any $(x,y)\in \Int(\L)$, we compute the stalk
  \begin{align*}
    (\Int(\pi)^*\iota_!\GS)(x,y)  &= (\iota_!\GS)(\pi(x),\pi(y))\\
    &=\begin{cases}
    \GS(\pi(y)) & \text{if $\pi(x)=\bottom$;}\\
    0 & \text{otherwise.}
    \end{cases}
  \end{align*}
  Noting that $\pi(x)=\bottom$ if and only if $x=\bottom$, we see that
  this is exactly the stalk $(\iota_!\pi^*\GS)(x,y)$.
\end{proof}

\begin{proposition}\label{prop:H_pullback}
  The pullback map $\Int(\pi)^*\colon H^\cdot(\Int(\L),\iota_!\DP)\to
  H^\cdot(\Int(\L'),\Int(\pi)^*\iota_!\DP)$ is an isomorphism.
\end{proposition}
\begin{proof}  
  Since $\pi$ is an initial
  quotient with connected fibres, by Lemma~\ref{lem:pi connected},
  we see $\Int(\pi)$ is a quotient, using
  Proposition~\ref{prop:initial-quotient}.  For any interval
  $y_0\leq y_1$ in $\L$, by inspection,
  \[
  \Int(\pi)^{-1}\big(\Int(\L)_{\geq (y_0,y_1)}\big) =
  \Int\big(\pi^{-1}([y_0,y_1]\big)\big). 
  \]
  The order complex of the latter is homeomorphic to
  $\abs{\pi^{-1}([y_0,y_1])}$ by \cite[Thm.\ 4.1]{Wal88}, which is
  contractible by Lemma~\ref{lem:pi fibre}.

  Now we check that each fibre $\Int(\pi)^{-1}([y_0,y_1])$ is connected.
  If neither $y_0\geq p$ nor $y_1\geq p$, then the fibre is just
  $\set{[y_0,y_1]}$.  If $p\leq y_0\leq y_1$, then using the join decomposition
  as in Lemma~\ref{lem:pi connected}, there are minimal elements $z_{y_0}\leq
  z_{y_1}$ for which $p\vee z_{y_i}=y_i$ for $i=0$ and $1$.  Then if
  $\Int(\pi)[(p,x),(p,y)]=[y_0,y_1]$ for some $x\leq y$ in $\L$, we check
  \[
    [x,y]  \geq [z_{y_0},y] \leq [z_{y_0},z_{y_1}],
  \]
  so every point in the fibre is connected by a path
  to $[(p,z_{y_0}),(p,z_{y_1})]$.

  Finally, suppose $y_1\geq p$ but $y_0\not\geq p$.  Then $y_0$ has a unique
  preimage.  This preimage can be written $y_0=u \vee v$,
  where $u\leq p$ and $v\leq z_{y_1}$, using the join decomposition of $y_1$.

  Then
  \[
  \Int(\pi)^{-1}([y_0,y_1]) = \set{[y_0,(p,x)]\colon
    p\vee x = y_1\text{~and~} y_0\leq x}.
  \]
  That set contains $[y_0,(p, u\vee z_{y_1})]$,
  since $p\vee u\vee z_{y_1} = y_1$,
  and $y_0=u\vee v \leq u \vee z_{y_1}$.  In fact, for any interval
  $[y_0,(p,x)]$ in the fibre, we have $x\geq z_{y_1}$
  (Lemma~\ref{lem:pi connected}) and $x\geq u$ because $x\geq y_0$, using
  the join decomposition again.  So the fibre contains a unique maximum,
  $[y_0,(p, u\vee z_{y_1})]$, and again we conclude that it is connected.
  
  The claim now follows by Proposition~\ref{prop:Quillen}.
\end{proof}

\begin{proof}[Proof of Theorem~\ref{thm:quasiiso}:]
Taking $\Q$-duals, instead we produce a quasi-isomorphism
$(C(\L'),\delta)\to (C(\L),\delta)$.
We will do so as follows, letting $\DP=\DP(\L)$ and $\DP'=\DP(\L')$ again.

We apply the exact functor $\iota_!$ to the sequence from
Theorem~\ref{thm:sheaf keel}, obtaining
a short exact sequence of sheaves on $\Int(\L')$:
\begin{equation}\label{eq:quasiiso}
\begin{tikzcd}
0\ar[r] & \iota_!\pi^*\DP\ar[r,"\iota_!\phi_{\Dp}"] & \iota_!\DP'\ar[r] & 
\iota_!\alpha_!(\QS\boxtimes \Q)\ar[r] & 0.
\end{tikzcd}
\end{equation}
Let us show the cohomology of
$\iota_!\alpha_!(\QS\boxtimes\Q)$ vanishes identically.  First, it is
easy to check that, since $\alpha$ is an open embedding (on $\L$), then
 $\Int(\alpha)$ is a closed embedding.  So we see
 $\iota_!\alpha_!=\Int(\alpha)_!\iota_!=\Int(\alpha)_*\iota_!$.
It is also easy to check that the functor $\Int(-)$ commutes with products.
Then
\begin{align*}
H^\cdot(\Int(\L'),\iota_!\alpha_!(\QS\boxtimes \Q)) &\cong
H^\cdot(\Int(\L'),\Int(\alpha)_*\iota_!(\QS\boxtimes\Q)) \\
&\cong H^\cdot(\Int(\L_{(p)})\times \Int(\set{0<1}),\iota_!\QS\boxtimes
\iota_!\Q).
\end{align*}
A routine calculation shows that 
$H^\cdot(\Int(\set{0<1}),\iota_!\Q)=0$ (e.g., as in \cite{Ba75}).
It follows by the K\"unneth formula that 
the cohomology of $\iota_!\alpha_!(\QS\boxtimes\Q)$ vanishes identically.
Applying the cohomology long exact sequence to \eqref{eq:quasiiso}, we
see $H^\cdot(\iota_!\phi_{\Dp})$ is an isomorphism.
\begin{figure}  
  \begin{tikzpicture}
    \node (TL) at (-1,0) {$(0,0):\, \Q$};
    \node (TR) at (1,0) {$(1,1):\, 0$};
    \node (B) at (0,-1) {$(0,1):\, \Q$};
    \draw[-stealth] (TL) -- (B);
    \draw[-stealth] (TR) -- (B);
  \end{tikzpicture}
  \caption{$\iota_!\Q$ on $\Int(\set{0<1})$}
\end{figure}

By Theorem~\ref{thm:resolution}, the cohomology of the cochain
complex $(C(\L'),\delta)$ is isomorphic to
\begin{align*}
 H^\cdot(\Int(\L'),\iota_!\DP') &\cong H^\cdot(\Int(\L'),\iota_!\pi^*\DP)\text{~by the above;}\\
&=H^\cdot(\Int(\L'),\Int(\pi)^*\iota_!\DP) \text{~by Lemma~\ref{lem:iota_vs_pi};}\\
 &\cong H^\cdot(\Int(\L),\iota_!\DP)\text{~by Proposition~\ref{prop:H_pullback},}\\
 &= H^\cdot(C(\L),\delta) \text{~by Theorem~\ref{thm:resolution} again.}
\end{align*}
\end{proof}
\begin{remark}
  We are unable to show directly in the argument above that the
  quasi-isomorphism above is given by
  the dual of the \cdga\ map $\phi_B\colon B(\L)\to B(\L')$.  Nevertheless,
  we can conclude this after the fact using Theorem~\ref{thm:model},
  since we showed
  there that $\phi_B$ induces a cohomology isomorphism in degree zero, and
  (using the result above) that the higher cohomology vanishes.
\end{remark}

\begin{ack}
  The second author would like to thank Udit Mavinkurve for pointing out an
  error in an earlier version of Proposition~\ref{prop:Quillen}.
The authors thank an anonymous referee for a careful reading and many valuable
comments.
\end{ack}

\bibliographystyle{amsalpha}
\bibliography{leray}

\newcommand{\etalchar}[1]{$^{#1}$}
\def\cprime{$'$}
\providecommand{\bysame}{\leavevmode\hbox to3em{\hrulefill}\thinspace}
\providecommand{\MR}{\relax\ifhmode\unskip\space\fi MR }
\providecommand{\MRhref}[2]{%
  \href{http://www.ams.org/mathscinet-getitem?mr=#1}{#2}
}
\providecommand{\href}[2]{#2}
\begin{thebibliography}{BHM{\etalchar{+}}20}

\bibitem[ADH20]{ADH20}
Federico Ardila, Graham Denham, and June Huh, \emph{Lagrangian geometry of
  matroids}, 2020.

\bibitem[AHK18]{AHK15}
K.~{Adiprasito}, J.~{Huh}, and E.~{Katz}, \emph{Hodge theory for combinatorial
  geometries}, Annals of Mathematics \textbf{188} (2018), 381--452.

\bibitem[Bac75]{Ba75}
Kenneth Bac\l{}awski, \emph{Whitney numbers of geometric lattices}, Advances in
  Math. \textbf{16} (1975), 125--138.

\bibitem[BHM{\etalchar{+}}20]{BHMPW20}
Tom Braden, June Huh, Jacob~P. Matherne, Nicholas Proudfoot, and Botong Wang,
  \emph{A semi-small decomposition of the chow ring of a matroid}, 2020.

\bibitem[Bib16]{Bi16}
Christin Bibby, \emph{Cohomology of abelian arrangements}, Proc. Amer. Math.
  Soc. \textbf{144} (2016), no.~7, 3093--3104.

\bibitem[Bj{\"o}82]{Bj82}
Anders Bj{\"o}rner, \emph{On the homology of geometric lattices}, Algebra
  Universalis \textbf{14} (1982), no.~1, 107--128.

\bibitem[Bri73]{Br73}
Egbert Brieskorn, \emph{Sur les groupes de tresses [d'apr\`es {V}. {I}.
  {A}rnol'd]}, S\'eminaire Bourbaki, 24\`eme ann\'ee (1971/1972), Exp. No. 401,
  Springer-Verlag, Berlin, 1973, pp.~21--44. Lecture Notes in Math., Vol. 317.

\bibitem[Cra67]{Cr67}
Henry~H. Crapo, \emph{A higher invariant for matroids}, J. Combinatorial Theory
  \textbf{2} (1967), 406--417. \MR{215744}

\bibitem[DCP95]{DP95}
Corrado De~Concini and Claudio Procesi, \emph{Wonderful models of subspace
  arrangements}, Selecta Math. (N.S.) \textbf{1} (1995), no.~3, 459--494.

\bibitem[Deh62]{deH62}
Ren{\'e} Deheuvels, \emph{Homologie des ensembles ordonn\'es et des espaces
  topologiques}, Bull. Soc. Math. France \textbf{90} (1962), 261--321.

\bibitem[Dup15]{du15}
Cl{\'e}ment Dupont, \emph{The {O}rlik-{S}olomon model for hypersurface
  arrangements}, Ann. Inst. Fourier (Grenoble) \textbf{65} (2015), no.~6,
  2507--2545.

\bibitem[ET19]{ET19}
Brent Everitt and Paul Turner, \emph{Deletion-restriction for sheaf homology of
  graded atomic lattices}, 2019, arXiv:1902.00399v2.

\bibitem[Fei05]{Fe05}
Eva~Maria Feichtner, \emph{De {C}oncini-{P}rocesi wonderful arrangement models:
  a discrete geometer's point of view}, Combinatorial and computational
  geometry, Math. Sci. Res. Inst. Publ., vol.~52, Cambridge Univ. Press,
  Cambridge, 2005, pp.~333--360.

\bibitem[FK04]{FK04}
Eva~Maria Feichtner and Dmitry~N. Kozlov, \emph{Incidence combinatorics of
  resolutions}, Selecta Math. (N.S.) \textbf{10} (2004), no.~1, 37--60.

\bibitem[FM05]{FM05}
Eva~Maria Feichtner and Irene M\"uller, \emph{On the topology of nested set
  complexes}, Proc. Amer. Math. Soc. \textbf{133} (2005), no.~4, 999--1006.

\bibitem[FY04]{FY04}
Eva~Maria Feichtner and Sergey Yuzvinsky, \emph{Chow rings of toric varieties
  defined by atomic lattices}, Invent. Math. \textbf{155} (2004), no.~3,
  515--536.

\bibitem[Kaw04]{Ka04}
Yukihito Kawahara, \emph{On matroids and {O}rlik-{S}olomon algebras}, Ann.
  Comb. \textbf{8} (2004), no.~1, 63--80.

\bibitem[Koz08]{Koz08}
Dmitry Kozlov, \emph{Combinatorial algebraic topology}, Algorithms and
  Computation in Mathematics, vol.~21, Springer, Berlin, 2008.

\bibitem[Loo93]{Loo93}
Eduard Looijenga, \emph{Cohomology of {${\mathcal{M}}_3$} and
  {${\mathcal{M}}^1_3$}}, Mapping class groups and moduli spaces of {R}iemann
  surfaces ({G}\"ottingen, 1991/{S}eattle, {WA}, 1991), Contemp. Math., vol.
  150, Amer. Math. Soc., Providence, RI, 1993, pp.~205--228.

\bibitem[OS80]{OS80}
Peter Orlik and Louis Solomon, \emph{Combinatorics and topology of complements
  of hyperplanes}, Invent. Math. \textbf{56} (1980), no.~2, 167--189.

\bibitem[OT92]{OTbook}
Peter Orlik and Hiroaki Terao, \emph{Arrangements of hyperplanes}, Grundlehren
  der Mathematischen Wissenschaften, vol. 300, Springer-Verlag, Berlin, 1992.

\bibitem[Oxl11]{Oxbook}
James Oxley, \emph{Matroid theory}, second ed., Oxford Graduate Texts in
  Mathematics, vol.~21, Oxford University Press, Oxford, 2011.

\bibitem[PS09]{PS09}
Stefan Papadima and Alexander~I. Suciu, \emph{Toric complexes and {A}rtin
  kernels}, Adv. Math. \textbf{220} (2009), no.~2, 441--477.

\bibitem[Sha93]{Sh93}
B.~Z. Shapiro, \emph{The mixed {H}odge structure of the complement to an
  arbitrary arrangement of affine complex hyperplanes is pure}, Proc. Amer.
  Math. Soc. \textbf{117} (1993), no.~4, 931--933.

\bibitem[SV91]{SV91}
Vadim~V. Schechtman and Alexander~N. Varchenko, \emph{Arrangements of
  hyperplanes and {L}ie algebra homology}, Invent. Math. \textbf{106} (1991),
  no.~1, 139--194.

\bibitem[Wal88]{Wal88}
James~W. Walker, \emph{Canonical homeomorphisms of posets}, European J. Combin.
  \textbf{9} (1988), no.~2, 97--107.

\bibitem[WW86]{WW86}
Michelle~L. Wachs and James~W. Walker, \emph{On geometric semilattices}, Order
  \textbf{2} (1986), no.~4, 367--385.

\bibitem[Yuz95]{Yuz95}
Sergey Yuzvinsky, \emph{Cohomology of the {B}rieskorn-{O}rlik-{S}olomon
  algebras}, Comm. Algebra \textbf{23} (1995), no.~14, 5339--5354.

\bibitem[Yuz97]{Yuz97}
\bysame, \emph{Cohomology bases for the {D}e {C}oncini-{P}rocesi models of
  hyperplane arrangements and sums over trees}, Invent. Math. \textbf{127}
  (1997), no.~2, 319--335.

\bibitem[Yuz01]{Yu01}
S.~Yuzvinsky, \emph{Orlik-{S}olomon algebras in algebra and topology}, Russian
  Math. Surveys \textbf{56} (2001), no.~2, 293--364.

\end{thebibliography}
\end{document}